\documentclass[10pt,a4paper,reqno]{amsart}

\usepackage[colorlinks=true, linkcolor=magenta, citecolor=cyan, urlcolor=blue,hyperfootnotes=false]{hyperref}
\usepackage{amssymb,graphicx,mathrsfs,xcolor,enumitem,eucal}
\numberwithin{equation}{section}\newtheorem{theorem}{Theorem}[section]
\newtheorem{corollary}[theorem]{Corollary}\newtheorem{lemma}[theorem]{Lemma}

\newtheorem{proposition}[theorem]{Proposition}\theoremstyle{remark}
\newtheorem{remark}{Remark}[section]
\theoremstyle{definition}

\newcommand{\bra}[1]{\langle #1 \rangle} 
\newcommand{\brasigma}[1]{\langle #1 \rangle_\sigma}     
\newcommand{\one}[1]{\mathbf{1}_{#1}}

\title[Scattering for NLS]
{Scattering in the energy space
for the NLS with variable coefficients}
\date{\today}    
\author{Biagio Cassano}
\address{Biagio Cassano: 
SAPIENZA --- Universit\`a di Roma,
Dipartimento di Matematica, 
Piazzale A.~Moro 2, I-00185 Roma, Italy}
\email{cassano@mat.uniroma1.it}
\author{Piero D'Ancona}
\address{Piero D'Ancona: 
SAPIENZA --- Universit\`a di Roma,
Dipartimento di Matematica, 
Piazzale A.~Moro 2, I-00185 Roma, Italy}
\email{dancona@mat.uniroma1.it}

\keywords{}

\newcommand{\R}{{\mathbb R}}

\newcommand{\bu}{{\bar u}}

\usepackage{mathtools}
\DeclarePairedDelimiter{\abs}{\lvert}{\rvert}
\DeclarePairedDelimiter{\norma}{\lVert}{\rVert} 
\usepackage{color}

\newcommand{\ignora}[1]{}
{\left\lbrace\begin{array}{@{}l@{}}}%
{\end{array}\right.}  

\begin{document}
\begin{abstract}
  We consider the NLS with variable coefficients
  in dimension $n\ge3$
  \begin{equation*}
    i \partial_t u - Lu +f(u)=0,
    \qquad
    Lv=\nabla^{b}\cdot(a(x)\nabla^{b}v)-c(x)v,
    \qquad
    \nabla^{b}=\nabla+ib(x),
  \end{equation*}
  on $\mathbb{R}^{n}$ or more generally on an
  exterior domain with Dirichlet boundary conditions,
  for a gauge invariant, defocusing nonlinearity of power type
  $f(u)\simeq|u|^{\gamma-1}u$.
  We assume that $L$ is a small, long range perturbation
  of $\Delta$, plus a potential with a large positive part. 
  The first main result
  of the paper is a bilinear smoothing (interaction
  Morawetz) estimate for the solution.

  As an application, under the conditional assumption
  that Strichartz estimates are valid for the linear flow
  $e^{itL}$, we prove global well posedness
  in the energy space for subcritical powers 
  $\gamma<1+\frac{4}{n-2}$, and scattering provided
  $\gamma>1+\frac4n$.
  When the domain is $\mathbb{R}^{n}$,
  by extending the Strichartz estimates due to Tataru
  \cite{Tataru08-a}, we prove that the conditional assumption
  is satisfied and deduce well posedness
  and scattering in the energy space.
\end{abstract}
\maketitle


\section{Introduction}\label{sec:introduction}

We study the Cauchy problem in the energy space
for the semilinear Schr\"{o}dinger equation
\begin{equation}\label{eq:main}
  i \partial_t u - Lu +f(u)=0,
  \qquad
  u(0,x)=u_{0}(x)
\end{equation}
on an exterior domain $\Omega=\mathbb{R}^{n}\setminus \omega$
with $C^{1}$ boundary, in dimension $n\ge 3$, where
$\omega$ is compact and possibly empty. Here
$L$ is a second order elliptic operator defined on $\Omega$
with Dirichlet boundary conditions, of the form
\begin{equation}\label{eq:opL}
  Lv=\nabla^{b}\cdot(a(x)\nabla^{b}v)-c(x)v,
  \qquad
  \nabla^{b}=\nabla+ib(x),
\end{equation}
where $a(x)=[a_{jk}(x)]_{j,k=1}^{n}$, 
$b(x)=(b_{1}(x),\dots,b_{n}(x))$
and $c(x)$ satisfy
\begin{equation}\label{eq:assopL}
  a,b,c\ \ \text{are real valued,}\ \ 
  a_{jk}=a_{kj}\ \ \text{and}\ \ 
  NI\ge a(x)\ge\nu I
  \ \ \text{for some}\ \ N\ge\nu>0.
\end{equation}
The low dimensional cases $n\le 2$ require substantial
modifications of our techniques and will be the object
of future work.

Our main results can be summarized as follows.
Assume that 
\begin{enumerate}
[label=(\roman*)]
  \item the principal part of $L$
  is a small, long range perturbation of $\Delta$;
  \item $b,c$ have an almost critical decay, 
  with $b$ and $c_{-}:=\max\{0,-c\}$ small;
  \item the boundary $\partial\Omega$ is starshaped 
  with respect to the metric induced by $a(x)$;
  \item the nonlinearity $f(u)\simeq|u|^{\gamma-1}u$ is of power 
  type, gauge invariant, defocusing, with $\gamma$ in the 
  subcritical range $1\le\gamma<1+\frac{4}{n-2}$.
\end{enumerate}
Then we prove:
\begin{enumerate}
  \item a virial identity for \eqref{eq:main}, from which
  we deduce a smoothing and a bilinear smoothing
  (interaction Morawetz) estimate for solutions of
  \eqref{eq:main}.
  \item global well posedness and scattering
  in the energy space for the Cauchy problem \eqref{eq:main},
  under the black box assumption that Strichartz estimates
  are valid for the linear flow $e^{itL}$;
  scattering requires $\gamma>1+\frac4n$.
  \item in the case $\Omega=\mathbb{R}^{n}$, 
  we extend the Strichartz estimates
  proved by Tataru \cite{Tataru08-a} to the case of
  large electric potentials; hence we can drop the black
  box assumption and we obtain well posedness
  and scattering in the energy space for \eqref{eq:main}.
\end{enumerate}

Note that for exterior domains, 
Strichartz estimates are known but only locally in time,
see e.g.~\cite{BlairSmithSogge12-a},
\cite{BaskinMarzuolaWunsch12-a} and
the references therein.
However, research on this
topic is advancing rapidly, thus in the general case
$\Omega \not = \mathbb{R}^{n}$ we decided to assume 
\emph{a priori} the validity of Strichartz estimates.
In the case $\Omega=\mathbb{R}^{n}$
sufficiently strong results are already available
and we use them to close the proof of scattering.
On a related note we
mention the global smoothing estimates on the 
exterior of polygonal domains proved
in \cite{BaskinMarzuolaWunsch12-a}.

The theory of Strichartz estimates on
$\mathbb{R}^{n}$ is extensive and many results are known.
We mention in particular
\cite{Yajima95-a},
\cite{Yajima99-a},
\cite{Yajima99-a},
\cite{RodnianskiSchlag04-a}
\cite{DAnconaFanelli06-a}
for the case of electric potentials,
\cite{DAnconaFanelli08-a}
and \cite{ErdoganGoldbergSchlag09-a}
for magnetic potentials,
and, for operators with fully variable coefficients,
\cite{StaffilaniTataru02-a},
\cite{RobbianoZuily05-a}
and \cite{Tataru08-a} 
(see also the refences therein).
Note that large perturbations in the second order
terms require suitable nontrapping assumptions,
which are implicit here in the assumption that
$|a(x)-I|$ is sufficiently small.

Scattering theory is a important subject and the number of
references is huge. For a comprehensive review of the
classical theory and an extensive bibliography we refer to
\cite{Cazenave03-a}
(see also
\cite{GinibreVelo85-a}).
Smoothing estimates are also a classical subject,
originated in 
\cite{Kato65-b} and
\cite{Morawetz68-a},
\cite{Morawetz75-a}.
The bilinear version of smoothing estimates, also
called \emph{interaction Morawetz} estimates,
was introduced as a tool in scattering theory in
\cite{CollianderKeelStaffilani04-a}, 
\cite{TaoVisanZhang07-a}
and recently adapted to
Schr\"{o}dinger equations with an
electromagnetic potential in
\cite{CollianderCzubakLee14-a}.
We mention that here we follow the simpler approach 
developed in
\cite{Visciglia09-a}, \cite{CassanoTarulli14-a}.

We conclude the introduction with a detailed exposition
of our results. Here and in the rest of the paper we make 
frequent use of the basic properties of Lorentz spaces 
$L^{p,q}$, in particular precised H\"{o}lder, Young and Sobolev 
inequalities, for which we refer to \cite{ONeil63-a}.

In the following we denote by
$|a(x)|$ the operator norm of the matrix $a(x)$,
and we use the notations
\begin{equation*}
  \textstyle
  |a'|=\sum_{|\alpha|=1}|\partial^{\alpha}a(x)|,
  \qquad
  |a''|=\sum_{|\alpha|=2}|\partial^{\alpha}a(x)|,
  \qquad
  |a'''|=\sum_{|\alpha|=3}|\partial^{\alpha}a(x)|,
\end{equation*}
\begin{equation*}
  \textstyle
  |b'|=\sum_{j,k}|\partial_{x_{j}}b_{k}|,
  \qquad
  |c'|=\sum_{j}|\partial_{x_{j}}c|.
\end{equation*}

\subsection{The operator \texorpdfstring{$L$}{}
and its heat kernel \texorpdfstring{$e^{tL}$}{}}
\label{sub:selfadjoint}
The results of this section are valid for 
all dimensions $n\ge3$.
Very mild conditions on the coefficients of $L$
are sufficient for selfadjointness:
in Proposition \ref{pro:selfadjoint} we prove 
by standard arguments that if
\begin{equation}\label{eq:assselfadjL}
  b\in L^{n,\infty},
  \qquad
  c\in L^{\frac n2,\infty},
  \qquad
  \|c_{-}\|_{L^{\frac n2,\infty}}<\epsilon,
\end{equation}
with $\epsilon$ small enough (and $a(x)\in L^{\infty}$),
then the operator $L$ defined on $C^{\infty}_{c}(\Omega)$
extends in the sense of forms
to a selfadjoint, nonpositive operator
with domain $H^{1}_{0}(\Omega)\cap H^{2}(\Omega)$.
Throughout the paper, this operator will be referred to
as the operator $L$ with Dirichlet boundary conditions;
note that in all our results the assumptions
are stronger than \eqref{eq:assselfadjL}.

Under the additional assumption
\begin{equation*}
  b^{2}+|\nabla \cdot b|\in L^{2}_{loc},
  \qquad
  c\in L^{\frac n2,1},
  \qquad
  \|c_{-}\|_{L^{\frac n2,1}}<\epsilon
\end{equation*}
with $\epsilon$ small enough, we prove in
Proposition \ref{pro:heatk} that the heat kernel of $L$
satisfies a gaussian upper estimate of the form
\begin{equation*}
  |e^{tL}(x,y)|
  \le C' t^{-\frac n2}e^{-\frac{|x-y|^{2}}{Ct}},
  \qquad
  t>0.
\end{equation*}
In Proposition \ref{pro:powereq}, assuming further that
\begin{equation*}
  \|a-I\|_{L^{\infty}}+
  \||b|+|a'|\|_{L^{n,\infty}}+
  \|b'\|_{L^{\frac n2,\infty}}
  <
  \epsilon
\end{equation*}
for $\epsilon$ small enough,
using the previous bound we deduce the equivalence
\begin{equation}\label{eq:equivfrac}
  \|(-L)^{\sigma}v\|_{L^{p}}
  \simeq
  \|(-\Delta)^{\sigma}v\|_{L^{p}},
  \qquad
  1<p<\frac{n}{2 \sigma},
  \qquad
  0\le \sigma\le1.
\end{equation}

\subsection{Morawetz and interaction Morawetz estimates}
\label{sub:morawetz}

From now on we restrict to the case 
when the operator $L$ is a suitable long range 
perturbation of $\Delta$ on $\Omega$; 
the precise conditions are the following.

Let $n \geq 3$ and 
 assume that for some $0<\delta\le1$
\begin{equation}\label{eq:assader}
  |a'(x)|+|x||a''(x)|+|x|^{2}|a'''(x)|\le C_a 
  \langle x\rangle^{-1-\delta},
\end{equation}
where $\bra{x}:=(1+\abs{x}^2)^{1/2}$.
Moreover, $b$ and the matrix
$db(x):=
  [\partial_{j}b_{\ell}-\partial_{\ell}b_{j}]_{j,\ell=1}^{n}$
satisfy
\begin{equation}\label{eq:assdb}
  \textstyle
  b\in L^{n,\infty},
  \qquad
  |db(x)|\le \frac{C_{b}}{|x|^{2+\delta}+|x|^{2-\delta}}.
\end{equation}
The potential $c(x)$ satisfies
\begin{equation}\label{eq:assc}
  \textstyle
  -\frac{C_{-}^{2}}{|x|^{2}}
  \le
  c(x)
  \le
  \frac{C_{+}^{2}}{|x|^{2}}
\end{equation}
(which implies $c\in L^{\frac n2,\infty}$)
and is \emph{repulsive} with respect
to the metric $a(x)$, meaning that
\begin{equation}\label{eq:assac}
  \textstyle
  a(x)x \cdot\nabla c(x)\le 
  \frac{C_{c}}{|x|\langle x\rangle^{1+\delta}}.
\end{equation}
The nonlinearity  
$ f\colon \mathbb{C} \to \mathbb{C} $ is such that $f(0)=0$ and,
for some $1 \leq \gamma < 1+\frac{4}{n-2}$,
\begin{equation}\label{eq:growthfu}
  \abs{f(z)-f(w)}\leq (|z|+|w|)^{\gamma-1}|z-w|, 
  \quad \text{ for all }z,w\in \mathbb{C}.
\end{equation}
Note that it is easy to adapt our proofs to handle
nonlinearities satisfying the more general assumption
\begin{equation*}
  \abs{f(z)-f(w)}\leq (1+|z|^{\gamma-1}+|w|^{\gamma-1})|z-w|.
\end{equation*}
We also assume that $f$ is  \emph{gauge invariant}, that is to say
\begin{equation}\label{eq:gauinv}
  f(\R)\subseteq \R
  \quad\text{and}
  \quad f(e^{i\theta}z)=e^{i\theta}f(z)
  \quad \text{ for all }\theta \in \R,\,z\in \mathbb{C}.
\end{equation}
Moreover, writing
\begin{equation}\label{eq:defnF}
  \textstyle
  F(z):=\int_0^{\abs{z}}f(s)\,ds,
\end{equation}
we assume that $f$ is \emph{repulsive}, i.e.,
\begin{equation}
  \label{eq:repulsivity}
  f(z)\bar z - 2 F(z)\geq 0 \ \text{ for all }z \in \mathbb{C}.
\end{equation}
Finally, concerning the domain
$\Omega$, we assume that $\partial \Omega$ is $C^{1}$
and $a(x)$--\emph{starshaped},
meaning that at all points $x\in\partial \Omega$
the exterior normal $\vec\nu$ to $\partial\Omega$ satisfies
\begin{equation}\label{eq:assbdry}
  a(x)x\cdot\vec{\nu}(x)\le0.
\end{equation}
In the following statement we use the
Morrey-Campanato type norms defined by
\begin{equation*}
  \textstyle
  \|v\|_{\dot X}^{2}:=
  \sup\limits_{R>0}\frac{1}{R^{2}}\int_{\Omega\cap\{|x|=R\}}
  |v|^{2}dS,
  \qquad
  \|v\|_{\dot Y}^{2}:=
  \sup\limits_{R>0}\frac{1}{R}\int_{\Omega\cap\{|x|\le R\}}
  |v|^{2}dx.
\end{equation*}
Moreover we use the notation $L^{2}_{T}=L^{2}(0,T)$
to denote integration in $t$ on the interval $[0,T]$,
while $L^{p}_{T}L^{q}=L^{p}(0,T;L^{q}(\Omega))$
and $L^{p}L^{q}=L^{p}(\mathbb{R};L^{q}(\Omega))$.

\begin{theorem}[Smoothing]\label{the:1}
  Let $n\ge4 $, $L$ the operator in \eqref{eq:opL}, 
  \eqref{eq:assopL}
  with Dirichlet b.c.~on the exterior domain $\Omega$,
  and assume
  \eqref{eq:assader},
  \eqref{eq:assdb},
  \eqref{eq:assac} and
  \eqref{eq:assbdry}.
  Let $u \in C(\mathbb{R}, H^1_{0}(\Omega))$
  be a solution of Problem \eqref{eq:main}.
  Then, if $N/\nu-1$ and the
  constants $C_a,C_{b},C_{-},C_{c}$ are sufficiently small,
  $u$ satisfies for all $T>0$ the estimate
  \begin{equation}\label{eq:smoomor}
     \norma{u}_{\dot X_{x}L^{2}_{T}}^2 
         + \norma{\nabla^{b}u}_{\dot Y_{x}L^{2}_{T}}^2
     \textstyle
    + \int_0^T\int_{\Omega}
         \frac{f(u)\overline{ u}-2F(u)}{\abs{x}}\,dxdt 
     \lesssim
     \|u(0)\|_{\dot H^{\frac{1}{2}}}^2
        + \|u(T)\|_{\dot H^{\frac{1}{2}}}^2
  \end{equation}
   with an implicit constant independent of $T$.
\end{theorem}

Theorem \ref{the:1} actually holds even in the case
$n=3$, but we need a condition on $a(x)$ which 
essentially forces it to be diagonal, and this is
of course too restrictive for our purposes
(see \eqref{eq:assratio} below). 
Thus in the 3D case we modify our approach and
prove an estimate in terms of 
\emph{nonhomogeneous} Morrey-Campanato norms
\begin{equation*}
  \textstyle
  \|v\|_{X}^{2}:=
  \sup\limits_{R>0}\frac{1}{\langle R\rangle^{2}}
  \int_{\Omega\cap\{|x|=R\}}
  |v|^{2}dS,
  \qquad
  \|v\|_{Y}^{2}:=
  \sup\limits_{R>1}\frac{1}{R}\int_{\Omega\cap\{|x|\le R\}}
  |v|^{2}dx.
\end{equation*}
We also need some slightly stronger assumptions 
on the coefficients: we require
\begin{equation}\label{eq:assperturb}
  |a(x)- I|\le  C_{I} \langle x\rangle^{-\delta},
  \qquad
  C_{I}<1,
\end{equation}
moreover we assume
\begin{equation}\label{eq:assdbstrong}
  \textstyle
  b\in L^{3,\infty},
  \qquad
  |db(x)|\le \frac{C_{b}}{|x|^{2+\delta}+|x|}.
\end{equation}
Then we have:

\begin{theorem}[Smoothing, $n=3$]\label{the:1strong}
  Let $L$ the operator in \eqref{eq:opL}, 
  \eqref{eq:assopL}
  with Dirichlet b.c.~on the exterior domain $\Omega$,
  and assume
  \eqref{eq:assader}, \eqref{eq:assperturb}
  \eqref{eq:assdbstrong},
  \eqref{eq:assc}, \eqref{eq:assac}, 
  \eqref{eq:gauinv}, \eqref{eq:repulsivity}, and
  \eqref{eq:assbdry}.
Let $u \in C(\mathbb{R}, H^1_{0}(\Omega))$
be a solution of Problem \eqref{eq:main}.
Then, if $N/\nu-1$ and the
constants $C_a,C_I,C_{b},C_{-},C_{c}$ are sufficiently small,
the solution $u$ satisfies for all $T>0$ the estimate
\begin{equation}\label{eq:smoomorstrong}
   \norma{u}_{ X_{x}L^{2}_{T}}^2 
       + \norma{\nabla^{b}u}_{ Y_{x}L^{2}_{T}}^2
   \textstyle
  + \int_0^T\int_{\Omega}
       \frac{f(u)\overline{ u}-2F(u)}{\bra{x}}\,dxdt 
   \lesssim
   \|u(0)\|_{\dot H^{\frac{1}{2}}}^2
      + \|u(T)\|_{\dot H^{\frac{1}{2}}}^2
\end{equation}
 with an implicit constant independent of $T$.
\end{theorem}

The previous results are  \emph{a priori} estimates on
a global solution $u$, for which conservation of energy might
not hold; this is why we state 
estimates \eqref{eq:smoomor},\eqref{eq:smoomorstrong}  on a
finite time interval $[0,T]$ and we need the
norm of $u$ both at $t=0$ and at $t=T$ at the right hand side.
Note that it is possible to give explicit bounds on
the smallness assumption on the coefficients, see
Remark \ref{rem:explicsmall1}.

\begin{remark}[]\label{rem:cacciafDanlu}
  The proofs of Theorems \ref{the:1} and \ref{the:1strong} have
  a substantial overlap with the proof in
  \cite{CacciafestaDAnconaLuca14-a} 
  of resolvent estimates for the Helmholtz equation
  \begin{equation*}
    Lu+zu=f,
    \qquad
    z\in \mathbb{C}\setminus \mathbb{R}.
  \end{equation*}
  One can indeed deduce estimates for the \emph{linear}
  Schr\"{o}dinger equation from the corresponding
  estimates for Helmholtz, via Kato's theory of
  smoothing \cite{Kato65-a},
  but with a loss in the sharpness of the estimates
  (see Corollary 1.3 in \cite{CacciafestaDAnconaLuca14-a}
  for details; see also \cite{Cacciafesta13-a} for
  earlier results in a simpler setting).
\end{remark}

\begin{remark}\label{rem:sharpere}
  Note that in \eqref{eq:smoomor} and \eqref{eq:smoomorstrong}
  the space-time norms are reversed in $(x,t)$, 
  due to the method of proof. In the hypoteses of 
  Theorem \ref{the:1}, thanks to \eqref{eq:smoomor} 
  and \eqref{eq:stnorma1}, \eqref{eq:stnorma4},
  and in the hypoteses of Theorem \ref{the:1strong}, thanks to
  \eqref{eq:smoomorstrong} and 
  \eqref{eq:stnorma4}, \eqref{eq:stnorma12},
  we deduce the standard weighted $L^{2}$ estimate
  \begin{equation}\label{eq:smoomorwei}
    \norma{\bra{x}^{-3/2-}u}_{L^2_T L^2_x}^2 
        + \norma{\bra{x}^{-1/2-}\nabla^{b}u}_{L^2_T L^2_x}^2
    \textstyle
   + \int_0^T\int_{\Omega}
        \frac{f(u)\overline{ u}-2F(u)}{\abs{x}}\,dxdt 
    \lesssim
    \|u(0)\|_{\dot H^{\frac{1}{2}}}^2
            + \|u(T)\|_{\dot H^{\frac{1}{2}}}^2.
  \end{equation}
  By \eqref{eq:equivmagw} we can replace $\nabla^{b}$
  with $\nabla$ at the left hand side, obtaining
  \begin{equation}\label{eq:smoomorwei2}
    \norma{\bra{x}^{-3/2-}u}_{L^2_T L^2_x}^2 
        + \norma{\bra{x}^{-1/2-}\nabla u}_{L^2_T L^2_x}^2
    \textstyle
    \lesssim
    \|u(0)\|_{\dot H^{\frac{1}{2}}}^2
            + \|u(T)\|_{\dot H^{\frac{1}{2}}}^2.
  \end{equation}
  If the assumptions on $b,c$ are slightly stronger so that
  the heat kernel $e^{tL}$ satisfies an upper gaussian bound,
  we can apply the techniques in 
  \cite{CacciafestaDAncona12-a}
  to obtain a further estimate of weighted $L^{2}$ tipe.
  In the next Corollary we assume
  $\Omega=\mathbb{R}^{n}$ to keep the proof simple but this
  would not be necessary.
\end{remark}

\begin{corollary}[]\label{cor:smooheat}
  Let $n\ge3$, $\Omega=\mathbb{R}^{n}$, let $L$ be as in
  Theorem \ref{the:1} or as in Theorem \ref{the:1strong}, 
  and assume that
  \begin{equation*}
    b^{2}+|\nabla \cdot b|\in L^{2}_{loc},
    \qquad
    c\in L^{\frac n2,1},
    \qquad
    \|c_{-}\|_{L^{\frac n2,1}}<\epsilon.
  \end{equation*}
  Then for $\epsilon$ small enough the
  flow $e^{itL}$ satisfies the estimate
  \begin{equation}\label{eq:smoomorzero}
    \|\bra{x}^{-1-}e^{itL}u_{0}\|_{L^{2}_{t}L^{2}_{x}}
    \lesssim
    \|u_{0}\|_{L^{2}}.
  \end{equation}
\end{corollary}

The next results are  bilinear smoothing (interaction Morawetz)
estimates for equation \eqref{eq:main}, which are the crucial
tool in the proof of scattering. Note that the assumptions
are essentially the same as in Theorems \ref{the:1}, 
\ref{the:1strong},
and the constant $C_{c'}$ may be large.

\begin{theorem}[Bilinear smoothing, $n\ge 4$]\label{the:2}
  Let $n\ge4$ and let $\Omega,L$ be as in
  Theorem \ref{the:1}. In addition, assume that
  \begin{equation}\label{eq:assacnabla}
    \textstyle
      |x|^{2}\vert\nabla c \vert \le
      C_{c'}\bra{x}^{-1-\delta}.
  \end{equation}
  Let $u \in C(\mathbb{R}, H^1_{0}(\Omega))$ be
  a solution of \eqref{eq:main}.
  Then, if the constants $C_a,C_{b},C_{-},C_{c}$ and
  $N/\nu-1$ are small enough,
  $u$ satisfies the estimate
   \begin{equation}\label{eq:interactiontesi}
     \int_0^T
     \int_{\Omega \times \Omega}
     \frac{|u(t,x)|^{2}|u(t,y)|^{2}}{|x-y|^{3}}
     dxdydt
     \lesssim
     \|u(0)\|_{L^2}^{2}
     \left[\norma{u(0)}_{\dot H^{\frac{1}{2}}}
     +\norma{u(T)}_{\dot H^{\frac{1}{2}}}\right]^{2}.
   \end{equation}
\end{theorem}
\begin{theorem}[Bilinear smoothing, $n=3$]\label{the:2strong}
  Let $n=3$ and let $\Omega,L$ be as in
  Theorem \ref{the:1strong}. In addition, assume \eqref{eq:assacnabla}.
  Let $u \in C(\mathbb{R}, H^1_{0}(\Omega))$ be
  a solution of \eqref{eq:main}.
  Then, if the constants $C_a,C_I,C_{b},C_{-},C_{c}$ and
  $N/\nu-1$ are small enough,
  $u$ satisfies the estimate
   \begin{equation}\label{eq:interactiontesistrong}
     \norma{u}_{L^4(0,T;L^4(\Omega))}^4
     \lesssim
     \|u(0)\|_{L^2}^{2}
     \left[\norma{u(0)}_{\dot H^{\frac{1}{2}}}
     +\norma{u(T)}_{\dot H^{\frac{1}{2}}}\right]^{2}.
   \end{equation}
\end{theorem}


\subsection{Global existence and scattering}\label{sub:scattering}

The proof of well posedness and scattering for \eqref{eq:main}
in the energy space relies in an essential way
on Strichartz estimates for the linear flow $e^{itL}$. 
As mentioned above,
these are known in the case $\Omega=\mathbb{R}^{n}$
under various assumptions on the coefficients,
while the results for exterior domains are far from complete.
For this reason we decided to state our main results by 
assuming the validity of Strichartz estimates in a black 
box form, and then specialize them to some
situations where Strichartz
estimates are already available. Recalling that an
\emph{admissible (non endpoint) couple} is a couple
of indices $(p,q)$ with $2<p\le \infty$ and $2/p+n/q=n/2$,
our black box assumption has the following form:

\textsc{Assumption (S)}. 
\textit{
The Schr\"{o}dinger flow
$e^{itL}$ satisfies the Strichartz estimates
\begin{equation}\label{eq:bbstr}
  \|e^{itL}u_{0}\|_{L^{p_{1}}L^{q_{1}}}
  \lesssim
  \|u_{0}\|_{L^{2}},
  \qquad
  \textstyle
  \|\int_{0}^{t}e^{i(t-s)L}Fds\|_{L^{p_{1}}L^{q_{1}}}
  \lesssim
  \|F\|_{L^{p_{2}'}L^{q_{2}'}}
\end{equation}
for all admissible couples $(p_{j},q_{j})$,
while the derivative of the flow $\nabla e^{itL}$ satisfies
\begin{equation}\label{eq:1bbstr}
  \|\nabla e^{itL}u_{0}\|_{L^{p_{1}}L^{q_{1}}}
  \lesssim
  \|\nabla u_{0}\|_{L^{2}},
  \qquad
  \textstyle
  \|\nabla\int_{0}^{t}e^{i(t-s)L}Fds\|_{L^{p_{1}}L^{q_{1}}}
  \lesssim
  \|\nabla F\|_{L^{p_{2}'}L^{q_{2}'}}
\end{equation}
for admissible couples $(p_{j},q_{j})$ such that $q_{1}<n$.
}

Note that it is not trivial to deduce \eqref{eq:1bbstr}
from \eqref{eq:bbstr}: indeed,
for this step one needs the equivalence of norms
\begin{equation*}
  \|(-L)^{\frac12}v\|_{L^{q}}\simeq
  \|\nabla v\|_{L^{q}}
\end{equation*}
with $q$ in the appropriate range. 
Under fairly general assumptions on $L$,
we are able to prove this equivalence for all
$1<q<n$ (see \eqref{eq:equivfrac}),
and this is the reason for the restriction on $q_{1}$ in (S).

Using Assumption (S) we can
prove local well posedness in the energy space,
and global well posedness provided the nonlinearity is
\emph{defocusing}, i.e.,
\begin{equation}\label{eq:defocus}
  \textstyle
  F(r)=\int_{0}^{r}f(s)ds\ge0
  \ \ \text{for}\ \ s\in \mathbb{R}
\end{equation}
(this is the content of 
Proposition \ref{pro:localex} and Theorem \ref{the:global}):

\begin{theorem}[]\label{the:localglobal}
  Let $n\ge3$, let $\Omega=\mathbb{R}^{n}\setminus \omega$ 
  be an exterior domain with compact and possibly
  empty $C^{1}$ boundary, let $L$ be the selfadjoint
  operator with Dirichlet b.c.~defined by \eqref{eq:opL},
  \eqref{eq:assopL}, \eqref{eq:assselfadjL},
  and assume (S) holds.
  \begin{enumerate}
  [label=\textit{(\roman*)}]
    \item (Local existence in $H^{1}$).
    If $f\in C^{1}(\mathbb{C},\mathbb{C})$ satisfies
    $f(0)=0$ and
    $|f(z)-f(w)| \lesssim(|z|+|w|)^{\gamma-1}|z-w|$
    for some $1\le \gamma<1+\frac{4}{n-2}$,
    then for all $u_{0}\in H^{1}_{0}(\Omega)$ there exists
    $T=T(\|u_{0}\|_{H^{1}})$ and a unique solution
    $u\in C([-T,T];H^{1}_{0}(\Omega))$. 
    \item (Global existence in $H^{1}$).
    Assume in addition that
    $b^{2}+|\nabla \cdot b|\in L^{2}_{loc}$,
    $c\in L^{\frac n2,1}$,
    \begin{equation*}
      \|a-I \|_{L^{\infty}}+
      \||b|+|a'|\|_{L^{n,\infty}}+
      \|b'\|_{L^{\frac n2,\infty}}
      +\|c_{-}\|_{L^{\frac n2,1}}
      <\epsilon
    \end{equation*}
    for $\epsilon$ small enough, and that
    $f(u)$ is gauge invariant \eqref{eq:gauinv} 
    and defocusing \eqref{eq:defocus}.
    Then for all initial data
    $u_{0}\in H^{1}_{0}(\Omega)$ problem 
    \eqref{eq:maineq} has a unique global solution
    $u\in C\cap L^{\infty}(\mathbb{R};H^{1}_{0}(\Omega))$.
    The solution
    has constant energy for all $t\in \mathbb{R}$:
    \begin{equation*}
      \textstyle
      E(t)=\frac12\int_{\Omega}
      a(x)\nabla^{b}u \cdot \overline{\nabla^{b}u}dx
      +\frac12\int_{\Omega}c(x)|u|^{2}dx
      +\int_{\Omega}F(u)dx
      \equiv
      E(0).
    \end{equation*}
  \end{enumerate}
\end{theorem}

Combining the global existence result with
the bilinear smoothing estimate in
Theorems \ref{the:2} and \ref{the:2strong}, we obtain the main results of
this paper. Note that a power nonlinearity
$f(u)=|u|^{\gamma-1}u$ 
with $1+\frac4n< \gamma<1+\frac{4}{n-2}$
satisfies all conditions of the following Theorems:
\begin{theorem}[Scattering on $\Omega$, under (S)]
  \label{the:3}
  Let $n\ge 3$, $\Omega=\mathbb{R}^{n}\setminus \omega$
  an exterior domain with $C^{1}$ compact and
  possibly empty boundary satisfying \eqref{eq:assbdry},
  $L$ the operator \eqref{eq:opL}
  with Dirichlet b.c.~on $\Omega$. 
  Assume $a,b,c$ satisfy,
  for some $\epsilon,C>0$, $\delta\in(0,1]$
  \begin{equation*}
    |x|a(x)x \cdot \nabla c
    <\epsilon \bra{x}^{-1-\delta},
    \qquad
    |x||c|+|x|^{2}|c'|<C \bra{x}^{-1-\delta},
  \end{equation*}
  and in addition
      \begin{gather*}
        \textstyle
        \|a-I\|_{L^{\infty}}+ 
        |x|^2 c_{-}+ \|c_{-}\|_{L^{\frac n2,1}}<\epsilon,
        \quad
        |x||b|+|x|^{2}|b'|< \epsilon|x|^{\delta}\bra{x}^{-2\delta},
        \qquad
        \text{ if $n \ge 4$};
        \\
        \textstyle
        \bra{x}^{\delta} \|a-I\|_{L^{\infty}}+\bra{x}^2 c_{-}+\|c_{-}\|_{L^{\frac n2,1}}<\epsilon,    
        \quad
        |x||b|+ |x|\bra{x}^{1+\delta}|b'|< \epsilon,
        \qquad
        \text{ if $n = 3$}.
      \end{gather*}
  Finally     $|a'|+|x||a''|+|x|^{2}|a'''|
    <\epsilon \bra{x}^{-1-\delta}$, and $f:\mathbb{C}\to \mathbb{C}$ 
  is gauge invariant \eqref{eq:gauinv},
  repulsive \eqref{eq:repulsivity},
  defocusing \eqref{eq:defocus}
  and satisfies
  $f(0)=0$, 
  $|f(z)-f(w)| \lesssim(|z|+|w|)^{\gamma-1}|z-w|$
  for some $1+\frac4n< \gamma<1+\frac{4}{n-2}$.
  Then if (S) holds and
  $\epsilon$ is small enough we have:
  \begin{enumerate}
  [noitemsep,topsep=0pt,parsep=0pt,partopsep=0pt,
  label=\textit{(\roman*)}]
    \item (Existence of wave operators) For every 
    $u_{+}\in H^{1}_{0}(\Omega)$ there exists a unique
    $u_{0}\in H^{1}_{0}(\Omega)$ such that the global
    solution $u(t)$ to \eqref{eq:main} satisfies
    $\|e^{-itL}u_{+}-u(t)\|_{H^{1}}\to0$ as $t\to+\infty$.
    An analogous result holds for $t\to-\infty$.
    \item (Asymptotic completeness) For every
    $u_{0}\in H^{1}_{0}(\Omega)$ there exists a unique
    $u_{+}\in H^{1}_{0}(\Omega)$ such that the global
    solution $u(t)$ to \eqref{eq:main} satisfies
    $\|e^{-itL}u_{+}-u(t)\|_{H^{1}}\to0$ as $t\to+\infty$.
    An analogous result holds for $t\to-\infty$.
  \end{enumerate}
\end{theorem}

When $\Omega=\mathbb{R}^{n}$, Strichartz estimates
for $e^{itL}$ were proved by Tataru \cite{Tataru08-a}
in the case $L$ is a small, long range perturbations of
$\Delta$. In Theorems \ref{the:strich1} - \ref{the:strich2}
we adapt the result in \cite{Tataru08-a}
to our situation, and in particular,
combining it with the smoothing estimate 
\eqref{eq:smoomor}, we extend Strichartz estimates
to potentials $c(x)$ with a large positive part.
In addition we deduce the necessary estimates also for
the derivative of the flow $\nabla e^{itL}$
(Corollary \ref{cor:strichder}). As a consequence,
Assumption (S) is satisfied and we obtain the final
result of the paper:

\begin{theorem}[Scattering on $\mathbb{R}^{n}$]\label{the:4}
  Let $n\ge3$, assume $a,b,c$ satisfy $c\in L^{n}_{loc}$ and
  \begin{gather*}
    |a-I|+
    \bra{x}(|a'|+|b|)+
    \bra{x}^{2}(|a''|+|b'|)+
    \bra{x}^{3}|a'''|
    <\epsilon \bra{x}^{-\delta}, 
    \\
    |x| \bra{x} a(x) x \cdot \nabla c < \epsilon\bra{x}^{-\delta},
    \quad 
    \|c_{-}\|_{L^{\frac n2,1}}<\epsilon,
    \quad
    |x||c| + |x|^2 |c'|< C \bra{x}^{-1-\delta}.
    \\
|x|^{2}c_{-}<\epsilon, \quad   \text{if $n\ge 4$}, \qquad \bra{x}^{2} c_{-}<\epsilon, \quad \text{ if $n=3$},
  \end{gather*}
  for some $C>0$, $\delta\in(0,1]$ and some $\epsilon$
  small enough,
  and let $L$ be the selfadjoint operator defined by
  \eqref{eq:opL}-\eqref{eq:assopL} on $\mathbb{R}^{n}$.
  Finally, assume $f:\mathbb{C}\to \mathbb{C}$
  is gauge invariant \eqref{eq:gauinv},
  repulsive \eqref{eq:repulsivity},
  defocusing \eqref{eq:defocus}
  and satisfies
  $f(0)=0$, 
  $|f(z)-f(w)| \lesssim(|z|+|w|)^{\gamma-1}|z-w|$
  for some $1+\frac4n< \gamma<1+\frac{4}{n-2}$.

  Then the conclusions (i), (ii) of Theorem \ref{the:3} are valid.
\end{theorem}

\section{Notations and elementary identities}\label{sec:notations}

Using the convention of implicit summation over repeated indices,
we define the operators
\begin{equation}\label{eq:notAAb}
  A^{b}v:=\nabla^b\cdot(a(x)\nabla^bv)=
  \partial_j^b(a_{jk}(x)\partial_k^bv),
  \qquad
 Av:=\nabla\cdot(a(x)\nabla v)=
  \partial_j(a_{jk}(x)\partial_kv)
\end{equation}
so that $L=A^{b}-c$.
The quadratic form associated with $A$ is given by
\begin{equation*}
  a(w,z):=a_{jk}(x)w_k\overline{z}_j.
\end{equation*}
We shall use the notations
\begin{equation*}
  \textstyle
  \widehat{x}=\frac{x}{|x|}=
  (\widehat{x}_{1},\dots,\widehat{x}_{n}),\qquad
  \widehat{x}_{j}=\frac{x_{j}}{|x|},
\end{equation*}
\begin{equation*}
  \widehat{a}(x)=a_{\ell m}(x)\widehat{x}_{\ell}\widehat{x}_{m},
  \qquad
  \overline{a}(x)=\mathop{\textrm{trace}}a(x)=a_{mm}(x).
\end{equation*}
Since $a(x)$ is positive definite, we have
\begin{equation*}
  0\le \widehat{a}= a \widehat{x}\cdot \widehat{x}
  \le |a \widehat{x}|\le \overline{a}.
\end{equation*}
Indices after a semicolon refer to partial derivatives:
\begin{equation*}
  a_{jk;\ell}:=\partial_{\ell}a_{jk},\qquad
  a_{jk;\ell m}:=\partial_{\ell}\partial_{m}a_{jk},\qquad
  a_{jk;\ell mp}:=\partial_\ell\partial_ m\partial_p a_{jk}.
\end{equation*}
Notice the formulas
\begin{equation*}
  \partial_{k}(\widehat{x}_{\ell})=
  |x|^{-1}[\delta_{k\ell}-\widehat{x}_{k}\widehat{x}_{\ell}],
\end{equation*}
\begin{equation*}
  \partial_{k}(\widehat{x}_{\ell}\widehat{x}_{m})=
  |x|^{-1}[\delta_{k\ell}\widehat{x}_{m}+\delta_{km}
     \widehat{x}_{\ell}
  -2 \widehat{x}_{k} \widehat{x}_{\ell}\widehat{x}_{m}],
\end{equation*}
\begin{equation*}
\begin{split}
  \textstyle
  \partial_{j}\partial_{k}(\widehat{x}_{\ell}\widehat{x}_{m})=
  &
  \frac{1}{|x|^{2}}
  [
  \delta_{k\ell}\delta_{jm}+\delta_{km}\delta_{j\ell}
  +8 \widehat{x}_{j}\widehat{x}_{k}\widehat{x}_{\ell}
     \widehat{x}_{m}
    \\
  &
  -2 \delta_{k\ell}\widehat{x}_{j}\widehat{x}_{m}
  -2 \delta_{km}\widehat{x}_{j}\widehat{x}_{\ell}
  -2\delta_{jk}\widehat{x}_{\ell}\widehat{x}_{m}
  -2\delta_{j\ell}\widehat{x}_{k}\widehat{x}_{m}
  -2\delta_{jm}\widehat{x}_{k}\widehat{x}_{\ell}
  ]
\end{split}
\end{equation*}
which imply
\begin{equation*}
  a_{jk}a_{\ell m}\widehat{x}_{j}
  \partial_{k}(\widehat{x}_{\ell}\widehat{x}_{m})=
  2|x|^{-1}[|a \widehat{x}|^{2}-\widehat{a}^{2}],
\end{equation*}
and
\begin{equation*}
  \textstyle
  a_{jk}a_{\ell m}
  \partial_{j}\partial_{k}(\widehat{x}_{\ell}\widehat{x}_{m})=
  \frac{2}{|x|^{2}}
  [
  a_{\ell m}a_{\ell m}
  -4(|a \widehat{x}|^{2}-\widehat{a}^{2})
  -\overline{a}\widehat{a}
  ].
\end{equation*}
Using the previous identities, we see that for any
radial function $\psi(x)=\psi(|x|)$ we can write
\begin{equation}\label{eq:apsigen}
  A \psi(x)=\partial_{\ell}(a_{\ell m}\widehat{x}_{m}\psi')
  =
  \widehat{a}\psi''+
  \frac{\overline{a}-\widehat{a}}{|x|}
  \psi'
  +
  a_{\ell m;\ell}\widehat{x}_{m}\psi'
\end{equation}
where $\psi'$ denotes the derivative of $\psi(r)$ with respect
to the radial variable.


We now give the definitions 
of the Morrey-Campanato type norms $\dot X,\dot Y,X,Y$
and recall some relations
between them and usual weighted $L^{2}$ norms.

For an open subset $\Omega \subseteq\mathbb{R}^{n}$, $n\ge2$,
we use the notations
\begin{equation*}
    \Omega_{=R}=\Omega\cap\{x \colon |x|=R\},\quad
    \Omega_{\le R}=\Omega\cap\{x \colon |x|\le R\},\quad
    \Omega_{\ge R}=\Omega\cap\{x \colon |x|\ge R\},
\end{equation*}
\begin{equation*}
    \Omega_{R_{1}\le|x|\le R_{2}}
    =\Omega\cap\{x \colon R_{1}\le|x|\le R_{2}\}.
\end{equation*}
The homogeneous and  inhomogeneous norms $\dot X$ 
and  $X$ of a function 
$v:\Omega\to \mathbb{C}$
are defined as
\begin{equation*}
  \textstyle
  \|v\|_{\dot X}^{2}:=
  \sup\limits_{R>0}\frac{1}{R^{2}}\int_{\Omega_{=R}}
  |v|^{2}dS,
  \quad
  \|v\|_{X}^{2}:=
  \sup\limits_{R>0}\frac{1}{\langle R\rangle^{2}}
  \int_{\Omega_{=R}}
  |v|^{2}dS,
  \end{equation*}
where $dS$ is the surface measure on $\Omega_{=R}$
and $\langle R\rangle=\sqrt{1+R^{2}}$.
We shall also need proper Morrey-Campanato spaces, both in
the homogeneous version $\dot Y$ and in the non homogenous
version $Y$; their norms are defined as
\begin{equation}\label{eq:MC3}
  \textstyle
  \|v\|_{\dot Y}^{2}:=
  \sup\limits_{R>0}\frac{1}{R}\int_{\Omega_{\le R}}
  |v|^{2}dx,
  \qquad
  \|v\|_{Y}^{2}:=
  \sup\limits_{R>0}
  \frac{1}{\langle R\rangle}\int_{\Omega_{\le R}}
  |v|^{2}dx.
\end{equation}
The following equivalence is easy to prove:
\begin{equation}\label{eq:equivnonhom}
  \textstyle
  \|v\|_{Y}^{2}\le
  \sup_{R\ge1}\frac1R\int_{\Omega_{\le R}}|v|^{2}\le
  \sqrt{2}\|v\|_{Y}^{2}.
\end{equation}

The following Lemmas collect a few estimates to be used
in the rest of the paper, which follow immediately from
the definitions (proofs are straightforward, and 
full details can be found in
\cite{CacciafestaDAnconaLuca14-a}).

\begin{lemma}\label{lem:normest1}
  For any $v\in C^{\infty}(\mathbb{R}^{n})$,
  \begin{equation}\label{eq:stnorma8}
    \||x|^{-1}v\|_{\dot Y}\le
    \|v\|_{\dot X},
    \qquad
    \|\bra{x}^{-1}v\|_{ Y}\le
    \|v\|_{ X},
  \end{equation}
  \begin{equation}\label{eq:stnorma2}
    \textstyle
    \sup\limits_{R>0}
    \int_{\Omega_{\ge R}}
    \frac{R^{n-1}}{|x|^{n+2}}|v|^{2}dx
    \le
    \frac{1}{n-1} \|v\|_{\dot X}^{2},
    \qquad
    \sup\limits_{R>1}
    \int_{\Omega_{\ge R}}
    \frac{R^{n-1}}{|x|^{n+2}}|v|^{2}dx
    \le
    \frac{2}{n-1} \|v\|_{X}^{2}.
  \end{equation}
\end{lemma}

\begin{lemma}\label{lem:normest2}
  For any $0<\delta<1$ and $v\in C^{\infty}(\mathbb{R}^{n})$,
  \begin{equation}\label{eq:stnorma1}
    \textstyle
    \int_{\Omega}
    \frac{|v|^{2}}{|x|^{2}\langle x\rangle^{1+\delta}}
    \le  2  \delta^{-1}\|v\|_{\dot X}^{2},
  \end{equation}
  \begin{equation}\label{eq:stnorma3}
    \textstyle
    \int_{\Omega_{\ge1}}
    \frac{|v|^{2}}{|x|^{3}\langle x\rangle^{\delta}}
    \le
    \int_{\Omega_{\ge1}}
    \frac{|v|^{2}}{|x|^{3+\delta}}
    \le  
    2 \delta^{-1}\|v\|_{X}^{2},
  \end{equation}
  \begin{equation}\label{eq:stnorma4}
    \textstyle
    \int_{\Omega}\frac{|v|^{2}}{\langle x\rangle^{1+\delta}}
    \le
    8 \delta^{-1}
    \|v\|_{Y}^{2}
    \le
    8 \delta^{-1}
    \|v\|_{\dot Y}^{2}.
  \end{equation}
\end{lemma}

\begin{lemma}\label{lem:normest4}
  For any $R>0$, $0<\delta<1$ and 
  $v,w\in C^{\infty}(\mathbb{R}^{n})$,
  \begin{equation}\label{eq:stnorma10}
    \textstyle
    \int_{\Omega_{\le1}}\frac{|vw|}{|x|^{2-\delta}}
    +
    \int_{\Omega_{\ge1}}\frac{|vw|}{|x|^{2+\delta}}
    \le
    9\delta^{-1}
    \|v\|_{\dot X}\|w\|_{\dot Y}.
  \end{equation}
\end{lemma}

In the following Lemma we prove some magnetic
Hardy type inequalities, which require $n\ge3$,
expressed in terms of the nonhomogeneous $X,Y$ norms
(compare \eqref{eq:hardymag1} with Theorem A.1 in 
\cite{FanelliVega09-a}):
\begin{lemma}\label{lem:normest5}
  Let $n\ge3$ and
  assume $b(x)=(b_{1}(x),\dots,b_{n}(x))$ is continuous
  up to the boundary of $\Omega$ with values in $\mathbb{R}^{n}$.
  For any $0<\delta<1, y \in \Omega$ and
  $v\in C^{\infty}_{c}(\Omega)$, we have:
  \begin{equation}\label{eq:hardymag1}
    \textstyle
    \||x-y|^{-1}v\|_{L^{2}(\Omega)}\le
    \frac{2}{n-2}\|\nabla^{b}v\|_{L^{2}(\Omega)},
  \end{equation}
  \begin{equation}\label{eq:stnorma12}
    \||x|^{-1}v\|^{2}_{Y}\le
    6\|\nabla^{b}v\|^{2}_{Y}+ 3\|v\|^{2}_{X},
  \end{equation}
  \begin{equation}\label{eq:stnorma13}
    \textstyle
    \int_{\Omega_{\le1}} \frac{|\nabla^{b}v| |v|}{|x|}dx+
    \int_{\Omega_{\ge1}}\frac{|\nabla^{b}v| |v|}{|x|^{2+\delta}}dx
    \le
    9 \delta^{-1}
    (\|\nabla^{b}v\|^{2}_{Y}+\|v\|^{2}_{X}),
  \end{equation}
  \begin{equation}\label{eq:stnorma7}
    \textstyle
    \|v\|_{X}\le
    4\sup_{R>1}
    R^{-2} \int_{\Omega_{=R}}|v|^{2}dS
    +
    13\|\nabla^{b} v\|^{2}_{Y}.
  \end{equation}
\end{lemma}

\begin{proof}
  We give the complete proof of \eqref{eq:hardymag1};
  the remaining inequalities are proved in
  \cite{CacciafestaDAnconaLuca14-a}.
  Integrating on $\Omega$ the identity
  \begin{equation*}
    \textstyle
    \nabla \cdot 
    \left\{\frac{x-y}{|x-y|^2}|v|^2\right\} = 
    \Re\left[ 2 c(x) \overline{\nabla^b f(x)}\frac{x-y}
    {|x-y|^2}\right] + (n-2)\frac{|c(x)|^2}{|x-y|^2}
  \end{equation*}
  and noticing that boundary term vanishes, we get
  \begin{equation*}
    \textstyle
    \frac{n-2}{2}
    \int_{\Omega} \frac{\abs{f(x)}^2}{\abs{x-y}^2} \, dx 
    \leq 
    \textstyle
    \Re \int_{\Omega} \frac{(x-y) f(x)}{\abs{x-y}^2} 
    \overline{\nabla^b f(x)}\,dx  
    \leq 
    \textstyle\left(\int_{\Omega} 
    \frac{\abs{f(x)}^2}{\abs{x-y}^2}\right)^{\frac12} 
    \left(\int_{\Omega} \abs{\nabla^b f(x)}^2\,dx\right)^{\frac12}. \qedhere
  \end{equation*}
\end{proof}

By a density argument, it is clear that the previous estimates 
are valid
not only for smooth functions but also for functions 
in $D(L)=H^{1}_{0}(\Omega)\cap H^{2}(\Omega)$.

We conclude this section with some additional properties
of the magnetic norms.

\begin{lemma}[]\label{lem:equivmag}
  Let $n\ge3$.
  If $b\in L^{n,\infty}(\Omega)$, the following equivalence holds:
  \begin{equation}\label{eq:equivmag}
    \|\nabla^{b}v\|_{L^{2}(\Omega)}
    \simeq
    \|\nabla v\|_{L^{2}(\Omega)}.
  \end{equation}
  Moreover, for $s>0$ we have
  \begin{equation}\label{eq:equivmagw}
    \|\bra{x}^{-s}\nabla^{b}v\|_{L^{2}(\Omega)}
    +
    \|\bra{x}^{-s-1}v\|_{L^{2}(\Omega)}
    \simeq
    \|\bra{x}^{-s}\nabla v\|_{L^{2}(\Omega)}
    +
    \|\bra{x}^{-s-1}v\|_{L^{2}(\Omega)}.
  \end{equation}
\end{lemma}

\begin{proof}
  By H\"{o}lder inequality and Sobolev embedding 
  in Lorentz spaces, we can write
  \begin{equation*}
    \|\nabla^{b}v\|_{L^{2}}
    \le
    \|\nabla v\|_{L^{2}}+\|bv\|_{L^{2}}
    \le
    \|\nabla v\|_{L^{2}}+\|b\|_{L^{n,\infty}}
    \|v\|_{L^{\frac{2n}{n-2},2}}
    \lesssim
    (1+\|b\|_{L^{n,\infty}})\|\nabla v\|_{L^{2}}.
  \end{equation*}
  Conversely, writing $\nabla=\nabla^{b}-ib$, we have
  \begin{equation*}
    \|\nabla v\|_{L^{2}}
    \le
    \|\nabla^{b}v\|_{L^{2}}+\|bv\|_{L^{2}}
    \lesssim
    \|\nabla^{b}v\|_{L^{2}}+\|b\|_{L^{n,\infty}}
    \|v\|_{L^{\frac{2n}{n-2},2}}.
  \end{equation*}
  Recall now the pointwise diamagnetic inequality
  \begin{equation}\label{eq:diam}
    |\nabla |v||\le
    |\nabla^{b}v|
  \end{equation}
  which is true for $b\in L^{2}_{loc}$.
  Thus, again by Sobolev-Lorentz embedding,
  \begin{equation*}
    \|v\|_{L^{\frac{2n}{n-2},2}}
    \lesssim
    \|\nabla|v|\|_{L^{2}}
    \le
    \|\nabla^{b}v\|_{L^{2}}
  \end{equation*}
  and we obtain \eqref{eq:equivmag}.
  Next we can write
  \begin{equation*}
    \|\bra{x}^{-s}\nabla v\|_{L^{2}}
    +
    \|\bra{x}^{-s-1} v\|_{L^{2}}
    \simeq
    \|\nabla(\bra{x}^{-s}v)\|_{L^{2}}
    +
    \|\bra{x}^{-s-1} v\|_{L^{2}}
  \end{equation*}
  and
  \begin{equation*}
    \|\bra{x}^{-s}\nabla^{b} v\|_{L^{2}}
    +
    \|\bra{x}^{-s-1} v\|_{L^{2}}
    \simeq
    \|\nabla^{b}(\bra{x}^{-s}v)\|_{L^{2}}
    +
    \|\bra{x}^{-s-1} v\|_{L^{2}}
  \end{equation*}
  which, together with \eqref{eq:equivmag}, imply
  \eqref{eq:equivmagw}.
\end{proof}

\begin{lemma}[]\label{lem:distsob}
  Let $n\ge3$ and consider the operator
  $L=A^{b}-c$ with Dirichlet b.c.~on $\Omega$, under assumptions
  \eqref{eq:assopL}, \eqref{eq:assader},
  \eqref{eq:assdb},
  \eqref{eq:assac} and
  \eqref{eq:assbdry}.
  If the constant
  $C_{-}$ is sufficiently small, the operator $L$ is
  selfadjoint and nonpositive. If in addition
  $b\in L^{n,\infty}(\Omega)$ then 
  for all $0\le s\le1$ we have the equivalence
  \begin{equation}\label{eq:homsob}
    \|(-L)^{\frac s2}v\|_{L^{2}(\Omega)}
    \simeq
    \|v\|_{\dot H^{s}(\Omega)}.
  \end{equation}
\end{lemma}

\begin{proof}
  Selfadjointness and positivity
  are standard, and actually hold under 
  less restrictive assumptions on the coefficients
  (see Proposition \ref{pro:powereq} below for a more general
  result). Next, \eqref{eq:homsob} is trivial for $s=0$,
  while for $s=1$ we have
  \begin{equation*}
    \textstyle
    \|(-L)^{\frac12}v\|_{L^{2}}^{2}=
    (-Lv,v)_{L^{2}(\Omega)}=
    a(\nabla^{b}v,\nabla^{b}v)+\int_{\Omega} c|v|^{2}dx
  \end{equation*}
  which implies, using \eqref{eq:equivmag},
  \begin{equation*}
    \textstyle
    \|(-L)^{\frac12}v\|_{L^{2}}^{2}
    \simeq
    \|\nabla^{b}v\|^{2}_{L^{2}}
    +\int_{\Omega} c|v|^{2}dx
    \simeq
    \|\nabla v\|^{2}_{L^{2}}
    +\int_{\Omega} c|v|^{2}dx.
  \end{equation*}
  By Hardy's inequality we obtain the claim for $s=1$, 
  provided $C_{-}$ is sufficiently small, and by
  complex interpolation we conclude the proof
  (recalling the complex interpolation formula
  $[D(H^{\sigma_{0}}),D(H^{\sigma_{1}})]_{\theta}=
  D(H^{\sigma_{\theta}})$ with
  $\sigma_{\theta}=(1-\theta)\sigma_{0}+\theta \sigma_{1}$
  which is valid for any selfadjoint operator $H$).
\end{proof}

\section{Virial identity}\label{sec:virial} 

In \cite{CacciafestaDAnconaLuca14-a}
a virial identity for the Helmholtz equation
with variable coefficients was obtained by adapting
the Morawetz multiplier method. We show here how to
modify the technique in order to prove the analogous
virial identity for the nonlinear
Schr\"{o}dinger equation \eqref{eq:main}.
To make sense of the formal manipulations, one needs some
additional smoothness (e.g., $u\in H^{2}(\Omega)$ is enough),
which can be obtained by an approximation procedure
similar to the proof of the conservation of energy in
Theorem \ref{the:global} below; we omit the details.
The identity is the following:

\begin{proposition}[Virial Identity]\label{pro:virialidentiy}
  Assume $a,b,c,f(z)$ are as in Theorem \ref{the:1}, let $u$
  be a solution of \eqref{eq:main} and $\psi\colon \R^n \to \R$
  an arbitrary weight. 
  Then the following identity holds:
  \begin{equation}\label{eq:virial}
    \begin{split}
      \partial_t [\Im(a(\nabla\psi,\nabla^b u)u)]=& 
      -\textstyle \frac12 A^{2}\psi|u|^{2}
    + \Re(\alpha_{\ell m}\ 
      \partial^{b}_{m}u\ \overline{\partial^{b}_{\ell}u}) \\
      & - a(\nabla \psi,\nabla c)|u|^{2} \\
      & +2\Im(a_{jk}\partial^{b}_{k}u
       (\partial_{j}b_{\ell}-\partial_{\ell}b_{j})
       a_{\ell m}\partial_{m}\psi\ \overline{u}) \\
      & + A\psi [f(u)\bar u - 2 F(u)] \\
      & + \partial_j \{ -\Re Q_j + 2 F(u) a_{jk}\partial_k \psi  + \Im[u_t \bar u a_{jk}\partial_k \psi]\} ,
    \end{split}
  \end{equation}
  where
  \begin{gather}%
    \alpha_{\ell m}=2a_{jm}\partial_{j}(a_{\ell k}
       \partial_{k}\psi)
            -a_{jk}\partial_{k}\psi \partial_{j}a_{\ell m}, \label{eq:alpha}\\
    Q_{j}=
    a_{jk}\partial^{b}_{k}u \cdot 
      [A^{b},\psi]\overline{u}
      -\frac12 a_{jk}(\partial_{k}A \psi)|u|^{2}
      -a_{jk}\partial_{k}\psi 
      \left[c|u|^{2}+a(\nabla^{b}u,\nabla^{b}u)\right]. \label{eq:Qj}
  \end{gather}
\end{proposition}
\begin{proof}
  We multiply both sides of \eqref{eq:main} by 
  the multiplier
  \begin{equation*}
    [A^b,\psi]\bar u = (A \psi) \bar u + 2 a(\nabla \psi,\nabla u)
  \end{equation*}
  and take real parts. We recall the following identity 
  (which however can be checked directly with some lengthy
  but elementary computations) from
  Proposition 2.1 of 
  \cite{CacciafestaDAnconaLuca14-a}:
  \begin{equation}\label{eq:id1}
  \begin{split}
    \Re[ (-A^{b}u+cu)[A^b,\psi]\overline u ]=&
    -\textstyle \frac12 A^{2}\psi|u|^{2}
    + \Re(\alpha_{\ell m}\ 
      \partial^{b}_{m}u\ \overline{\partial^{b}_{\ell}u})
    \\
    &
    -  a(\nabla \psi,\nabla c)|u|^{2}\\
    &
    +2\Im(a_{jk}\partial^{b}_{k}u
       (\partial_{j}b_{\ell}-\partial_{\ell}b_{j})
       a_{\ell m}\partial_{m}\psi\ \overline{u})
    \\
    &
    -\Re\, \partial_{j}Q_{j},
  \end{split}
  \end{equation}
  where $\alpha_{lm}$ are defined by \eqref{eq:alpha} and $Q_{j}$ by \eqref{eq:Qj}.
  For the terms containing $f(u)$ we can write
  \begin{equation}\label{eq:f}
    \Re(f(u)[A^b,\psi]\bar u)= \,A\psi [f(u)\bar u - 2 F(u)] 
      + \nabla\cdot \{ 2F(u) a \nabla \psi \}.
  \end{equation}
  Indeed, by the assumptions on $f$, there exists a function 
  $g\colon [0,+\infty)\to\R$ such that $f(z)=g(\abs{z}^2)z$. 
  As a consequence,
  \begin{equation*}
  \begin{split}
    \nabla F(u)&=
    \textstyle
    \nabla \, \int_0^{\abs{u}}f(s)\,ds=
    \nabla \int_0^{\abs{u}}g(s^2)s\,ds=
    \\
    &=\frac{1}{2}\nabla
    \textstyle
    \int_0^{\abs{u}^2}g(t)\,dt
    =\Re(g(\abs{u}^2)u\nabla \bar{u})=\Re(f(u)\nabla \bar{u})=\\
    &=\Re(f(u)\overline{\nabla^b u}),
  \end{split}
  \end{equation*}
  since $\Re(f(u)ib\bu)=0$.
  We conclude that
  \begin{equation*}
    \begin{split}
      \Re[f(u)[A\psi \bu + 2 a (\nabla\psi,\nabla^b u)]]=& A\psi f(u)\bu 
      + 2 \nabla \psi^t a \Re(f(u)\overline{\nabla^b u})= \\
      =& A\psi f(u)\bu + 2 \nabla \psi^t a \nabla F(u)=\\
      =& A \psi f(u)\bu + 2 [a\nabla\psi] \cdot \nabla F(u)=\\
      =& A\psi[f(u)\bu- 2F(u)]+\nabla\cdot \{2F(u)a\nabla\psi\},
    \end{split}
  \end{equation*}
  and \eqref{eq:f} is proved.
  Finally, for the terms containing $iu_{t}$ we have the
  identity
  \begin{equation}\label{eq:virialederivata}
    \Re(i \partial_t u [A^b,\psi]\bar u)=
     \, \partial_t [-\Im a(\nabla \psi,{\nabla^b u}) u] + 
    \nabla\cdot \{\Im(u_t \bar u a \nabla \psi)\}.
  \end{equation}
  This can be proved directly as follows:
  \begin{equation*}
    \begin{split}
      \Re[i u_t&[A\psi \bu + 2a(\nabla\psi,\nabla^b u)]]= 
      -\Im[u_t \nabla \cdot (a\nabla\psi)\bu + 2\nabla\psi^t a 
      \overline{\nabla u}u_t - 2i\nabla\psi^t a b \bu u_t]=\\
      =& -\Im[-\nabla u_t^t a \nabla\psi \bu - u_t 
      \overline{\nabla u}^t a \nabla\psi 
      + 2 \nabla \psi^t a \overline{\nabla u}u_t 
      -2i\nabla\psi^t a b \bu u_t + 
      \nabla\cdot\{u_t \bu a \nabla \psi\}]=\\
      =& -\Im[\overline{\nabla u_t}^t a \nabla \psi u 
      + \overline{\nabla u}^t a \nabla \psi u_t] 
      + 2 \Im[i \nabla\psi^t a (b\bu)u_t]+
      \nabla\cdot\{\Im[u_t \bu a \nabla \psi]\}=\\
      =& - \Im[\partial_t(\overline{\nabla u}^t a \nabla \psi u)] 
      + \Im[i \partial_t(\nabla\psi^t a \overline{bu}u)]
      +\nabla\cdot\{\Im[u_t \bu a \nabla \psi]\}=\\
      =& - \Im[\partial_t(\overline{\nabla u}^t a \nabla \psi u)] 
      + \Im[i \partial_t(\nabla\psi^t a \overline{bu}u)]
      +\nabla\cdot\{\Im[u_t \bu a \nabla \psi]\}=\\
      =& \partial_t[-\Im a(\nabla \psi,\nabla^b u)u]
      +\nabla\cdot\{\Im[u_t \bu a \nabla \psi]\}.
    \end{split}
  \end{equation*}
  Gathering \eqref{eq:id1}, \eqref{eq:f} and 
  \eqref{eq:virialederivata} we obtain \eqref{eq:virial}.
\end{proof}

\section{Proof of Theorems \ref{the:1}, \ref{the:1strong}: 
the smoothing estimate}
\label{sec:proofsthms}

Since the arguments for Theorems \ref{the:1} and \ref{the:1strong} 
largely overlap, we shall
proceed with both proofs in parallel. 
The proof consists in integrating the 
virial identity \eqref{eq:virial}
on $\Omega$ and estimating carefully all the terms. 
Note that some of the following computations are similar
to those of Section 4 in \cite{CacciafestaDAnconaLuca14-a}.

\begin{remark}\label{rem:explicsmall1}
  At several steps, we shall need to assume
  that the constants $N/\nu-1$, 
  $C_a$, $C_I$, $C_{c}$, $C_{b}$, $C_{-}$ are small enough.
  We can give explicit conditions on the 
  smallness required in Theorem \ref{the:1} and in Theorem \ref{the:1strong}.
  In both the Theorems the smallness of $C_{-}$ is only required in order to
  make $L$ a selfadjoint, nonpositive operator.
  In view of the magnetic Hardy inequality \eqref{eq:hardymag1},
  it is sufficient to assume
  \begin{equation}\label{eq:condCm}
    \textstyle
    C_{-}\le \frac{2 \sqrt{\nu}}{n-2}.
  \end{equation}

  In Theorem \ref{the:1}
  it is sufficient that
  \begin{equation}\label{eq:assratio}
    \textstyle
    \frac{N}{\nu}\le
    \sqrt{\frac{n^{2}+2n+15}{6(n+2)}}
    \quad \text{for}\ 
    3\le n\le 25,
    \qquad \qquad
    \frac{N}{\nu}<
    \frac{7n-1}{3(n+3)}
    \quad \text{for}\ 
    n\ge 26
  \end{equation}
  and that the following quantities are positive:
  \begin{align*}
    & \textstyle
    \frac{K_0}{2}\nu^2 -\frac{5 N^{2}C_{b} 
      + 12 n C_a(N+C_a)+C_{c}}{\delta}>0, 
    \\
    \textstyle
    \frac{n-1}{3n} \nu^2 - 
    &
    \textstyle
    \frac{ 5N^{2}C_{b}+ 24NC_a}{\delta}>0,  
    \qquad
    \textstyle
    \left(n -\frac{N}{\nu}\right) - \frac{n}{n-1} \nu C_a >0, 
  \end{align*}
  where 
  \begin{equation*}
    \textstyle
    K_{0}:=
    \frac{n-1}{6}-\frac{n+3}{2}\frac{N}{\nu}+n>0.
  \end{equation*}
  We remark that $n-N/\nu >0$ thanks to 
  \eqref{eq:assratio}. On the other hand, the condition
  $K_0>0$ is equivalent to the second equation  
  in \eqref{eq:assratio} and is implied by the first equation in 
  \eqref{eq:assratio} in the case $n\leq 26$.

  In Theorem \ref{the:1strong}
  it is sufficient that the following quantities are positive:
  \begin{gather*}
\textstyle
        \frac{(1-C_I)^2}{78} 
  -8 \delta^{-1}
  [C_{c}+9C_{I}+41 C_{a}(N+C_{a})]
  - 9 \delta^{-1}N^{2}C_{b} >0 ,
  \\
\textstyle
   \frac{(1-C_I)^2}{6} 
  -13 \delta^{-1}
  [C_{c}+38 C_{a}(N+C_{a})]
  - 9 \delta^{-1}N^{2}C_{b} >0
  \\
\textstyle
  \left(n -\frac{N}{\nu}\right) - \frac{n}{n-1} \nu C_a >0.
  \end{gather*}

\end{remark}

The proof is divided into several steps.

\subsection{Choice of the weight \texorpdfstring{$\psi$}{}}
\label{sub:weight}

Define
\begin{equation}\label{psi0}
  \textstyle
  \psi_1(r)=\int_0^r\psi_{1}'(s)ds
\end{equation}
where
\begin{equation*}
  \psi'_1(r)=
  \begin{cases}
  \frac{n-1}{2n}r,
  & r\leq1
  \\
  \frac{1}{2}-\frac{1}{2nr^{n-1}},
  & r>1.
  \end{cases}
\end{equation*}
Then $\psi$ is the
radial function, depending on a scaling parameter $R>0$,
\begin{equation*}
  \textstyle
  \psi(|x|)\equiv\psi_R(|x|):=R\psi_1\left(\frac{|x|}{R}\right).
\end{equation*}
Here and in the following, with a slight abuse, we shall use
the same letter $\psi$ to denote a function $\psi(r)$
defined for $r\in \mathbb{R}^{+}$ and
the radial function $\psi(x)=\psi(|x|)$ 
defined on $\mathbb{R}^{n}$. We compute the first radial
derivatives $\psi^{(j)}(r)=(\frac{x}{|x|}\cdot \nabla)^{j}\psi(x)$
for $|x|>0$:
\begin{equation}\label{eq:psi1}
  \psi'(x)=
  \begin{cases}
  \frac{n-1}{2n}\cdot\frac{|x|}{R},& |x|\leq R
  \\
  \frac{1}{2}-\frac{R^{n-1}}{2n|x|^{n-1}},
  & |x|>R
  \end{cases}
\end{equation}
which can be equivalently written as
\begin{equation*}
  \textstyle
  \psi'(x)=
  \frac{|x|}{2nR}
  \left[
    n \frac{R}{R\vee|x|}-(\frac{R}{R\vee|x|})^{n}
  \right]
\end{equation*}
and implies in particular
\begin{equation}\label{eq:bddpsip}
  \textstyle
  0\le \psi'\le \frac{1}{2}.
\end{equation}
Then we have
\begin{equation}\label{eq:psi2}
  \psi''(x)=
  \textstyle
  \frac{n-1}{2n}\cdot\frac{R^{n-1}}{(R\vee|x|)^n}
  =
  \frac{n-1}{2n}\cdot
  \begin{cases}
  \frac{1}{R} & |x|\leq R
  \\
  \frac{R^{n-1}}{|x|^n} & |x|>R,
  \end{cases}
\end{equation}
\begin{equation}\label{eq:psi3}
  \psi'''(x)=
  \textstyle
  -\frac{n-1}{2}\frac{R^{n-1}}{|x|^{n+1}}\one{|x|\ge R}
\end{equation}
\begin{equation}\label{eq:psi4}
  \psi^{IV}(x)=
  \textstyle
  \frac{n^{2}-1}{2}\cdot \frac{R^{n-1}}{|x|^{n+2}}\one{|x|\ge R}
  -
  \frac{n-1}{2}\frac{1}{R^{2}}\delta_{|x|= R}
  .
\end{equation}
\begin{equation}\label{eq:diffpsi}
  \psi''-\frac{\psi'}{|x|}=
  \begin{cases}
    0
     &|x|\le R\\
    -\frac{1}{2|x|}
    \left(1-\frac{R^{n-1}}{|x|^{n-1}}\right)
     &|x|>R.
  \end{cases}
\end{equation}
Moreover the function (see \eqref{eq:apsigen})
\begin{equation}\label{eq:Apsiphi}
  \textstyle
  A \psi=
  \widehat{a}\psi''
  +
  \frac{\overline{a}-\widehat{a}}{|x|}
  \psi'
  +
  a_{\ell m;\ell}\widehat{x}_{m}\psi'.
\end{equation}
is continuous and piecewise Lipschitz.

\subsection{Estimate of the terms in \texorpdfstring{$|u|^{2}$}{}}
\label{sub:v2}
In this section we consider the terms
\begin{equation}\label{eq:termsv2}
  I_{\abs{u}^2}= -\frac12 A^{2}\psi|u|^{2}
  -a(\nabla \psi,\nabla c)|u|^{2}.
\end{equation}
We compute the quantity $A^{2}\psi$:
after some long but elementary 
computations (see \cite{CacciafestaDAnconaLuca14-a}) we have
\begin{equation}\label{eq:splitting}
  A^{2}\psi(x)=S(x)+R(x)
\end{equation}
where
\begin{equation}\label{eq:Sx2}
\begin{split}
  S(x)=
  &
  \textstyle
  \widehat{a}^{2}\psi^{IV}+
  [2 \overline{a}\widehat{a}-6 \widehat{a}^{2}+4|a 
  \widehat{x}|^{2}]
     \frac{\psi'''}{|x|}+
    \\
  &+
  \textstyle
  [2a_{\ell m}a_{\ell m}
     +\overline{a}^{2}
     -6 \overline{a}\widehat{a}
     +15 \widehat{a}^{2}
     -12 |a \widehat{x}|^{2}]
     \left(\frac{\psi''}{|x|^{2}}-\frac{\psi'}{|x|^{3}}\right)
\end{split}
\end{equation}
and
\begin{equation*}
  \textstyle
\begin{split}
  R(x)=
  &
  \widehat{a}a_{\ell m;\ell}\widehat{x}_{m}{\psi'''}+
  (\overline{a}-\widehat{a})
  a_{jk;j}\widehat{x}_{k}
  \textstyle
  \left(
  \frac{\psi''}{|x|}-\frac{\psi'}{|x|^{2}}
  \right)+
    \\
  &+
    [
  \partial_{j}(a_{jk}a_{\ell m;k}\widehat{x}_{\ell}
  \widehat{x}_{m})
  +\partial_{j}(a_{jk}a_{\ell m})\partial_{k}(\widehat{x}_{\ell}
      \widehat{x}_{m})
  ]
  \textstyle
  \left(
  {\psi''}-\frac{\psi'}{|x|}
  \right)
  +
  (A \overline{a})\frac{\psi'}{|x|}+
    \\
  &+
    2a_{jk}a_{\ell m;k}\widehat{x}_{\ell}\widehat{x}_{m}
    \widehat{x}_{j}
    \textstyle
    \left({\psi'''}-\frac{\psi''}{|x|}\right)
    +2a(\nabla \overline{a},\nabla \frac{\psi'}{|x|})+
    \\ 
  &+
  \textstyle
  A(a_{\ell m;\ell}\widehat{x}_{m}\psi').
\end{split}
\end{equation*}
The remainder $R(x)$ can be estimated as follows:
recalling that, by definition of $\psi$, we have
\begin{equation*}
  \textstyle
  |\psi'|\le \frac{|x|}{2 (R\vee|x|)},\qquad
  |\psi''|\le \frac{n-1}{2n (R\vee|x|)},\qquad
  |\psi'''|\le \frac{n-1}{2 \abs{x}(R\vee|x|)}
\end{equation*}
and using assumption \eqref{eq:assader},
we find that
\begin{equation}\label{eq:estR}
  |R(x)|\le
  \frac{12 n C_a(N+C_a)}{|x|\bra{x}^{1+\delta}(R\vee|x|)}.
\end{equation}

\subsubsection{Proof of Theorem \ref{the:1}}
We prove that, assuming \eqref{eq:assac}, 
\eqref{eq:assader}, \eqref{eq:assopL}, 
\eqref{eq:assratio},
we have
\begin{equation}\label{eq:Iv2}
\begin{split}
  \textstyle
  \int_{\Omega}\int_0^T I_{\abs{u}^2} \,dtdx \geq &
  \textstyle
  \frac{n-1}{2}\nu
  \frac{1}{R^{2}}\int_{\Omega_{=R}}\widehat{a}
  \norma{u}_{L^2_T}dS\\
  &-
  \textstyle
  \left[\frac{n+3}{2}N-n\nu\right]
  (n-1)
  \int_{\Omega_{\ge R}}
  \widehat{a}
  \frac{R^{n-1}}{|x|^{n+2}}
  \norma{u}_{L^2_T}\,dx \\
  &-(12 n C_a(N+C_a)+C_{c})\delta^{-1}
  \|u\|_{\dot X_x L^2_T}^{2}.
\end{split}
\end{equation}
We focus on the main term $S(x)$.
With our choice of the weight $\psi$
we have in the region $|x|\le R$
\begin{equation}\label{eq:SlessR}
  \textstyle
  S(x)=
  -\frac{n-1}{2}\widehat{a}^{2}
  \frac{1}{R^{2}}\delta_{|x|=R}
\end{equation}
while in the region $|x|>R$
\begin{equation}\label{eq:SgreR}
\begin{split}
  S(x)=
  &
  \textstyle
  (n-1)
  \left[
    \frac{n+3}{2}\widehat{a}-\overline{a}
  \right]
  \widehat{a}
  \frac{R^{n-1}}{|x|^{n+2}}
  -
  2(n-1)
  [|a \widehat{x}|^{2}-\widehat{a}^{2}]
  \frac{R^{n-1}}{|x|^{n+2}}
    \\
  &
  \textstyle
  -
  [
    2a_{\ell m}a_{\ell m}+\overline{a}^{2}
    -6 \overline{a}\widehat{a}+15 \widehat{a}^{2}
    -12|a \widehat{x}|^{2}
  ]
  \bigl(
  1-(\frac{R}{|x|})^{n-1}
  \bigr)
  \frac{1}{2|x|^{3}}.
\end{split}
\end{equation}
Note that $a_{\ell m}a_{\ell m}$ is the square of the
Hilbert-Schmidt norm of the matrix $a(x)$.
We deduce from assumption \eqref{eq:assopL}
\begin{equation*}
  nN\ge\overline{a}\ge n\nu,\qquad
  N\ge|a \widehat{x}|\ge\widehat{a}\ge\nu,\qquad
  a_{\ell m}a_{\ell m}\ge n\nu^{2},
\end{equation*}
so that
\begin{equation}\label{eq:nge4S1}
  \textstyle
  S(x)\le -\frac{n-1}{2}\nu\widehat{a}
  \frac{1}{R^{2}}\delta_{|x|=R}
  \qquad\text{for}\ |x|\le R.
\end{equation}
On the other hand, we have
\begin{equation}\label{eq:formtoratio}
  2|a(x)|_{HS}^{2}+\overline{a}^{2}
    -6 \overline{a}(x)\widehat{a}(x)
    +15 \widehat{a}^{2}(x)
    -12|a(x) \widehat{x}|^{2}\ge
    (2n+n^{2}+15)\nu^{2}-6(n+2)N^{2}
    \ge0
\end{equation}
by \eqref{eq:assratio} (note that the
second condition in \eqref{eq:assratio} implies the
first one when $n\ge26$), thus we get
\begin{equation}\label{eq:nge4S2}
  S(x)\le
  \textstyle
  (n-1)
  \left[\frac{n+3}{2}N-n\nu\right]\widehat{a}
  \frac{R^{n-1}}{|x|^{n+2}}
  \qquad\text{for}\ |x|\ge R.
\end{equation}

Now we can estimate from below the integral
\begin{equation*}
  \textstyle
  -\int_{\Omega}\int_0^T A^{2}\psi|u|^{2}\,dtdx
  =-\int_{\Omega} A^{2}\psi\norma{u(x)}^{2}_{L^2_T}\,dx
  =I+II
\end{equation*}
where
\begin{equation*}
  \textstyle
  I=-\int_{\Omega} S(x)\norma{u(x)}^{2}_{L^2_T}dx,\qquad
  II=-\int_{\Omega} R(x)\norma{u(x)}^{2}_{L^2_T}dx.
\end{equation*}
By \eqref{eq:estR} and \eqref{eq:stnorma1} we have immediately
for any $R>0$
\begin{equation}\label{eq:Rv2}
  \textstyle
  II
  \ge
  -
  24 n \delta^{-1} C_a(N+C_a)
  \|u\|_{\dot X_x L^2_T}^{2}.
\end{equation}
Note that we must first integrate in time over $[0,T]$,
then in space over $\Omega_{=R}$
and finally divide by $R^{2}$ and take the sup in $R>0$;
this gives a reverse norm $\dot X_{x}L^{2}_{t}$ in the
previous estimate.
Concerning the $S(x)$ term $I$,  we have
by \eqref{eq:nge4S1}, \eqref{eq:nge4S2} 
\begin{equation}\label{eq:Sv2}
  \textstyle
  I
  \ge
  \frac{n-1}{2}\nu
  \frac{1}{R^{2}}\int_{\Omega_{=R}}\widehat{a}
  \norma{u}_{L^2_T}^2dS
  -
  \left[\frac{n+3}{2}N-n\nu\right]
  (n-1)
  \int_{\Omega_{\ge R}}
  \widehat{a}
  \frac{R^{n-1}}{|x|^{n+2}}
  \norma{u}_{L^2_T}^2\,dx
\end{equation}
for all $R>0$. 

It remains to consider the second term in \eqref{eq:termsv2};
we have
\begin{equation}\label{eq:repuls}
  \textstyle
  -a(\nabla \psi,\nabla c)|u|^{2}=
  -a(\widehat{x},\nabla c)\psi'|u|^{2}
  \ge 
  -\frac{C_{c}}{|x|^{2}\langle x\rangle^{1+\delta}}
  \psi'|u|^{2}
\end{equation}
thanks to assumption \eqref{eq:assac}.
Since $0<\psi'<1/2$, by estimate \eqref{eq:stnorma1} we obtain
\begin{equation}\label{eq:estc}
  \textstyle
  -\int_{\Omega}\int_0^T a(\nabla \psi,\nabla c)|u|^{2} \,dtdx
  \ge 
  -C_{c}\delta^{-1}\|u\|_{\dot X_x L^2_T}^{2}
\end{equation}
Collecting \eqref{eq:Sv2}, \eqref{eq:Rv2}, and 
\eqref{eq:estc} we get \eqref{eq:Iv2}.

\subsubsection{Proof of Theorem \ref{the:1strong}}
We prove that, assuming  \eqref{eq:assopL}, \eqref{eq:assader},  \eqref{eq:assperturb},
\eqref{eq:assac}, 
we have for all $R>1$
\begin{equation}\label{eq:estIv2bis}
\begin{split}
  \textstyle
  \int_{\Omega}\int_0^T I_{|v|^{2}}\,dtdx\ge
  \, &(1-C_{I})
  \frac{1}{R^{2}}\int_{\Omega_{=R}}\norma{u}_{L^2_T}^{2}dS
  \\
  &
  -8 \delta^{-1}
  [C_{c}+9C_{I}+41 C_{a}(N+C_{a})]
  \norma{u}^{2}_{X L^2_T}
  \\
  &
  -13 \delta^{-1}
  [C_{c}+36 C_{a}(N+C_{a})]
  \|\nabla^{b}u\|^{2}_{Y L^2_T}.
\end{split}
\end{equation}
Writing $a(x)=I+q(x)$ i.e.
$q_{\ell m}:= a_{\ell m}-\delta_{\ell m}$ we have,
with the notations
$\widehat{q}=q_{\ell m}\widehat{x}_{\ell}\widehat{x}_{m}$ and
$\overline{q}=q_{\ell\ell}$,
\begin{equation*}
  a_{\ell m}a_{\ell m}=
  \delta_{\ell m}\delta_{\ell m}+
  2 \delta_{\ell m}q_{\ell m}
     +q_{\ell m}q_{\ell m}
  =
  3+2 \overline{q}+q_{\ell m}q_{\ell m}
\end{equation*}
and also
\begin{equation*}
  \widehat{a}=1+\widehat{q},\qquad
  \overline{a}=3+\overline{q},\qquad
  |a \widehat{x}|^{2}=
  1+2 \widehat{q}+|q \widehat{x}|^{2}.
\end{equation*}
Note that $|q|=|a(x)-I|\le C_{I}\bra{x}^{-\delta}<1$
by assumption \eqref{eq:assperturb}, which implies
\begin{equation*}
  |\overline{q}|\le 3C_{I}\bra{x}^{-\delta},\qquad
  |\widehat{q}|\le C_{I}\bra{x}^{-\delta},\qquad
  |q \widehat{x}|\le C_{I}\bra{x}^{-\delta}
\end{equation*}
so that
\begin{equation*}
\begin{split}
  2a_{\ell m}a_{\ell m}+\overline{a}^{2}
  -6 \overline{a}\widehat{a}+15 \widehat{a}^{2}
  -12|a \widehat{x}|^{2}
  = &
  4 \overline{q}-12 \widehat{q}
  +2q_{\ell m}q_{\ell m}+\overline{q}^{2}
  -6 \overline{q}\widehat{q}
  +15 \widehat{q}^{2}
  -12|q \widehat{x}|^{2}
  \\
  \ge &
  4 \overline{q}-12 \widehat{q}
  -6 \overline{q}\widehat{q}
  -12|q \widehat{x}|^{2}
  \ge
  -46 C_{I}\bra{x}^{-\delta}.
\end{split}
\end{equation*}
We have also $1-C_{I}\le \widehat{a}\le 1+C_{I}$ so that
($n=3$)
\begin{equation*}
  \textstyle
  -\frac{n-1}{2}\widehat{a}^{2}\le 
  -(1-C_{I})^{2},
  \quad
  \textstyle
  \left(\frac{n+3}{2}\widehat{a}-\overline{a}\right)\widehat{a}
   \le 6  C_{I}(1+C_{I})
  < 12 C_{I}
\end{equation*}
Thus under the assumptions of Theorem \ref{the:1strong}
we obtain the estimates
\begin{equation}\label{eq:neq3S1}
  \textstyle
  S(x)\le -
  (1-C_{I})^{2}
  \frac{1}{R^{2}}\delta_{|x|=R}
  \quad\text{for}\ |x|\le R
\end{equation}
and
\begin{equation}\label{eq:neq3S2}
  \textstyle
  S(x)\le 
  24 C_{I}
  \left[
    \frac{R^{2}}{|x|^{5}}+
    \frac{1}{|x|^{3}\langle x\rangle^{\delta}}
  \right]
  \quad\text{for}\ |x|> R.
\end{equation}

Now we can estimate from below the integral
\begin{equation*}
  \textstyle
  -\int_{\Omega}\int_0^T A^{2}\psi|u|^{2}\,dtdx
  =-\int_{\Omega} A^{2}\psi\norma{u(x)}^{2}_{L^2_T}\,dx
  =I+II
\end{equation*}
where
\begin{equation*}
  \textstyle
  I=-\int_{\Omega} S(x)\norma{u(x)}^{2}_{L^2_T}dx,\qquad
  II=-\int_{\Omega} R(x)\norma{u(x)}^{2}_{L^2_T}dx.
\end{equation*}

Concerning the $S(x)$ term $I$, using \eqref{eq:stnorma2} and 
\eqref{eq:stnorma3} in \eqref{eq:neq3S1}, \eqref{eq:neq3S2},
we have for all $R>1$ 
\begin{equation}\label{eq:Sv2n3}
  I
  \ge
  (1-C_{I})^{2}
  \frac{1}{R^{2}}\int_{\Omega_{=R}}
  \norma{u}_{L^2_T}^{2}\,dS
  -72 C_{I} \delta^{-1}
  \|u\|^{2}_{X_x L^2_T}.
\end{equation}

We estimate the now the $II$--term: for all $R>1$, thanks to 
\eqref{eq:estR}, we have
\begin{equation}\label{eq:dastimareII}
\begin{split}
  \textstyle
  II
  \ge &
  - 36 C_a(N+C_a) \int_0^T \int_{\Omega} |x|^{-1}\bra{x}^{-1-\delta}(R\vee|x|)^{-1}\abs{u(t,x)}^2 \,dxdt
  \\
  \ge &
  \textstyle
  - 36 C_a(N+C_a) \int_0^T \left[\int_{\Omega_{\leq 1}}+\int_{\Omega_{\ge 1}}\right] 
  |x|^{-2}\bra{x}^{-1-\delta}\abs{u(t,x)}^2 \,dxdt
\end{split}
\end{equation}
We observe that, thanks to \eqref{eq:stnorma3}, we have
\begin{equation}\label{eq:usopotenziale1}
  \textstyle
  \int_0^T \int_{\Omega_{\geq 1}} \frac{\abs{u}^2}{\abs{x}^2 \bra{x}^{1+\delta}}\,dxdt 
  =
  \int_{\Omega_{\geq 1}} \frac{\norma{u(x)}_{L^2_T}}{\abs{x}^2\bra{x}^{1+\delta}} 
  \leq
  2 \delta^{-1} \norma{u}_{X L^2_T}.
\end{equation}
Moreover, thanks to \eqref{eq:hardymag1} and \eqref{eq:equivnonhom}, we estimate
\begin{equation}\label{eq:usopotenziale2}
\begin{split}
  \textstyle
  \int_0^T \int_{\Omega_{\leq 1}} \frac{\abs{u}^2}{\abs{x}^2\bra{x}^{1+\delta}}\,dxdt 
  \leq&
  \textstyle
  \int_0^T \int_{\Omega_{\leq 1}} {\abs{x}^{-2}}{\abs{u}^2}\,dxdt 
  \leq \,
  4 \int_0^T \int_{\Omega_{\leq 1}} \abs{\nabla^b u}^2\, dxdt
  \\
  \textstyle
  =& \,
  4 \norma{\nabla^b u}^2_{L^2(\Omega_{\leq 1}) L^2_T}
\end{split}
\end{equation}
Gathering \eqref{eq:usopotenziale1} and \eqref{eq:usopotenziale2}, we have
\begin{equation}\label{eq:stimadimbassasingolare}
  \textstyle
  \int_0^T \int_{\Omega} \frac{\abs{u}^2}{\abs{x}^2\bra{x}^{1+\delta}}\,dxdt
  \leq
  2 \delta^{-1} \norma{u}^2_{X L^2_T} + 
  4 \norma{\nabla^b u}^2_{L^2(\Omega_{\leq 1}) L^2_T}.
\end{equation}
We get immediately from \eqref{eq:dastimareII} and \eqref{eq:stimadimbassasingolare} that
\begin{equation}\label{eq:Rv2n3}
\begin{split}
  \textstyle
  II
  \ge 
  - 324 \delta^{-1} C_a(N+C_a) \left[
    \norma{u}^2_{X L^2_T} + 
    \norma{\nabla^b u}^2_{L^2(\Omega_{\leq 1}) L^2_T}
  \right].
\end{split}
\end{equation}

We consider the second term in \eqref{eq:termsv2};
thanks to
\eqref{eq:assac} and \eqref{eq:stimadimbassasingolare} we have
\begin{align}
  \textstyle
  -\int_0^T \int_{\Omega} a(\nabla\psi,\nabla c) \abs{u}^2 \,dxdt 
  \ge &
  - \int_0^T \int_{\Omega}
  \frac{C_{c}}{2}\frac{\abs{u}^2}{\abs{x}^2\bra{x}^{1+\delta}}\,dxdt
  \notag
  \\ 
  \label{eq:estcbis}
  \ge & 
  \textstyle
   - C_{c} \delta^{-1} \norma{u}^2_{X L^2_T} -
   2 C_{c}\norma{\nabla^b u}^2_{L^2(\Omega_{\leq 1}) L^2_T}.
\end{align}

Recalling \eqref{eq:Rv2n3}, \eqref{eq:Sv2n3}
and \eqref{eq:estcbis} we finally get
\begin{equation*}%
\begin{split}
  \textstyle
  \int_{\Omega}\int_0^T I_{|u|^{2}}\,dtdx\ge \,
  &
  \textstyle
  (1-C_{I})^{2}
  \frac{1}{R^{2}}\int_{\Omega_{=R}}\norma{u}^{2}_{L^2_T} dS 
  \\
  &
  -(72C_{I}\delta^{-1} + \delta^{-1}C_{c})\|u\|_{X L^2_T}^{2}
  -2 C_{c}\|\nabla^{b}u\|^{2}_{L^{2}(\Omega_{\le1})L^2_T}
  \\
    &
  -324 \delta^{-1}
  C_{a}(N+C_{a})
  \left[
  \|u\|_{X L^2_T}^{2}+\|\nabla^{b}u\|^{2}_{L^{2}(\Omega_{\le1})L^2_T}
  \right]
\end{split}
\end{equation*}
whence, noticing that 
$\|w\|_{L^{2}(\Omega_{\le1})}\le \sqrt{2}\|w\|_{Y}$,
we have \eqref{eq:estIv2bis} for all $R>1$.

\subsection{Estimate of the terms in
\texorpdfstring{$|\nabla^{b}u|^{2}$}{}}
\label{sub:nablav}

For such terms, using assumption
\eqref{eq:assader}, we shall prove for all $R>0$ the estimate
\begin{equation}\label{eq:estimatingnablau}
  \textstyle
  \int_{\Omega}
  \alpha_{lm} 
      \Re(\partial_l^b u
      \overline{\partial_m^b u})\,dx
  \ge
  \frac{n-1}{nR}\nu^{2}
  \int_{\Omega_{\le R}}
  \norma{\nabla^b u(x)}_{L^2_T}^2 \, dx
  -
  24NC_a\delta^{-1}\|\nabla^{b}u\|^{2}_{ Y_x L^2_T},
\end{equation}
where $\alpha_{lm}$ are the quantities
defined in \eqref{eq:alpha}. The computations here are very
similar to those in Section 4 of 
\cite{CacciafestaDAnconaLuca14-a}.
We split the quantities $\alpha_{\ell m}$ as
\begin{equation*}
  \alpha_{\ell m}(x)=
  s_{\ell m}(x)+r_{\ell m}(x)
\end{equation*}
where the remainder $r_{\ell m}$ gathers all terms containing
some derivative of the $a_{jk}$.
Since the weight $\psi$ is radial, we have
\begin{equation*}
  \textstyle
  s_{\ell m}(x)=
  2 a_{jm} a_{\ell k} \widehat{x}_{j}\widehat{x}_{k}
  \left(
    \psi''-\frac{\psi'}{|x|}
  \right)
  +
  2 a_{jm}a_{j\ell}
  \frac{\psi'}{|x|}
\end{equation*}
while
\begin{equation*}
  r_{\ell m}(x)=
  [2a_{jm}a_{\ell k;j}-a_{jk}a_{\ell m;j}]\widehat{x}_{k}\psi'.
\end{equation*}
We estimate directly
\begin{equation*}%
  |r_{\ell m}(x)
  \Re(\partial^{b}_{\ell}u\overline{\partial^{b}_{m}u})
  |\le 
  3|a(x)||a'(x)||\nabla^{b}u(x)|^{2}
\end{equation*}
and by assumption \eqref{eq:assader}
we obtain
\begin{equation*}
  |r_{\ell m}(x)
  \Re(\partial^{b}_{\ell}u\overline{\partial^{b}_{m}u})
  |\le 3N C_a\langle x\rangle^{-1-\delta}|\nabla^{b}u|^{2}.
\end{equation*}
Integrating in $t \in [0,T]$ first
and then in $x \in\Omega$, we get
\begin{equation*}
\begin{split}
\textstyle
  \int_{\Omega}\int_0^T |r_{\ell m}(x)
  \Re(\partial^{b}_{\ell}u\overline{\partial^{b}_{m}u})|\,dtdx
  &
  \textstyle
  \le 3N C_a\int_{\Omega} \langle x\rangle^{-1-\delta}
  \int_0^T |\nabla^{b}u|^{2}\,dt\,dx\\
  &
  \textstyle
  = 3N C_a\int_{\Omega} \langle x\rangle^{-1-\delta} 
  \norma{\nabla^b u(x)}_{L^2_T}^2\,dx.
\end{split}
\end{equation*}
Thus, using \eqref{eq:stnorma4}, we obtain the estimate
\begin{equation}\label{eq:estRD}
  \textstyle
  \int_{\Omega}\int_0^T
  |r_{\ell m}
  \Re(\partial^{b}_{\ell}u\overline{\partial^{b}_{m}u})|\,dtdx
  \le
  24N C_a\delta^{-1}\|\nabla^{b}u\|^{2}_{ Y_x L^2_T}.
\end{equation}
Concerning the terms $s_{\ell m}$, 
in the region $|x|>R$ we have
\begin{equation*}
  \textstyle
  s_{\ell m}(x)=
  [a_{jm}a_{j\ell}- a_{jm} a_{\ell k} \widehat{x}_{j}
  \widehat{x}_{k}]
    \frac{1}{|x|}
  +\frac{R^{n-1}}{|x|^{n}}
  a_{jm} a_{\ell k} \widehat{x}_{j}\widehat{x}_{k}
  -
  a_{jm}a_{j\ell}\frac{R^{n-1}}{n|x|^{n}}
\end{equation*}
so that, in the sense of positivity of matrices,
\begin{equation*}
  \textstyle
  s_{\ell m}(x)\ge
  [a_{jm}a_{j\ell}- a_{jm} a_{\ell k} \widehat{x}_{j}
  \widehat{x}_{k}]
    \frac{n-1}{n|x|}\ge0
    \quad\text{for $|x|>R$}
\end{equation*}
(indeed, one has
$a_{jm}a_{j\ell}\ge a_{jm}a_{\ell k}\widehat{x}_{j}
\widehat{x}_{k}$
as matrices); on the other hand,
in the region $|x|\le R$ we have
\begin{equation*}
  \textstyle
  s_{\ell m}(x)=
  a_{jm}a_{j\ell}
    \frac{n-1}{nR}
  \quad\text{for $|x|\le R$}.
\end{equation*}
Thus, by the assumption $a(x)\ge\nu I$, one has for all $x$
\begin{equation}\label{eq:estder}
  \textstyle
  s_{\ell m}(x)
  \Re(\partial^{b}_{\ell}u\overline{\partial^{b}_{m}u})
  \ge
  \frac{n-1}{nR}
  \nu^{2}
  \one{|x|\le R}(x)
  |\nabla^{b}u|^{2}.
\end{equation}
Integrating \eqref{eq:estder} with respect to 
$t \in [0,T]$ and $x \in\Omega$, and recalling \eqref{eq:estRD},
we conclude the proof of \eqref{eq:estimatingnablau}.

\subsection{Estimate of the magnetic terms}
\label{sub:estimate_magnetic}

We now consider the terms
\begin{equation*}
  I_{b}:=
  2\Im[a_{jk}\partial^{b}_{k}u
     (\partial_{j}b_{\ell}-\partial_{\ell}b_{j})
     a_{\ell m}\partial_{m}\psi\ \overline{u}]
  =
  2\Im 
  \left[
    (db \cdot a \widehat{x}) \cdot (a \nabla^{b}u)
    \overline{u}\psi'
  \right]
\end{equation*}
where the identity holds for any radial $\psi$, while 
$db$ is the matrix 
\begin{equation*}
  db=[\partial_{j}b_{\ell}-\partial_{\ell}b_{j}]_{j,\ell=1}^{n}.
\end{equation*}

\subsubsection{Proof of Theorem \ref{the:1}}\label{subsub:estimate_magnetic}
We shall prove the estimate
\begin{equation}\label{eq:magnetic4}
  \textstyle
\int_{\Omega} \int_0^T 
|I_{b}|\,dx
 \le
 5 \delta^{-1}N^{2}C_{b}
 (\|\nabla^{b}u\|_{\dot Y_x L^2_T}^{2}
 +\|u\|_{\dot X_x  L^2_T}^{2}).
\end{equation}
Indeed, since $0\le \psi'\le 1/2$ and $|a(x)|\le N$, 
by \eqref{eq:assdb} we have
\begin{equation*}
  \textstyle
  |I_{b}(x)|\le 
  2N^{2}|db(x)|\cdot|\nabla^{b}u||u|\psi'
  \le
  N^2 \frac{\abs{\nabla^b u}\abs{u}}{\abs{x}^{2+\delta}+\abs{x}^{2-\delta}}.
\end{equation*}
We integrate in $t\in [0,T]$, then in $x \in \Omega$,
and we use the H\"older inequality in time:
\begin{equation*}
  \textstyle
  \int_{\Omega}\int_0^T \abs{I_b(x)}\,dtdx
  \leq 
  \textstyle
  N^2 \int_{\Omega}\int_0^T
   \frac{\abs{\nabla^b u}\abs{u}}{\abs{x}^{2+\delta}
   +\abs{x}^{2-\delta}}\,dtdx
   \leq
   N^2 \int_{\Omega} \frac{\norma{\nabla^b u}_{L^2_T}
   \norma{ u}_{L^2_T}}{\abs{x}^{2+\delta}
   +\abs{x}^{2-\delta}}
   \,dx
\end{equation*}
and by estimate \eqref{eq:stnorma10} we obtain \eqref{eq:magnetic4}.

\subsubsection{Proof of Theorem \ref{the:1strong}}
\label{subsub:estimate_magneticstrong}
In this case we prove the estimate
\begin{equation}\label{eq:estIbbis}
  \textstyle
  \int_{\Omega}\int_0^T
  |I_{b}|\,dtdx
   \le
   9 \delta^{-1}N^{2}C_{b}
   (\|\nabla^{b}u\|_{ Y_x L^2_T}^{2}+\|u\|_{ X_x L^2_T}^{2}).
\end{equation}
The proof is completely analogous to the previous one,
using \eqref{eq:assdbstrong} and \eqref{eq:stnorma13}.

\subsection{Estimate of the terms containing 
\texorpdfstring{$f(u)$}{}} \label{sec:f}

We prove here that there exists a $\gamma_{0}>0$ such that
\begin{equation}\label{eq:estfu}
  A\psi [f(u)\bar u - 2F(u)] \geq  
  \frac{\gamma_{0}}{R\vee \abs{x}}[f(u)\bar u -2 F(u)].
\end{equation}
Thanks to \eqref{eq:repulsivity}, it is sufficient to 
check the pointwise inequality
\begin{equation*}
  A \psi(x)\geq \frac{\gamma_{0}}{R\vee \abs{x}}.
\end{equation*}
Indeed, for $\abs{x}\leq R$, 
\begin{equation*}
  \textstyle
  \widehat{a}\psi'' + \frac{\overline{a}-\widehat{a}}{|x|} \psi'=
  \frac{n-1}{2n}\left[\frac{\widehat a}{R}+\frac{\overline a -
      \widehat a}{R}\right]
  =\frac{n-1}{2n}\frac{\overline a}{R}
\end{equation*}
while for $\abs{x}>R$
\begin{equation*}
  \textstyle
  \widehat{a}\psi'' + \frac{\overline{a}-\widehat{a}}{|x|} \psi'=
   \,
  \frac{\widehat{a}}{\abs{x}}
  \frac{n-1}{2n}\frac{R^{n-1}}{\abs{x}^{n-1}}+
  \frac{\overline{a}-\widehat{a}}{\abs{x}}
  \left(\frac{1}{2}- \frac{1}{2n}\frac{R^{n-1}}{\abs{x}^{n-1}}\right) \\
  \geq \frac{\overline{a}-\widehat{a}}{\abs{x}}\frac{n-1}{2n}.
\end{equation*}
Moreover, by \eqref{eq:assader},
\begin{equation*}
  \textstyle
  a_{lm;l}\widehat x_m \psi' \geq 
  -\frac{C_a}{\bra{x}^{1+\delta}}\abs{\psi'} \geq -\frac{C_a}{\abs{x}}\abs{\psi'}.
\end{equation*}
Summing up we get
\begin{equation*}
\begin{split}
  A \psi \geq &
  \begin{cases}
    \frac{n-1}{2nR} \left(\overline a - C_a \right) \quad &\abs{x}\leq R \\
    \frac{1}{2\abs{x}}\left[\frac{n-1}{n}(\overline a - \widehat a)- C_a\right] \quad &\abs{x}>R,
  \end{cases} \\
  \geq & \frac{\gamma_{0}}{R\vee \abs{x}},
\end{split}
\end{equation*}
for any $\gamma_{0}>0$ such that
\begin{equation}\label{eq:gamma}
  \gamma_{0} < 
  \begin{cases}
    \frac{n-1}{2n}\left(\overline a - C_a \right) \quad &\abs{x}\leq R \\
    \frac{1}{2}\left[\frac{n-1}{n}(\overline a - \widehat a)- C_a\right] \quad &\abs{x}>R,
  \end{cases}
\end{equation}
which is possible provided $C_{a}$ is so small that
$C_a <  \frac{n-1}{n}(\overline{a}(x)-\widehat{a}(x))$
(see Remark \ref{rem:explicsmall1}).

\subsection{Estimate of the boundary terms}
\label{sec:boundary}

We now prove that
\begin{equation}\label{eq:termsbdry}
  \int_{\Omega} \partial_j \{ -\Re Q_j + 2 F(u) 
  a_{jk}\partial_k \psi  + \Im[u_t \bar u a_{jk}
  \partial_k \psi]\}\,dx \geq 0.
\end{equation}
Indeed, proceeding exactly as in 
\cite{CacciafestaDAnconaLuca14-a},
we see that assumption \eqref{eq:assbdry} implies
\begin{equation*}
  \textstyle
  \int_{\Omega} \partial_j \Re Q_j \, dx \leq 0.
\end{equation*}
Moreover, at any fixed $t\in [0,T]$ we have
\begin{equation*}
  \int_{\Omega} \nabla \cdot \{2F(u)a \nabla \psi 
  + \Im[u_t \bar u a \nabla \psi]\} = 0.
\end{equation*}
To see this, we integrate 
$\nabla \cdot \{2F(u)a \nabla \psi + 
\Im[u_t \bar u a \nabla \psi]\}$
over the set $\Omega \cap \{\abs{x}\leq R\}$ and  let 
$R \to +\infty$:
the integral over $\abs{x}=R$ tends to 0 since
 $a \nabla \psi \in L^{\infty}(\Omega)$ and thanks to \eqref{eq:growthfu}
\begin{equation}\label{eq:Fl1}
  \textstyle
  \abs{F(u)}\leq \left|\int_0^{\abs{u}}f(s)\,ds
  \right|\lesssim |u|^{\gamma+1} \in L^1(\Omega),
\end{equation}
(recall that $u\in H^1_0(\Omega)$), while
the integral over $\partial \Omega$ vanishes by the 
Diriclet boundary condition since $F(0)=0$.

\subsection{Estimate of the derivative term}
\label{sec:deriv}
We finally estimate the term at the left hand side of
\eqref{eq:virial}.
We need the following Lemma:

\begin{lemma}\label{stimaterminederivata}
  Let $v\in H^{1}_{0}(\Omega)$ and 
  $\psi\colon \mathbb{R}^n \to \mathbb{R}$ be such that
  $\nabla \psi$ and $\abs{x}A \psi $ are bounded. 
  Then there exist 
  $C=C(\norma{a}_{L^\infty},\norma{\nabla\psi}_{L^\infty}
  ,\norma{\abs{x}A\psi}_{L^\infty})>0$ such that
  \begin{equation*}
    \left\vert \int_{\Omega} 
    a(\nabla \psi, \nabla^b v)v\,dx\right\vert 
    \leq C \|v\|_{\dot H^{\frac12}}^2,
  \end{equation*}
\end{lemma}

\begin{proof}
  Define for $f,g \in C^{\infty}_{c}(\Omega)$
  \begin{equation*}
    \textstyle
    T(f,g):=\int_{\Omega} \nabla \psi (x) \cdot a(x) 
    \overline{\nabla^b f(x)}g(x)\,dx
    =
    \int_{\Omega} [a(x)\nabla \psi (x)]\cdot 
    \overline{\nabla^b f(x)}g(x)\,dx.
  \end{equation*}
  We have trivially
  \begin{equation*}
    \textstyle
    \abs{T(f,g)} 
    \leq \int_{\Omega} 
    \abs{[a(x)\nabla \psi (x)]\cdot \overline{\nabla^b f(x)}g(x)}\,dx 
    \le
    C \norma{\nabla^b f}_{L^2(\Omega)} \norma{g}_{L^2(\Omega)}
  \end{equation*}
  with $C=\norma{a\nabla \psi}_{L^\infty}$.
  On the other hand, integration by parts gives
  \begin{equation*}
  \begin{split}
    \abs{T(f,g)}=&
    \textstyle
    \Big\vert \int_{\R^n} [a(x)\nabla \psi (x)] 
    \overline{\nabla^b f(x)}g(x) \,dx \Big\vert= 
    \\
    = &
    \textstyle
     \Big\vert \int_{\R^n} [a(x) \nabla \psi (x)] 
     \nabla^b g(x) \overline{f(x)}\,dx + 
     \int_{\R^n} \nabla\cdot[a(x)\nabla\psi(x)] 
     g(x) \overline{f(x)}\,dx + \\
     &
     \textstyle
     \qquad \qquad \qquad
     -  \int_{\R^n} \nabla\cdot \{[a(x)\nabla\psi(x)] 
     g(x) \overline{f(x)}\}\,dx \Big\vert.
  \end{split}
  \end{equation*}
  Discarding the divergence term and using the boundedness 
  of $\abs{x}A\psi$ we have, for some
  $C=C(\norma{a}_{L^\infty},\norma{\nabla\psi}_{L^\infty},
  \norma{\abs{x}A\psi}_{L^\infty})>0$,
  \begin{equation*}
    \abs{T(f,g)}\leq 
    C \left[\norma{f}_{L^{2}(\Omega)}\norma{\nabla^b g}_{L^{2}(\Omega)} 
    + \norma{f}_{L^{2}(\Omega)}
    \norma {|x|^{-1}g}_{L^{2}(\Omega)}\right]
  \end{equation*}
  which implies, using the magnetic Hardy inequality 
  \eqref{eq:hardymag1},
  \begin{equation*}
    \abs{T(f,g)}\leq C \norma{f}_{L^{2}(\Omega)}
    \norma{\nabla^b g}_{L^{2}(\Omega)}
  \end{equation*}
  for a different 
  $C=C(\norma{a}_{L^\infty},
  \norma{\nabla\psi}_{L^\infty},
  \norma{\abs{x}A\psi}_{L^\infty})>0$.
  The claim then follows by the equivalence
  $\|\nabla^{b}v\|_{L^{2}}\simeq\|\nabla v\|_{L^{2}}$
  proved in Lemma \ref{lem:equivmag}, by complex interpolation
  and by density.
\end{proof}

Applying Lemma \ref{stimaterminederivata} we get
\begin{equation}\label{eq:derterm}
  \Im \int_{\Omega}  a(\nabla \psi, \nabla^b u)u\, dx \leq 
  \left\vert \int_{\Omega} a(\nabla \psi, \nabla^b u)u\, dx 
  \right\vert \leq
  \widetilde{C}  \norma{u}_{\dot H^{\frac{1}{2}}}^2 
\end{equation}
for some $\widetilde{C}$ depending on
$\norma{a}_{L^\infty}$, $\norma{\nabla\psi}_{L^\infty}$, 
$\norma{\abs{x}A\psi}_{L^\infty}$.
Note that even if $\psi$ depends on $R>0$, 
the constant $\widetilde{C}$ does not, since by
\eqref{eq:bddpsip}, \eqref{eq:assader},
\begin{equation*}
  \textstyle
  \norma{a \nabla \psi}_{L^\infty} \leq 
  \frac12
  \norma{a}_{L^\infty}, 
  \qquad
  \norma{\abs{x}A\psi}_{L^\infty} 
  \leq C(C_a, \norma{a}_{L^\infty}).
\end{equation*}

\subsection{Conclusion of the proof}\label{sub:conclusion}

From \eqref{eq:virial}, using \eqref{eq:estfu}, we have
  \begin{equation*}
    \begin{split}
      \partial_t [\Im(a(\nabla\psi,\nabla^b u)u)] \geq & 
      -\textstyle \frac12 A^{2}\psi|u|^{2}
     -\Re a(\nabla \psi,\nabla c)|u|^{2} 
      + \Re(\alpha_{\ell m}\ 
      \partial^{b}_{m}u\ \overline{\partial^{b}_{\ell}u}) \\
      & +2\Im(a_{jk}\partial^{b}_{k}u
       (\partial_{j}b_{\ell}-\partial_{\ell}b_{j})
       a_{\ell m}\partial_{m}\psi\ \overline{u}) \\
       & + \gamma_{0}[f(u)\bar u -2 F(u)](R\vee \abs{x})^{-1} \\
      & + \partial_j \{ -\Re Q_j + 
        2 F(u) a_{jk}\partial_k \psi  + 
      \Im[u_t \bar u a_{jk}\partial_k \psi]\} .
    \end{split}
  \end{equation*}
Integrating with respect to $t \in [0,T]$ and 
then $x \in \Omega$ we obtain
\begin{align}
  \textstyle
  \int_{\Omega} \int_0^T \partial_t 
  \Im &
  [a(\nabla \psi, \nabla^b u)u]\, dtdx \geq 
  \label{align:derivative} 
  \\
  & 
  \textstyle
  - \int_{\Omega} \int_0^T \left[\frac{1}{2}A^2 \psi + 
  \Re a(\nabla \psi, \nabla c)\right]\abs{u}^2\,dtdx \\
  & 
  \textstyle
  + \int_{\Omega} \int_0^T \Re\left[\alpha_{lm}\partial_m^b u 
  \overline{\partial_l^b u}\right]\,dtdx 
  \label{align:nablau} \\
  &
  \textstyle
   + 2 \int_{\Omega} 
   \int_0^T \Im [a_{jk}\partial_k^b u (\partial_j b_l - 
   \partial_l b_j)a_{lm} \partial_m \psi \bar u]\,dtdx 
   \label{align:magnetic} \\
  &
  \textstyle
   + \gamma_{0} \int_{\Omega}\int_0^T  
     \frac{f(u)\bar u -2 F(u)}{R\vee \abs{x}}\,dtdx \\
  &
  \textstyle
   + \int_{\Omega} 
   \int_0^T \partial_j \{ -\Re Q_j + 2 F(u) a_{jk}\partial_k \psi 
   + \Im[u_t \bar u a_{jk}\partial_k \psi]\}\,dtdx 
   \label{align:boundary}
\end{align}
We now use the estimates from the previous sections.

For the term \eqref{align:derivative}, we use \eqref{eq:derterm}:
\begin{equation*}
\begin{split}
  \textstyle
  \int_{\Omega} \int_0^T \partial_t &
  \Im [a(\nabla \psi, \nabla^b u)u]\, dtdx \\
  \leq & 
  \textstyle
  \int_{\Omega}  \Im a(\nabla \psi, \nabla^b u(0))u(0)
  \, dx + 
   \int_{\Omega} \Im a(\nabla \psi, \nabla^b u(T))u(T)
   \, dx\leq  \\
   \leq & 
   \textstyle 
   C \left( \|u(0)\|_{\dot H^{\frac{1}{2}}}^2
   + \|u(T)\|_{\dot H^{\frac{1}{2}}}^2\right),
  \end{split}
\end{equation*}
where $C$ depends on $\norma{a}_{L^\infty}$,
$\norma{\nabla \psi}_{\infty}$, 
$ \norma{\abs{x}A\psi}_{L^\infty}$, but not on $R>0$.

For \eqref{align:boundary} we swap the integrals, then
using \eqref{eq:termsbdry} we see that this term can be
discarded. 

\subsubsection{Proof of Theorem \ref{the:1}}
We estimate \eqref{align:nablau} 
using \eqref{eq:estimatingnablau} and recalling that 
$\norma{\cdot}_{Y}\leq \norma{\cdot}_{\dot Y}$,  while
\eqref{align:magnetic} is estimated using \eqref{eq:magnetic4}.
Summing up, we have obtained
\begin{align}
  C 
  \big( &
    \|u(0)\|_{\dot H^{\frac{1}{2}}}^2
       + \|u(T)\|_{\dot H^{\frac{1}{2}}}^2\big)
    \geq 
  \notag\\
  & 
  \textstyle
  \frac{1}{2}\left(
   \frac{n-1}{2}\nu
  \frac{1}{R^{2}}\int_{\Omega_{=R}}\widehat{a}
  \norma{u}_{L^2_T}dS
   -
  \left[\frac{n+3}{2}N-n\nu\right]
  (n-1)
  \int_{\Omega_{\ge R}}
  \widehat{a}
  \frac{R^{n-1}}{|x|^{n+2}}
  \norma{u}_{L^2_T}\,dx
  \right) \notag \\
  &  
  \textstyle
  -(12 n C_a(N+C_a)+C_{c})\delta^{-1}
  \|u\|_{\dot X_x L^2_T}^{2} \notag\\
  &
  \textstyle
  + \frac{n-1}{nR}\nu^{2}
  \int_{\Omega_{\le R}}
  \norma{\nabla^{b}u(x)}^{2}_{L^2_T}\,dx
  -
  24NC_a\delta^{-1}\|\nabla^{b}u\|^{2}_{\dot Y_x L^2_T} 
  \notag\\
  & 
  \textstyle
  -5 \delta^{-1}N^{2}C_{b}
   \left(\|\nabla^{b}u\|_{\dot Y_x L^2_T}^{2}
   +\|u\|_{\dot X_x L^2_T}^{2}\right) \notag\\
  &
  \textstyle
   + \gamma_{0} \int_{\Omega} 
   \int_0^T \frac{f(u)\bar u -2 F(u)}{R\vee \abs{x}}\,dtdx.
   \label{eq:finalizzabile}
\end{align}
We now take the sup over $R>0$ at the right hand side. 
Denote with $\Sigma(R)$ all the terms which depend on $R$:
\begin{equation*}
\begin{split}
  \Sigma(R):=&
  \textstyle
  \frac{1}{2}\left(
   \frac{n-1}{2}\nu
  \frac{1}{R^{2}}\int_{\Omega_{=R}}
  \widehat{a}\norma{u}_{L^2_T}dS
  \left[\frac{n+3}{2}N-n\nu\right]
  (n-1)
  \int_{\Omega_{\ge R}}
  \widehat{a}
  \frac{R^{n-1}}{|x|^{n+2}}
  \norma{u}_{L^2_T}\,dx
  \right)\\
  &
  \textstyle
  +\frac{n-1}{nR}\nu^{2}
  \int_{\Omega_{\le R}}
  \norma{\nabla^{b}u(x)}^{2}_{L^2_T}\,dx
   + \gamma_{0} \int_{\Omega} \int_0^T \frac{f(u)\bar u -2 F(u)}
   {R\vee \abs{x}}\,dtdx.
\end{split}
\end{equation*}
We shall use the simple remark that if three nonnegative
quantities $f,g,h$ depend on $R$, then
\begin{equation}\label{eq:propsup}
  \sup_{R>0} [f(R) + g(R) + h(R)]\geq 
  \frac13\left(\sup_{R>0} f(R)+ \sup_{R>0} g(R)
  + \sup_{R>0} h(R)\right).
\end{equation}
We now distinguish two cases.

First case: $\frac{n+3}{2}N\ge n\nu$.
Then by \eqref{eq:stnorma2} we get
\begin{equation*}
  \textstyle
  \Sigma(R)\ge
  Z(R)
  - \frac12
    \left[\frac{n+3}{2}N-n\nu\right]
    \|\widehat{a}^{1/2}u\|_{\dot X_x L^2_T}^{2},
\end{equation*}
where
\begin{equation*}
  Z(R):=
  \textstyle   
  \frac{n-1}{4}\nu
  \frac{1}{R^{2}}\int_{\Omega_{=R}}\widehat{a}
  \norma{u}_{L^2_T}^2 dS
  +\frac{n-1}{nR}\nu^{2}
  \int_{\Omega_{\le R}}
  \norma{\nabla^{b}u(x)}^{2}_{L^2_T}\,dx 
    + \gamma_{0} \int_{\Omega} \int_0^T 
    \frac{f(u)\bar u -2 F(u)}{R\vee \abs{x}}\,dtdx.
\end{equation*}
Thanks to \eqref{eq:stnorma2}, \eqref{eq:propsup}, 
and recalling that $\widehat{a}\ge\nu$, we obtain
\begin{equation*}
  \textstyle
  \sup_{R>0} Z(R)\geq \frac{n-1}{12}\nu^2
  \|u\|_{\dot X_x L^2_T}^{2}
    + \frac{n-1}{3n}\nu^2 
    \norma{\nabla^b u}_{\dot Y_x L^2_T}^2 
    + \frac{\gamma_{0}}3 
    \int_{\Omega}\int_0^T \frac{f(u(x))\bar u(x)-2F(u)}{\abs{x}}
    \,dtdx
\end{equation*}
and consequently, again by $\widehat a \geq \nu$,
\begin{equation} \label{eq:stimaS(R)}
  \textstyle
  \sup_{R>0} \Sigma(R) \geq
   \frac{K_0}2 \nu^2\|u\|_{\dot X_x L^2_T}^{2} 
     + \frac{n-1}{3n}\nu^2 
     \norma{\nabla^b u}_{\dot Y_x L^2_T}^2 
     + \frac{\gamma_{0}}3 \int_{\Omega}\int_0^T 
     \frac{f(u(x))\bar u(x)-2F(u)}{\abs{x}}\,dtdx,
\end{equation}
provided we define
\begin{equation}\label{eq:defK}
  \textstyle
  K_{0}:=
  \frac{n-1}{6}-\frac{n+3}{2}\frac{N}{\nu}+n
\end{equation}
which is a strictly positive quantity provided
we assume $N/\nu$ is small enough
(like in \eqref{eq:assratio}).

Second case: $\frac{n+3}{2}N \leq n\nu$.
Then we have
\begin{equation}
  \textstyle
  \Sigma(R)\geq
   \frac{n-1}{4}\nu \frac{1}{R^2}
   \int_{\Omega_{=R}}\widehat a \norma{u}_{L^2_T}^2 \,dS 
   +\frac{n-1}{nR}\nu^2 \int_{\Omega_{\leq R}}
   \norma{\nabla^b u(x)}_{L^2_T}^2 \,dx
   + \gamma \int_{\Omega} 
   \int_0^T \frac{f(u)\bar u -2 F(u)}{R\vee \abs{x}}\,dtdx.
\end{equation}
Thanks to \eqref{eq:propsup},
recalling that $\widehat a \geq \nu$, and observing that 
in this case $K_0 \leq \frac{n-1}{6}$, 
we obtain again \eqref{eq:stimaS(R)}.

By \eqref{eq:finalizzabile}, \eqref{eq:stimaS(R)} 
we conclude that
\begin{equation}\label{eq:concl}
  \textstyle
  M_1\norma{u}_{\dot X_x L^2_T}^2 
  + 
  M_2 \norma{\nabla^b u}_{\dot Y_x L^2_T}^2 
  +
  M_3 \int_{\Omega}\int_0^T \frac{f(u(x))\bar u(x)-2F(u)}{\abs{x}}\,dtdx 
  \leq 
   C \big( \norma{u(0)}_{\dot{ {H}}^{\frac{1}{2}}}^2+ \norma{u(T)}_{\dot{ {H}}^{\frac{1}{2}}}^2\big)
\end{equation}
for some $C>0$, where $\gamma_0$ is defined in \eqref{eq:gamma} and
\begin{equation*}
  \textstyle
  M_1:=
  \frac{K_0}{2}\nu^2 -\frac{5 N^{2}C_{b} 
  + 12 n C_a(N+C_a)+C_{c}}{\delta},
\end{equation*}
\begin{equation*}
  \textstyle
  M_2:= \frac{n-1}{3n} \nu^2 - 
    \frac{ 5N^{2}C_{b}+ 24NC_a}{\delta},
    \qquad
  M_3:= \frac{\gamma_0}3.
\end{equation*}
If the constants $C_{a},C_{b}$ and $C_{c}$ are sufficiently
small, these quantities are positive, and
the estimate \eqref{eq:concl} is effective.

\subsubsection{Proof of Theorem \ref{the:1strong}}

We estimate \eqref{align:nablau} 
 using \eqref{eq:estimatingnablau} and
\eqref{align:magnetic} thanks to \eqref{eq:estIbbis}.
Summing up, we have obtained 
\begin{align}
  C 
  \big( 
    \norma{u(0)}_{\dot{ {H}}^{\frac{1}{2}}}^2 &
    + \norma{u(T)}_{\dot{ {H}}^{\frac{1}{2}}}^2\big)
    \geq 
  \notag \\
  & (1-C_{I})^{2}
  \frac{1}{R^{2}}\int_{\Omega_{=R}}\norma{u}_{L^2_T}^{2}dS 
  \\
  &
  -8 \delta^{-1}
  [C_{c}+9C_{I}+41 C_{a}(N+C_{a})]
  \|u\|^{2}_{X L^2_T}
  \\
  &
  -13 \delta^{-1}
  [C_{c}+36 C_{a}(N+C_{a})]
  \|\nabla^{b}u\|^{2}_{Y L^2_T}
  \\
  &
  \textstyle
  + \frac{n-1}{nR}\nu^{2}
  \int_{\Omega_{\le R}}
  \norma{\nabla^{b}u(x)}^{2}_{L^2_T}\,dx
  -
  24NC_a\delta^{-1}\|\nabla^{b}u\|^{2}_{ Y_x L^2_T} 
  \notag\\
  & 
  \textstyle
   - 9 \delta^{-1}N^{2}C_{b}
   (\|\nabla^{b}u\|_{ Y_x L^2_T}^{2}+\|u\|_{ X_x L^2_T}^{2})  \notag\\
  &
  \textstyle
   + \gamma_0 \int_{\Omega} 
   \int_0^T \frac{f(u)\bar u -2 F(u)}{R\vee \abs{x}}\,dtdx.
   \label{eq:finalizzabile1}
\end{align}
We now take the sup over $R>1$ at the right hand side. 
We denote with $\Sigma(R)$ all the terms which depend on $R$:
\begin{equation*}
\begin{split}
  \Sigma(R):=
  &
  \textstyle
  (1-C_{I})^{2}
  \frac{1}{R^{2}}\int_{\Omega_{=R}}\norma{u}_{L^2_T}^{2}dS 
    \textstyle
  + \frac{n-1}{nR}\nu^{2}
  \int_{\Omega_{\le R}}
  \norma{\nabla^{b}u(x)}^{2}_{L^2_T}\,dx
  \\
   &
   + \gamma \int_{\Omega} 
   \int_0^T \frac{f(u)\bar u -2 F(u)}{R\vee \abs{x}}\,dtdx
\end{split}
\end{equation*}
Thanks to \eqref{eq:stnorma7}, we have, for $0<\theta<1$,
\begin{equation}\label{eq:perilsup1}
\begin{split}
  (1-C_{I})^{2}
  \sup_{R>1} 
  \frac{1}{R^{2}}\int_{\Omega_{=R}}\norma{u}_{L^2_T}^{2}dS \geq 
  \,&(1-\theta)(1-C_{I})^{2}
  \sup_{R>1} 
  \frac{1}{R^{2}}\int_{\Omega_{=R}}\norma{u}_{L^2_T}^{2}dS \\
  &+ \theta (1-C_{I})^{2} \left(\frac1{4}\norma{u}^2_{X L^2_T} - \frac{13}{4} \norma{\nabla^b u}^2_{Y L^2_T}   \right).
  \end{split}
\end{equation}
Note also that we can take $\nu=1-C_{I}$ 
and $N=1+C_{I}$ by assumption
\eqref{eq:assperturb}, while $n=3$. We obtain
\begin{equation}\label{eq:perilsup2}
  \sup_{R>1} \frac{n-1}{nR}\nu^{2}
  \int_{\Omega_{\le R}}
  \norma{\nabla^{b}u(x)}^{2}_{L^2_T}\,dx \geq 
  \frac23 (1-C_I)^2\norma{\nabla^{b}u}^{2}_{Y L^2_T}
\end{equation}
Finally
\begin{equation}\label{eq:perilsup3}
  \gamma_0 \, \sup_{R>1}   \int_{\Omega} 
   \int_0^T \frac{f(u)\bar u -2 F(u)}{R\vee \abs{x}}\,dtdx \geq 
   \gamma_0 \, \int_{\Omega}\int_0^T 
   \frac{f(u)\bar u -2 F(u)}{\bra{x}}\,dtdx.
\end{equation}
We take $\theta:=2/13$ (it is enough to choose $\theta$ 
such that $2/3 \geq (13\theta)/4 $). 
Thanks to \eqref{eq:propsup}, \eqref{eq:perilsup1}, 
\eqref{eq:perilsup2}, \eqref{eq:perilsup3},
we get
\begin{equation}
  \label{eq:stimaSigmaR}
  \begin{split}
    \sup_{R>1} \Sigma(R) \geq & \frac{(1-C_I)^2}{78} \norma{u}^2_{X L^2_T} 
                              +\frac{(1-C_I)^2}{6} \norma{\nabla^{b}u}^{2}_{Y L^2_T} 
                              \\
                              & 
                              +\frac{\gamma_0}{3} 
                              \int_{\Omega}\int_0^T \frac{f(u)\bar u -2 F(u)}{\bra{x}}\,dtdx.
                            \end{split}
                          \end{equation}
By \eqref{eq:finalizzabile1}, \eqref{eq:stimaSigmaR} 
we conclude that
\begin{equation}\label{eq:conclstrong}
\textstyle
\begin{split}
  M_1\norma{u}_{ X_x L^2_T}^2 
  + 
  M_2 \norma{\nabla^b u}_{ Y_x L^2_T}^2 
  +
  M_3 \int_{\Omega}\int_0^T \frac{f(u(x))\bar u(x)-2F(u)}{\bra{x}}\,dtdx &
  \\
  \leq 
   C \big( \norma{u(0)}_{\dot{ {H}}^{\frac{1}{2}}}^2+ \norma{u(T)}_{\dot{ {H}}^{\frac{1}{2}}}^2\big)&
 \end{split}
\end{equation}
for some $C>0$, where 
\begin{equation*}
\begin{split}
  \textstyle
  &M_1':=
  \frac{(1-C_I)^2}{78} 
  -8 \delta^{-1}
  [C_{c}+9C_{I}+41 C_{a}(N+C_{a})]
  - 9 \delta^{-1}N^{2}C_{b},
  \\
  &M_2':= \frac{(1-C_I)^2}{6} 
  -13 \delta^{-1}
  [C_{c}+38 C_{a}(N+C_{a})]
  - 9 \delta^{-1}N^{2}C_{b},
  \\   
  &M_3= \frac{\gamma_0}3,
\end{split}
\end{equation*}
and $\gamma_0$ is defined in \eqref{eq:gamma}.
If the constants $C_{a},C_{b}$,$C_{c}$ and $C_I$ are sufficiently
small,
these quantities are positive and 
the estimate \eqref{eq:conclstrong} is effective.

\subsection{Proof of Corollary \ref{cor:smooheat}}

Since $u=e^{itL}u_{0}$ satisfies equation \eqref{eq:main}
with the choice $f \equiv0$, we see that $u$
satisfies the smoothing estimate \eqref{eq:smoomorwei2}.
By complex interpolation, we deduce from
\eqref{eq:smoomorwei2} the estimate
\begin{equation*}
  \|\bra{x}^{-1-}(-\Delta)^{\frac14}u\|_{L^{2}_{T}L^{2}}
  \lesssim
  \norma{u_{0}}_{\dot{ H}^{\frac{1}{2}}}+ 
      \norma{u(T)}_{\dot{ H}^{\frac{1}{2}}}
\end{equation*}
for all $T>0$. Proceeding exactly as in the proof of
Corollary 1.4 in \cite{CacciafestaDAncona12-a},
from the gaussian bound for $e^{tL}$
in Proposition \ref{pro:heatk} we deduce 
the weighted estimate
\begin{equation*}
  \|w(x)(-L)^{\frac14}v\|_{L^{2}}\lesssim
  \|w(x)(-\Delta)^{\frac14}v\|_{L^{2}}
\end{equation*}
for any $A_{2}$ weight $w$,
and in particular for $w(x)=\bra{x}^{-s}$ for any $s>0$.
Thus we have
\begin{equation*}
  \|\bra{x}^{-1-}(-L)^{\frac14}u\|_{L^{2}_{T}L^{2}}
  \lesssim
  \|\bra{x}^{-1-}(-\Delta)^{\frac14}u\|_{L^{2}_{T}L^{2}}
  \lesssim
  \norma{u_{0}}_{\dot{ H}^{\frac{1}{2}}}+ 
      \norma{u(T)}_{\dot{ H}^{\frac{1}{2}}}
\end{equation*}
and commuting $(-L)^{\frac14}$ with $e^{itL}$,
and recalling the equivalence \eqref{eq:equivLH1},
we obtain
\begin{equation*}
  \|\bra{x}^{-1-}u\|_{L^{2}_{T}L^{2}}
  \lesssim
  \norma{u_{0}}_{L^{2}}+ \norma{u(T)}_{L^{2}}
  \simeq
  \|u_{0}\|_{L^{2}}
\end{equation*}
by the conservation of the $L^{2}$ norm.

\section{Proof of Theorems \ref{the:2}, \ref{the:2strong}:
the bilinear smoothing estimate}
\label{sec:interaction}

Since the arguments for Theorems \ref{the:2} and \ref{the:2strong} 
largely overlap, we shall again
proceed with both proofs in parallel. 

Let $u$ be a solution of \eqref{eq:main},
and write identity \eqref{eq:virial} with a weight
of the form $\psi=\psi(x-y)$, for $x,y\in \Omega$.
In the following formulas, to make notations lighter,
we shall write simply $u(x)$, $u(y)$ instead of
$u(t,x)$, $u(t,y)$.
We have
\begin{equation*}
  \begin{split}
  M(&x,y)=
  -\textstyle \frac12 A_x^{2}\psi(x-y)|u(x)|^{2}
    + \Re(\alpha_{\ell m}(x)\ 
  \partial^{b(x)}_{x_m}u(x)
  \ \overline{\partial^{b(x)}_{x_\ell}u(x)}) \\
  & -\Re a(x)(\nabla_x \psi(x-y),\nabla_x c(x))|u(x)|^{2} \\
  & +2\Im(a_{jk}(x)\partial^{b(x)}_{x_k} u(x)
   (\partial_{x_j}b(x)_{\ell}-\partial_{x_\ell}b(x)_{j})
   a_{\ell m}(x)\partial_{x_m}\psi(x-y)\ \overline{u(x)}) \\
  & + A_x\psi(x-y) [f(u(x))\bar u(x) - 2 F(u(x))] \\
  & + \partial_{x_j} 
  \{ -\Re Q_j(x) + 2 F(u(x)) a_{jk}(x)\partial_{x_k} \psi(x-y)  
  + \Im[u_t(x) \bar u(x) a_{jk}(x)\partial_{x_k} \psi(x-y)]\}.
  \end{split}
\end{equation*}
where $M(x,y)$ is defined by
\begin{equation*}
  M(x,y):=
  \partial_t 
  \{\Im(a_x(\nabla_x \psi(x-y),\nabla^{b(x)}_x u(x))u(x))\}.
\end{equation*}
Note that
in order to distinguish between the two groups
of variables $x$ and $y$, here and in the following
we used the notations
\begin{equation*}
  a(x)(z,w)=a_{jk}(x)z_{j}\overline{w_{k}},
  \qquad
  \partial_{x_{j}}^{b(x)}=\partial_{x_{j}}+ib_{j}(x),
  \qquad
  A^{b(x)}_{x}v=
  \partial^{b(x)}_{x_{j}}(a_{jk}(x)\partial^{b(x)}_{x_{k}}v(x,y))
\end{equation*}
and similarly $A_{x}$, $\nabla_{x}^{b(x)}$;
we shall however stick to simpler notations whenever possible.
The starting point for the proof is the identity
\begin{equation}\label{eq:virialinteraction}
\begin{split}
  \partial_t &
  \{\Im[a(x)(\nabla_x\psi(x-y),\nabla^b u(x))u(x)]\abs{u(y)}^2\}= 
  \\
    &\qquad= M(x,y)\abs{u(y)}^2 + 
    \Im[a(x)(\nabla \psi(x-y),\nabla^b u(x))u(x)]
    \partial_t\{\abs{u(y)}^2\}.
\end{split}
\end{equation}
Since $u$ is a solution of \eqref{eq:main}
and $c$, $f(u)\bar u$ are real valued, we have
\begin{equation*}
\begin{split}
  \partial_t\abs{u}^2= & 2 \Re[u_t \bar u]= 
  2\Re[\bar u (-i A^b u+icu+if(u))]= \\
  = & 2 \Re[-i A^b u \bar u + 
  i c \abs{u}^2 + i f(u)\bar u]= 2 \Im[A^b u \bar u]
\end{split}
\end{equation*}
and using the identity
\begin{equation*}
   A^b u \bar u= - a(\nabla^b u, \nabla^b u) + \nabla \cdot \{a \nabla^b u \bar u\},
\end{equation*}
by the reality of $a(\nabla^b u, \nabla^b u)$ we have
\begin{equation*}
  \partial_t\abs{u(y)}^2=2 \Im[A^{b(y)} u(y) \overline{u(y)}]= 2\nabla_y\cdot \{\Im[a(y) \nabla^{b(y)}_y u(y) \bar u(y)]\}.
\end{equation*}
Thus the last term in \eqref{eq:virialinteraction}
can be manipulated as follows:
\begin{equation*}
  \begin{split}
    \Im[a(x)&(\nabla \psi(x-y),\nabla^b u(x))u(x)]
    \partial_t[\abs{u(y)}^2]=\\
    = & 2\Im[a(x)(\nabla \psi(x-y),\nabla^b u(x))u(x)] 
    \nabla_y\cdot \{\Im[a(y)\nabla^{b(y)} u(y) \bar u(y)]\}=\\
    = & 2 \Im[ (a(x)\overline{\nabla_x^b u(x)})^t u(x)] 
    D^2\psi (x-y)
    \Im[(a(y)\nabla_y^b u(y))\overline{u(y)}] \\
      & + \nabla_y\cdot 
      \{ 2\Im[a(x)(\nabla \psi(x-y),\nabla^b u(x))u(x)] 
         \Im[a(y)\nabla^{b(y)} u(y) \bar u(y)]\}.
  \end{split}
\end{equation*}
Moreover, we rewrite the quantities $\alpha_{\ell m}$
in the form
\begin{equation*}
  \alpha_{\ell m}=
  2(a(x)D^{2}_{x}\psi(x-y)a(x))_{\ell m}
  +r_{\ell m}
\end{equation*}
where the first term is the $\ell m$ entry of the
matrix $a\cdot D^{2}\psi \cdot a$ and
\begin{equation}\label{eq:defr}
  r_{lm}= \partial_k \psi_y(2 a_{jm}a_{lk;j}-a_{jk}a_{lm;j}).
\end{equation}

We choose different weights for the proofs of Theorem \ref{the:2}
and Theorem \ref{the:2strong}:
in the proof of Theorem \ref{the:2strong}
we set
\begin{equation}\label{eq:defnpsi}
  \psi(x-y):= \brasigma{x-y},
\end{equation}
for $\sigma>0$,
where we use the following notation:
\begin{equation*}
  \brasigma{x-y}:=(\sigma^2+\abs{x-y}^2)^{1/2},
\end{equation*}
while in the proof of Theorem \ref{the:2} we take $\sigma=0$ 
in \eqref{eq:defnpsi}, that is to say, we choose
\begin{equation*}
  \psi(x-y):=\abs{x-y},
\end{equation*}
Note that in the following, with a small abuse,
we shall use the same notation for the radial weight function
$\psi(x)$ and for $\psi(r)=\psi(|x|)$.
We gather here some identities satisfied by
$\psi(r)=\brasigma{r}$ for $\sigma\ge 0$:
\begin{gather} \nonumber 
  \psi'=\frac{r}{\brasigma{r}},
   \quad \psi''=\frac{\sigma^2}{\brasigma{r}^3},
  \quad \psi'''=-3\sigma^2\frac{r}{\brasigma{r}^5},
   \quad
  \psi^{IV}= 12 \frac{\sigma^2}{\brasigma{r}^5}-15 
  \frac{\sigma^4}{\brasigma{r}^7},
  \\ \label{eq:derpsiexpl}
  \frac{1}{r^2}\left(\psi''-\frac{\psi'}{r}\right)
  =-\frac{1}{\brasigma{r}^3},
   \quad
  \psi'''-\frac{\psi''}{r}
  =-\sigma^2\left(\frac{r}{\brasigma{r}^5}
  +\frac{1}{\brasigma{r}^3 r}\right) 
  \leq \frac{4\sigma^2}{\brasigma{r}^3 r}.
\end{gather}
Moreover, for $\sigma \geq 0$,
we introduce the notation
\begin{equation*}
  \widehat{(x-y)}_{ m}^\sigma:=
  \frac{x_{m}-y_{m}}{\brasigma{x-y}}.
\end{equation*}
We have (see \eqref{eq:Apsiphi})
\begin{equation*}
\begin{split}
  A_x \brasigma{x-y} = 
  &
  a_{\ell m;\ell}(x)\widehat{(x-y)}_{ m}^\sigma
  +
  \\
  &+
  \frac{\sigma^{2}}{\brasigma{x-y}}
  a_{\ell m}(x)\widehat{(x-y)}_{\ell}^\sigma 
  \widehat{(x-y)}_{ m}^\sigma
  + \frac{\bar a (x)- \widehat{(x-y)}_{\ell }^\sigma
  a_{\ell m}(x)\widehat{(x-y)}_{m}^\sigma}{\brasigma{x-y}}
\end{split}
\end{equation*}
which implies, since the last two terms are non negative,
\begin{equation*}
  A_x \brasigma{x-y} \ge
  a_{\ell m;\ell}(x)\widehat{(x-y)}_{ m}^\sigma
  \geq - |a'(x)|\geq -\frac{C_a}{\bra{x}^{1+\delta}},
\end{equation*}
and, using assumption \eqref{eq:repulsivity}, 
\begin{equation*}
  A_x \brasigma{x-y} [f(u(x))\overline{u(x)}-2F(u(x))]
  \abs{u(y)}^2 \geq -\frac{C_a}{\bra{x}^{1+\delta}} 
  [f(u(x))\overline{u(x)}-2F(u(x))]\abs{u(y)}^2.
\end{equation*}

Now we integrate \eqref{eq:virialinteraction}
on $\Omega^{2}=\Omega_{x}\times \Omega_{y}$.
The divergence terms in $\nabla_{x}$
can be neglected exactly as in the proof of
Theorems \ref{the:1} and \ref{the:1strong}, 
while the divergence terms in
$\nabla_{y}$ vanish on $\partial \Omega_{y}$
and at infinity.
Taking into account the previous
computations we obtain the inequality
\begin{equation}\label{daintegrare}
  \begin{split}
    \textstyle
    2 \int_{\Omega^{2}} &
    \textstyle
    \Re[ (a(x)\overline{\nabla^b u(x)})^t D^2
    \psi(x-y) (a(x)\nabla^b u(x))]
    \abs{u(y)}^2 dxdy+ \\
    + & 2
    \textstyle
    \int_{\Omega^{2}} 
    \Im[(a(x)\overline{\nabla^b u(x)})^t
    u(x)] D^2 \psi(x-y)
    \Im[(a(y)\nabla^b u(y))\overline{u(y)}]   dxdy + \\
    -&
    \textstyle
    \frac{1}{2}\int_{\Omega^{2}} 
    A^2 \psi(x-y) \abs{u(x)}^2\abs{u(y)}^2 dxdy \leq  \\
    \leq  \partial_t 
    &
    \textstyle
    \int_{\Omega^{2}} 
    \Im[a(x)(\nabla\psi(x-y),\nabla^b u(x))u(x)]\abs{u(y)}^2 dxdy 
    +\int_{\Omega^{2}} R(x,y) \abs{u(y)}^2  dxdy,
  \end{split}
\end{equation}
where
\begin{equation}\label{eq:defR}
  \begin{split}
    R(x,y)= & - r_{lm}(x)\partial^b_m u(x) 
      \overline{\partial_l^b u(x)} \\
        & + \Re[a(x)(\nabla\psi(x-y),\nabla c(x))]\abs{u(x)}^2 \\
        & - 2 \Im[a_{jk}(x)\partial^{b(x)}_{k}u(x)
   (\partial_{j}b_{\ell}(x)-\partial_{\ell}b_{j}(x))
   a_{\ell m}(x)\partial_{m}\psi(x-y)\ \overline{u(x)}] \\
        & -C_a \bra{x}^{-1-\delta}
        [f(u(x)) \overline{u(x)}-2F(u(x))]\abs{u(y)}^2.
  \end{split}
\end{equation}
We remark that $R(x,y)$ depends on $y$ only through $\psi$.
In the following sections we integrate \eqref{daintegrare}
in time on an interval $[0,T]$ and we estimate each term
individually.

\subsection{Positivity of quadratic terms in the derivative}
\label{sec:neglecting}

The first two terms in \eqref{daintegrare} can be
dropped from the inequality since their sum is
nonnegative. To check this fact, we rewrite them as the sum
\begin{equation*}
\begin{split}
  \textstyle
  \int_{\Omega^{2}}&  (a(x)\overline{\nabla^b u(x)})^t D^2
  \psi(x-y) (a(x)\nabla^b u(x))
  \abs{u(y)}^2dxdy \\
  & \textstyle
  +\int_{\Omega^{2}}  (a(y)\overline{\nabla^b u(y)})^t D^2
  \psi(x-y) (a(y)\nabla^b u(y))
  \abs{u(x)}^2dxdy\\
  & \textstyle
  +2 \int_{\Omega^{2}} \Im[(a(x)\overline{\nabla^b u(x)})^t
     u(x)] D^2 \psi(x-y)
     \Im[(a(y)\nabla^b u(y))\overline{u(y)}]   dxdy
\end{split}
\end{equation*}
and then positivity follows from the
the following algebraic lemma with the choice
$T(x,y)=D^2_x \psi(x-y)$:

\begin{lemma}
Let $T(x,y)$ be a real, symmetric, nonnegative matrix
depending on $x,y\in \mathbb{R}^{n}$. Then the following
quantity is nonnegative a.e.:
\begin{equation*}
  \begin{split}
    (a(x) \overline{\nabla^b u(x)})^t &
    T(x,y) (a(x)\nabla^b u(x))
    \abs{u(y)}^2
    + (a(y)\overline{\nabla^b u(y)})^t T(x,y) (a(y)\nabla^b u(y))
    \abs{u(x)}^2 \\
    &+ 2  \Im[(a(x)\overline{\nabla^b u(x)})^t
    u(x)] T(x,y)
    \Im[(a(y)\nabla^b u(y))\overline{u(y)}]  \geq 0.
  \end{split}
\end{equation*}
\end{lemma}

\begin{proof}
  Let $\Sigma(x,y)$ be the quantity in the statement.
  Assume first $T=\text{diag}(\lambda_1,\dots,\lambda_n)$
  is diagonal at a point $(x,y)$, with $\lambda_{j}\ge0$.
  We have then
  \begin{equation*}
    \begin{split}
      \Sigma(x,y) & =
      \textstyle
      \sum_{j=1}^n \lambda_j 
      \big\{\abs{(a(x){\nabla^b u(x)})_j }^2 \abs{u(y)}^2+ 
        \abs{(a(y){\nabla^b u(y)})_j }^2\abs{u(x)}^2 \\
        &\qquad \quad+2 \Im[(a(x)\overline{\nabla^b u(x)})_j
        u(x)]  \Im[(a(y)\nabla^b u(y))_j \overline{u(y)}]  \big\} \\
      & \textstyle
      \geq\sum_{j=1}^n 
        \lambda_j \big\{\abs{(a(x){\nabla^b u(x)})_j }^2 
        \abs{u(y)}^2+ 
        \abs{(a(y){\nabla^b u(y)})_j }^2\abs{u(x)}^2 \\
        &\qquad \quad - 2 \abs{(a(x)\overline{\nabla^b u(x)})_j}
        \abs{u(x)}  \abs{(a(y)\nabla^b u(y))_j}
        \abs{\overline{u(y)}}  \big\} \geq 0.
  \end{split}
  \end{equation*}
  If $T(x,y)$ is non diagonal,
  we can find an orthonormal matrix
  $S=S(x,y)$ with real entries
  such that $T =S^t \Lambda S$, with 
  $\Lambda \geq 0$ real and diagonal. This implies
  \begin{equation*}
    \begin{split}
      \Sigma(x,y) 
      = &(S a(x)\overline{\nabla^b u(x)})^t \Lambda 
      (Sa(x)\nabla^b u(x))
      \abs{u(y)}^2 \\
      &+  (Sa(y)\overline{\nabla^b u(y)})^t \Lambda 
      (Sa(y)\nabla^b u(y))
      \abs{u(x)}^2 \\
      & + 2  \Im[(Sa(x)\overline{\nabla^b u(x)})^t
      u(x)] \Lambda
      \Im[(Sa(y)\nabla^b u(y))\overline{u(y)}],
    \end{split}
  \end{equation*}  
  and the claim follows from the previous step.
\end{proof}

\subsection{The \texorpdfstring{$\partial_t$}{} term} 
\label{sec:derivative}

We now consider the first term at the right hand side
of \eqref{daintegrare}, which is differentiated in time.
We need a Lemma:

\begin{lemma}
  Let $\psi(x-y)=\brasigma{x-y}$, for $\sigma \geq 0$.
 Then the following estimate holds:
 \begin{equation*}
  \left\vert \int_{\Omega^{2}} 
     a(x)(\nabla \psi(x-y),\nabla^b u(x))u(x)\, \varphi(y)\,dxdy \,
     \right\vert  
         \lesssim  \norma{\varphi}_{L^1} 
         \norma{u}_{\dot H^{\frac12}}^2,
 \end{equation*}
for an implicit constant independent on $\sigma$.
\end{lemma}

\begin{proof}
  For $f,g\in C^{\infty}_c(\Omega)$, set 
  \begin{equation*}
    \textstyle
    T(f,g):=\int_{\Omega^{2}} 
    a(x)(\nabla\psi(x-y),\nabla^b f(x))g(x)\,\varphi(y)\,dxdy.
  \end{equation*}
We have immediately
\begin{equation}\label{eq:time-1}
  \abs{T(f,g)}\leq \norma{a}_{L^{\infty}}
  \norma{\varphi}_{L^1}\norma{\nabla^b f}_{L^2}\norma{g}_{L^2}.
\end{equation}
On the other hand, integrating by parts we get
\begin{equation}\label{eq:time0}
  \begin{split}
    \abs{T(f,g)} \leq & 
    \textstyle
    \left\vert 
    \int_{\Omega^{2}} a(x)(\nabla\psi(x-y), \nabla^b g(x))f(x)
    \overline{\varphi(y)} dxdy\right\vert \\
    & + \textstyle
    \left\vert 
    \int_{\Omega^{2}} \partial_{x_m}a_{\ell m}(x)\partial_{x_\ell}
    \psi(x-y)\overline{f(x)}g(x)\varphi(y)\,dxdy \right\vert \\
    & + \textstyle
    \left\vert \int_{\Omega^{2}} a_{\ell m}(x)
    \partial_{x_\ell x_m}\psi(x-y)\overline{f(x)}g(x)
    \varphi(y)\, dydx      \right\vert.
    \end{split}
\end{equation}
By assumption \eqref{eq:assader}, we have
\begin{equation}\label{eq:time1}
\begin{split}
  \textstyle |
  \int_{\Omega^{2}} \partial_{x_m}a_{\ell m}(x) &
  \partial_{x_\ell} \psi (x-y) 
  \overline{f(x)}g(x)\varphi(y)\,
  dxdy | \le \\
  & 
  \leq C_a \norma{\varphi}_{L^1} 
  \norma{f}_{L^2} \left\Vert\bra{x}^{-1 - \delta}g  
  \right\Vert_{L^2}  \lesssim \norma{\varphi}_{L^1} 
  \norma{f}_{L^2} \norma{\nabla^b g}_{L^2},
\end{split}
\end{equation}
where in the last step we used \eqref{eq:hardymag1}.
By direct computation
\begin{equation}\label{eq:time2}
  \begin{split}
    \textstyle
    \vert \int_{\Omega^{2}} a_{\ell m}(x)&
    \partial_{x_\ell x_m}\psi(x-y)\overline{f(x)}g(x)
    \varphi(y) dydx
    \vert \\
    &\le \textstyle
     \int_{\R^{2n}} 
    \vert\bar a (x)- a(x)(\widehat{(x-y)}^\sigma,\widehat{(x-y)}^\sigma) 
    \vert \cdot
    \brasigma{x-y}^{-1}\cdot|\overline{f(x)}g(x)\varphi(y)|
    dxdy \\
    &\leq  \textstyle Nn
     \norma{\varphi}_{L^1} 
    \norma{f}_{L^2} 
    \sup_{y}\left(\int_{\R^{n}}
    \frac{\abs{g(x)}^2}{\abs{x-y}^2}\,dx\right)^{\frac12} \\
    &\lesssim    \norma{\varphi}_{L^1} 
    \norma{f}_{L^2} \norma{\nabla^b g}_{L^2}, 
  \end{split}
\end{equation}
again using \eqref{eq:hardymag1} in the last inequality.
By \eqref{eq:time1} and \eqref{eq:time2}, we deduce
from \eqref{eq:time0}
\begin{equation*}
  \abs{T(f,g)}\lesssim \norma{\varphi}_{L^1} 
  \norma{f}_{L^2}\norma{\nabla^b g}_{L^2}.
\end{equation*}
Recalling now the equivalence \eqref{eq:equivmag},
by complex interpolation beetwen this
estimate and \eqref{eq:time-1} we obtain
\begin{equation*}
  \textstyle
  |\int_{\R^n} a(x)(\nabla\psi(x-y),\nabla^b f(x))g(x)\,
  \varphi(y)\,dxdy| \lesssim 
  \norma{\varphi}_{L^1}
  \norma{f}_{\dot H^{\frac12}}
  \norma{g}_{\dot H^{\frac12}}
\end{equation*}
and taking $f=g=u$ we conclude the proof.
\end{proof}

If we choose $\varphi=\abs{u}^2$ in the previous Lemma,
we obtain
\begin{equation}\label{eq:derest}
  \begin{split}
    & \left\vert \int_0^T \partial_t \int_{\Omega^{2}} 
      \Im[a(x)(\nabla\psi(x-y),\nabla^b u(x))u(x)]\abs{u(y)}^2 
      dxdy\,dt\right\vert  \\
    & \qquad \qquad \lesssim
    \norma{u(0)}_{L^2}^2 
    \left[\norma{u(0)}_{\dot H^{\frac{1}{2}}}^2+
    \norma{u(T)}_{\dot H^{\frac{1}{2}}}^2\right],
  \end{split}
\end{equation}
since the $L^2$-norm of the solution is constant in time.

\subsection{The \texorpdfstring{$R(x,y)$}{} term}
\label{sec:remainder}

We now focus on the last term in \eqref{daintegrare}. 
Our goal is to prove
\begin{equation}\label{eq:estrem}
  \left\vert \int_0^T \int_{\Omega^{2}} 
  R(x,y)\abs{u(y)}^2dxdy\,dt\right\vert \lesssim
  \norma{u(0)}_{L^2}^2
  \left[\norma{u(0)}_{\dot H^{\frac{1}{2}}}^2
  +\norma{u(T)}_{\dot H^{\frac{1}{2}}}^2\right].
\end{equation}
The quantity $R(x,y)$, defined by
\eqref{eq:defR}, gives rise to four terms.

For the term containing $r_{\ell m}$ (see \eqref{eq:defr})
we notice that for all $\sigma \ge 0$ we have $|\nabla \psi|\le1$,
hence both in the proof of Theorem \ref{the:2} and Theorem \ref{the:2strong}
we have
\begin{equation*}
  |r_{\ell m}(x)|\le 2N|a'(x)|\le 2NC_{a}\bra{x}^{-1-\delta}
\end{equation*}
using \eqref{eq:assader}.
This implies
\begin{equation*}
  \textstyle   
  |\int_0^T \int_{\Omega^{2}}
    r_{\ell m}(x)\partial^b_m u(x) 
     \overline{\partial_\ell^b u(x)}\abs{u(y)}^2
     dxdydt|
  \lesssim  \norma{u(0)}_{L^2}^{2} \int_{\R^n} 
  \bra{x}^{-1-\delta}
  \norma{\nabla^b u(x)}^2_{L^2_t}dx
\end{equation*}
by the conservation of the $L^{2}$ norm.
In the proof of Theorem \ref{the:2},
by estimate \eqref{eq:stnorma4} and \eqref{eq:smoomor} we obtain
\begin{equation*}
  \textstyle   
  |\int_0^T \int_{\Omega^{2}}
    r_{\ell m}(x)\partial^b_m u(x) 
     \overline{\partial_\ell^b u(x)}\abs{u(y)}^2
     dxdydt|
  \lesssim  \norma{u(0)}_{L^2}^{2}
  \left[\norma{u(0)}_{\dot H^{\frac{1}{2}}}^2+
  \norma{u(T)}_{\dot H^{\frac{1}{2}}}^2\right],
\end{equation*}
and in the proof of Theorem \ref{the:2strong} we get the same result 
thanks to \eqref{eq:stnorma4} and \eqref{eq:smoomorstrong}.

We estimate differently the term containing $c$ in the two proofs.
In the proof of Theorem \ref{the:2},
recalling assumption \eqref{eq:assacnabla}, we have
\begin{equation*}
  \begin{split}
    \textstyle
    |\int_0^T \int_{\Omega^{2}}a(x) &
    (\nabla\psi(x-y),\nabla c(x))\abs{u(x)}^2
    \abs{u(y)}^2dxdydt|
    \\
    &
    \textstyle
    \lesssim \norma{u(0)}_{L^2}\int_{\Omega} 
    \norma{u(x)}^2_{L^2_t}
    |x|^{-2}\bra{x}^{-1-\delta}\,dx
    \lesssim 
    \norma{u(0)}_{L^2_x}^2 \norma{u}_{\dot X L^2_t}^2 
  \end{split}
\end{equation*}
using the inequality \eqref{eq:stnorma1},
and, thanks to \eqref{eq:smoomor},
\begin{equation*}
  \textstyle
    |\int_0^T \int_{\Omega^{2}}
     a(x)(\nabla\psi(x-y),\nabla c)
     \abs{u(x)}^2\abs{u(y)}^2dxdy\,dt|
    \lesssim  \norma{u(0)}_{L^2}^2
    \left[\norma{u(0)}_{\dot H^{\frac{1}{2}}}^2
    +\norma{u(T)}_{\dot H^{\frac{1}{2}}}^2\right].
\end{equation*}
In the proof of Theorem \ref{the:2strong},
recalling assumption \eqref{eq:assacnabla} and thanks to \eqref{eq:stimadimbassasingolare}, we have
\begin{equation}\label{eq:usodopo1}
  \begin{split}
    \textstyle
    |\int_0^T \int_{\Omega^{2}}a(x) &
    (\nabla\psi(x-y),\nabla c(x))\abs{u(x)}^2
    \abs{u(y)}^2dxdydt|
    \\
    &
    \textstyle
    \lesssim \norma{u(0)}_{L^2}
    \int_0^T \int_{\Omega} 
    |x|^{-2}\bra{x}^{-1-\delta}
    \abs{u(x)}^2
    \,dx\,dt
    \\
    &
    \textstyle
    \lesssim
    \norma{u(0)}_{L^2}
    \left[\norma{u}^2_{X L^2_T} 
    +
    \norma{\nabla^b u}^2_{Y L^2_T}
    \right]
    \\
    &
    \textstyle
    \lesssim
    \norma{u(0)}_{L^2}
    \left[\norma{u(0)}_{\dot H^{\frac{1}{2}}}^2
      +\norma{u(T)}_{\dot H^{\frac{1}{2}}}^2\right].
  \end{split}
\end{equation}

We turn to the estimate of the term containing $b(x)$.
In the proof of Theorem \ref{the:2}, $b$ satisfies \eqref{eq:assdb}, 
and we proceed exactly as in  Section 
\ref{subsub:estimate_magnetic} above, 
and then use \eqref{eq:smoomor}. 
In the proof of Theorem \ref{the:2strong}, 
$b$ satisfies \eqref{eq:assdbstrong}
and we proceed exactly as in  Section \ref{subsub:estimate_magneticstrong} above, 
and then use \eqref{eq:smoomorstrong}. 
In both cases we get
\begin{equation*}
  \textstyle
    \int_0^T \int_{\Omega^{2}} \abs{I_b(x)} 
    \abs{u(y)}^2dxdy\,dt 
    \lesssim \norma{u(0)}_{L^2}^2
    \left[\norma{u(0)}_{\dot H^{\frac{1}{2}}}^2
    +\norma{u(T)}_{\dot H^{\frac{1}{2}}}^2\right].
\end{equation*}

For the term containing $f(u)$ we write
\begin{equation*}
  \begin{split}
    C_a
    \textstyle
    \int_0^T| \int_{\Omega^{2}} &
    \bra{x}^{-1-\delta}
    [f(u(x))\overline{u(x)}-2F(u(x))] 
    \abs{u(y)}^2dxdydt 
    \\
    & 
    \textstyle
    \lesssim \norma{u(0)}_{L^2} 
    \left[\norma{u(0)}_{\dot H^{\frac{1}{2}}}^2
    +\norma{u(T)}_{\dot H^{\frac{1}{2}}}^2\right],
  \end{split}
\end{equation*}
by \eqref{eq:smoomor} in the proof of Theorem \ref{the:2} 
and \eqref{eq:smoomorstrong}  in the proof of Theorem \ref{the:2strong},
and this concludes the proof of \eqref{eq:estrem}.

\subsection{The main term}\label{sec:mainterm}
Integrating in time the inequality \eqref{daintegrare} on
$[0,T]$ and collecting estimates 
\eqref{eq:derest},
\eqref{eq:estrem}
and the results of Section \ref{sec:neglecting},
we have proved that
\begin{equation}\label{eq:finale1}
    - \int_0^T \int_{\Omega^{2}}
    A_{x}^2 \psi(x-y) \abs{u(x)}^2\abs{u(y)}^2\,dxdy dt
    \lesssim   \norma{u(0)}_{L^2}^2
    \left[\norma{u(0)}_{\dot H^{\frac{1}{2}}}^2
    +\norma{u(T)}_{\dot H^{\frac{1}{2}}}^2\right].
\end{equation}
We now compute explicitly the quantity $A^2_x \psi$:
we have
\begin{equation*}
  A^2_x \psi(x-y)= S(x,y) + E(x,y)
\end{equation*}
where, using the notations
\begin{equation*}
  \widetilde{a}=
  \widetilde a(x,y)= a(x)\widehat{(x-y)}\cdot\widehat{(x-y)},
  \qquad
  \widehat{x}=\frac{x}{|x|},
\end{equation*}
$S(x,y)$ and $E(x,y)$ are given by
\begin{equation}\label{eq:Sx2interaction}
\begin{split}
  S(x,y)=
  &
  \textstyle
  \widetilde{a}^{2}\psi^{IV}(x-y)+
  [2 \overline{a}(x)\widetilde{a}-6 \widetilde{a}^{2}+4|a(x) 
  \widehat{(x-y)}|^{2}]
     \frac{\psi'''(x-y)}{\abs{x-y}}+
    \\
  &+
  \textstyle
  [2a_{\ell m}(x) a_{\ell m}(x)
     +\overline{a}^{2}(x)
     -6 \overline{a}(x)\widetilde{a}
     +15 \widetilde{a}^{2}
     -12 |a(x) \widehat{(x-y)}|^{2}]\times \\
     & \qquad \qquad
     \textstyle
     \times\left(\frac{\psi''(x-y)}{|x-y|^{2}}-
     \frac{\psi'(x-y)}{|x-y|^{3}}\right)
\end{split}
\end{equation}
and
\begin{equation*}
  \textstyle
\begin{split}
  E(x,y)=
  &
  \widetilde{a}a_{\ell m;\ell}(x)\widehat{(x-y)}_{m}{\psi'''(x-y)}+
  (\overline{a}(x)-\widetilde{a})
  a_{jk;j}(x)\widehat{(x-y)}_{k}
  \textstyle
  \left(
  \frac{\psi''(x-y)}{|x-y|}-\frac{\psi'(x-y)}{|x-y|^{2}}
  \right)+
    \\
  &+
    [
  \partial_{j}(a_{jk}(x) a_{\ell m;k}(x)\widehat{(x-y)}_{\ell}
  \widehat{(x-y)}_{m})
  +
  \partial_{j}(a_{jk}(x) a_{\ell m}(x))\partial_{k}(\widehat{(x-y)}_{\ell}
      \widehat{(x-y)}_{m})
  ]
  \textstyle
  \times \\ & \qquad \qquad\times
  \textstyle
  \left(
  {\psi''(x-y)}-\frac{\psi'(x-y)}{|x-y|}
  \right)
  \\
  &+
  \textstyle
  (A_x \overline{a}(x))\frac{\psi'(x-y)}{|x-y|}
    \\
  &+
    2a_{jk}(x) a_{\ell m;k}(x)\widehat{(x-y)}_{\ell}\widehat{(x-y)}_{m}
    \widehat{(x-y)}_{j}
    \textstyle
    \left({\psi'''(x-y)}-\frac{\psi''(x-y)}{|x-y|}\right)
    \\ 
    &
    \textstyle
    +2a(x)(\nabla \overline{a}(x),\nabla \frac{\psi'(x-y)}{|x-y|})
    +
    A_x(a_{\ell m;\ell}(x)\widehat{(x-y)}_{m}\psi'(x-y)).
\end{split}
\end{equation*}
With long but elementary computations, for $n\ge 3$ and $\sigma \ge 0$ we have that
\begin{equation*}
  \abs{E(x,y)} \leq 5n C_a (N+C_a) 
  \left[
    \frac{1}{\bra{x}^{1+\delta}\abs{x-y}\brasigma{x-y}}
  +\frac{1}{\bra{x}^{1+\delta}\abs{x}\brasigma{x-y}}
  + \frac{1}{\bra{x}^{1+\delta}\abs{x}^2}\right],
\end{equation*}
whence
\begin{equation*}
  \int_{\Omega^2} E(x,y) \abs{u(x)}^2 \abs{u(y)}^2 \, dxdy \lesssim C_a [I + II + III]
\end{equation*}
with an implicit constant depending on $N$ and $n$,
where
\begin{equation*}
  I=\int_{\Omega^{2}}
  \frac{\abs{u(x)}^2\abs{u(y)}^2}
  {\bra{x}^{1+\delta}\abs{x-y}^2}\,dxdy,
  \qquad
  II=\int_{\Omega^{2}}\frac{\abs{u(x)}^2\abs{u(y)}^2}
  {\bra{x}^{1+\delta}\abs{x}\abs{x-y}}\,dxdy
\end{equation*}
and
\begin{equation*}
  III=\int_{\Omega^{2}}\frac{\abs{u(x)}^2\abs{u(y)}^2}
  {\bra{x}^{1+\delta}\abs{x}^2}\,dxdy.
\end{equation*}
We now extend $u(t,x)$ as zero outside $\Omega$,
i.e.~we define the function $U(t,x)$ as
\begin{equation*}
  U(t,x)=u(t,x)
  \ \text{for $x\in \Omega$},
  \qquad
  U(t,x)=0
  \ \text{for $x\not \in \Omega$.}
\end{equation*}
Before proceeding further, we need the following Lemma:

\begin{lemma}\label{lem:stima}
  Let $n\ge3$, $\delta\in(0,1]$.
  There exist $\eta=\eta(n,\delta)>0$
  such that for all $f\in \mathscr{S}$
  \begin{equation*}
    \begin{split}
    \left\Vert \abs{D}^{\frac{3-n}{2}-1} 
    \frac{f}{\bra{\cdot}^{1+\delta}}\right\Vert_{L^2(\R^n)}
    \leq \eta \norma{\abs{D}^{\frac{3-n}{2}}f}_{L^2(\R^n)}, 
    \\
    \left\Vert \abs{D}^{\frac{3-n}{2}-1} 
    \frac{f}{\abs{\cdot}^{\frac12}
    \bra{\cdot}^{\frac12+\delta}}\right\Vert_{L^2(\R^n)}
    \leq \eta \norma{\abs{D}^{\frac{3-n}{2}}f}_{L^2(\R^n)}.
  \end{split}
  \end{equation*}
\end{lemma}

\begin{proof}
  We prove the first inequality.
  By duality, it is equivalent to prove that
  \begin{equation}\label{eq:lemmaprovanda}
    \left\Vert\abs{D}^{\frac{n-3}{2}} 
    \frac{f}{\bra{x}^{1+\delta}}\right\Vert_{L^2(\R^n)} \lesssim 
    \norma{\abs{D}^{\frac{n-3}{2}+1} f}_{L^2(\R^n)}.
  \end{equation}
  If $n=3$, \eqref{eq:lemmaprovanda} is a simple consequence of Hardy inequality \eqref{eq:hardymag1},
  in the case $y=0$, $b\equiv 0$.
  If $n\ge 4$,
  by the Kato-Ponce inequality
  (see e.g.~\cite{GrafakosOh13-a})
  and Sobolev embedding, we have
  \begin{equation}\label{eq:grafakos}
  \begin{split}
    &\left\Vert\abs{D}^{\frac{n-3}{2}} 
    \frac{f}{\bra{x}^{1+\delta}}\right\Vert_{L^2(\R^n)} \\
    &\quad \quad \lesssim
    \norma{\abs{D}^{\frac{n-3}{2}}f}_{L^{\frac{2n}{n-2}}(\R^n)}
    \norma{\bra{x}^{-1-\delta}}_{L^{n}} 
    + \norma{\abs{D}^{\frac{n-3}{2}}\bra{x}^{-1}}_{L^{\frac{2n}{n-1}}}
    \norma{f}_{L^{2n}}\\
    & \quad \quad \lesssim\norma{\abs{D}^{\frac{n-3}{2}+1}f}_{L^2}
  \end{split}
  \end{equation}
  where the implicit constants clearly depend only on
  $n$ and $\delta$.
  The proof of the second inequality is analogous.
\end{proof}

Now, to estimate $I$ we write
\begin{equation*}
\begin{split}
  I
  &=
  \textstyle
  \int_{\mathbb{R}^{n}} \frac{\abs{U(x)}^2}
  {\bra{x}^{1+\delta}}\int_{\mathbb{R}^{n}} 
  \frac{\abs{U(y)}^2}{\abs{x-y}^{2}}dydx  
  \simeq \int_{\mathbb{R}^{n}} 
  \frac{\abs{U(x)}^2}{\bra{x}^{1+\delta}}
  \abs{D}^{2-n}\abs{U(x)}^2\,dx
  \\
  &= 
  \textstyle
  \int_{\R^n} \abs{D}^{\frac{3-n}{2}-1}
  (\bra{x}^{-1-\delta}\abs{U(x)}^2)
  \abs{D}^{\frac{3-n}{2}}\abs{U(x)}^2\,dx 
  \\
  &\leq \left\Vert 
  \abs{D}^{\frac{3-n}{2}-1}
  \frac{\abs{U}^2}
  {\bra{x}^{1+\delta}}
  \right\Vert_{L^2}\norma{\abs{D}^{\frac{3-n}{2}}\abs{U}^2}_{L^2}
\end{split}
\end{equation*}
and applying Lemma \ref{lem:stima} we obtain
\begin{equation*}
  I \le C(n,\delta) 
  \norma{\abs{D}^{\frac{3-n}{2}}\abs{U}^2}_{L^2}^2.
\end{equation*}
Next we split the integral $II$ 
\begin{equation*}
  II=
  \int_{\R^{2n}} 
  \frac{\abs{U(x)}^2\abs{U(y)}^2}
  {\bra{x}^{1+\delta}\abs{x}\abs{x-y}}dxdy 
  =\int_{A}+\int_{B}
\end{equation*}
in the regions
$A=\{(x,y):2\abs{x}\geq  \abs{y}\}$ and 
$B=\{(x,y):2\abs{x} < \abs{y}\}$.
On $A$ we have 
\begin{equation*}
  \begin{split}
  \int_{A}\frac{\abs{U(x)}^2\abs{U(y)}^2}
  {\bra{x}^{1+\delta}\abs{x}\abs{x-y}}dxdy 
  \lesssim &
  \int_{A} \frac{\abs{U(x)}^2\abs{U(y)}^2}
  {\bra{x}^{\frac12+\frac{\delta}{2}}
  \abs{x}^\frac{1}{2}\bra{y}^{\frac12+
  \frac{\delta}{2}}\abs{y}^\frac{1}{2}\abs{x-y}}\,dxdy \\
  \leq & \int_{\R^{2n}} \frac{\abs{U(x)}^2}
  {\bra{x}^{\frac12+\frac{\delta}{2}}\abs{x}^\frac{1}{2}}
    \frac{1}{\abs{x-y}}
    \frac{\abs{U(y)}^2}{\bra{y}^{\frac12+\frac{\delta}{2}}
    \abs{y}^\frac{1}{2}}\,dxdy \\
  = & \int_{\R^n} \frac{\abs{U(x)}^2}
  {\bra{x}^{\frac12+\frac{\delta}{2}}\abs{x}^\frac{1}{2}}
  \abs{D}^{1-n}\frac{\abs{U(x)}^2}
  {\bra{x}^{\frac12+\frac{\delta}{2}}\abs{x}^\frac{1}{2}}\,dx \\
  = &  \left\Vert \abs{D}^{\frac{1-n}{2}}
  \frac{\abs{U}^2}{\abs{\cdot}^{\frac{1}{2}}
  \bra{\cdot}^{\frac{1}{2}+\frac{\delta}{2}}}
  \right\Vert^2_{L^2(\R^n)} 
  \leq  \, C(n,\delta) 
  \norma{\abs{D}^{\frac{3-n}{2}}\abs{U}^2}^2_{L^2(\R^n)},
  \end{split}
\end{equation*}
where in the last step we used Lemma \ref{lem:stima}. 
On the region $B$ we have
 $\abs{x}\lesssim \abs{x-y}$, hence
\begin{equation*}
  \int_B \frac{\abs{U(x)}^2\abs{U(y)}^2}{\bra{x}^{1+\delta}
  \abs{x}\abs{x-y}}dxdy 
  \lesssim  \int_B \frac{\abs{U(x)}^2\abs{U(y)}^2}
  {\bra{x}^{1+\delta}\abs{x}^2}dxdy 
  \leq  III
\end{equation*}

Summing up, we have proved the estimate
\begin{equation}\label{eq:summingup}
  \begin{split}
    \textstyle
    -\int_{\Omega^{2}}A^2_x & \psi(x-y)\abs{u(x)}^2\abs{u(y)}^2dxdy 
    \\
    \textstyle
    &\gtrsim
    -\int_{\Omega^2} S(x,y) \,dxdy
    -III
    -C(n,N,\delta)C_{a}
    \norma{\abs{D}^{\frac{3-n}{2}}\abs{U}^2}_{L^2(\R^n)}^2
  \end{split}
\end{equation}
with an implicit constant depending on $N,n$ only.

\subsubsection{Proof of Theorem \ref{the:2}}

In this case, the expression for $S$ simplifies:
\begin{equation*}
    \textstyle
  S(x,y)=-\abs{x-y}^{-3}
  \left[2 a_{lm}(x)a_{lm}(x)+ 
  \overline{a}^2(x)-6 \overline{a}(x) \widetilde a(x)
  + 15 \widetilde a^2 - 12 \abs{a(x) \widehat{(x-y)}}^2
  \right].
\end{equation*}
Now recalling \eqref{eq:formtoratio},
we see that if $N/\nu-1$ is small enough we have
\begin{equation*}
  -S(x,y)\ge \epsilon_{0}|x-y|^{-3}
\end{equation*}
for some strictly positive constant $\epsilon_{0}$. This implies
\begin{equation}\label{eq:termS}
\begin{split}
  \textstyle
  \int_{\Omega^{2}} -S(x,y) \abs{u(x)}^2\abs{u(y)}^2 dxdy &\geq 
  \epsilon_{0} \int_{\Omega^{2}}  
  \frac{\abs{u(x)}^2\abs{u(y)}^2}{\abs{x-y}^3} dxdy \\
  & =  \epsilon_{0} \norma{\abs{D}^{\frac{3-n}{2}} 
   \abs{U}^2}_{L^2(\R^n)}^2.
\end{split}
\end{equation}
and, from \eqref{eq:summingup}, we get
\begin{equation*}
  \textstyle
  -\int_{\Omega^{2}}A^2_x \psi(x-y)\abs{u(x)}^2\abs{u(y)}^2dxdy
  \gtrsim
  -III
  +(\epsilon_{0}-C(n,N,\delta)C_{a})
  \norma{\abs{D}^{\frac{3-n}{2}}\abs{U}^2}_{L^2(\R^n)}^2
\end{equation*}
with an implicit constant depending on $N,n$ only.
If $C_{a}$ is sufficiently small (with respect
to $N,n,\nu$ and $\delta$), we obtain
\begin{equation*}
  \gtrsim -III+
  \norma{\abs{D}^{\frac{3-n}{2}}\abs{U}^2}_{L^2(\R^n)}^2
\end{equation*}
and integrating in time on $[0,T]$ and
recalling \eqref{eq:finale1}, we arrive at
the estimate
\begin{equation*}
  \textstyle
  \norma{\abs{D}^{\frac{3-n}{2}}\abs{U}^2}_{L^{2}_{T}L^2_{x}}^2
  \lesssim
  \int_{0}^{T}III dt+
  \|u(0)\|_{L^{2}}^{2}
  \left[\norma{u(0)}_{\dot H^{\frac{1}{2}}}^2
  +\norma{u(T)}_{\dot H^{\frac{1}{2}}}^2\right].
\end{equation*}
Note that by \eqref{eq:stnorma1} we can write
\begin{equation*}
  \textstyle
  \int_0^T IIIdt \le
  \|u(0)\|_{L^{2}}^{2}
  \||x|^{-1}\bra{x}^{-\frac12-\frac{\delta}{2}}u\|^{2}
    _{L^{2}_{x}L^{2}_{T}}
  \lesssim
  \|u(0)\|_{L^{2}}^{2}
  \|u\|_{\dot XL^{2}_{T}}^{2}
\end{equation*}
and recalling \eqref{eq:smoomor} this gives
\begin{equation*}
  \textstyle
  \int_0^T IIIdt \leq 
  \norma{u}_{L^2_x}^2\norma{u}_{L^2_T\dot X}^2 \\
  \leq  
  \norma{u(0)}_{L^2}^2 
  \left[\norma{u(0)}_{\dot H^{\frac{1}{2}}}^2
  +\norma{u(T)}_{\dot H^{\frac{1}{2}}}^2\right].
\end{equation*}
In conclusion we have
\begin{equation*}
  \norma{\abs{D}^{\frac{3-n}{2}}\abs{U}^2}_{L^{2}_{T}L^2_{x}}^2
  \lesssim
  \norma{u(0)}_{L^2}^2 
  \left[\norma{u(0)}_{\dot H^{\frac{1}{2}}}^2
  +\norma{u(T)}_{\dot H^{\frac{1}{2}}}^2\right].
\end{equation*}
Note that
\begin{equation*}
  \textstyle
  \norma{\abs{D}^{\frac{3-n}{2}}\abs{U}^2}_{L^2_{x}}^2
  =
  \int_{\mathbb{R}^{n}}
  |D|^{\frac{3-n}{2}}|U|^{2}\cdot|D|^{\frac{3-n}{2}}|U|^{2}dx
  =
  \int_{\mathbb{R}^{n}}
  |U|^{2}\cdot|D|^{3-n}|U|^{2}dx
\end{equation*}
and this can be written, apart from a constant,
\begin{equation*}
  \textstyle
  =
  \int_{\mathbb{R}^{2n}}
  \frac{|U(x)|^{2}|U(y)|^{2}}{|x-y|^{3}}dxdy
  =
 \int_{\Omega^{2}}
 \frac{|u(x)|^{2}|u(y)|^{2}}{|x-y|^{3}}dxdy 
\end{equation*}
which concludes the proof of the Theorem.

\subsubsection{Proof of Theorem \ref{the:2strong}}
We recall the following identities for $a$:
\begin{gather*}
  a = I + q, \quad a_{lm}=\delta_{lm}+q_{lm}, \\
  \bar a = 3 + \bar q, \quad a_{lm}a_{lm}= 3 + 2\bar q + q_{lm}q_{lm}, \\
  \widetilde a = 1 + \widetilde q, \quad 
  \abs{a \widehat{(x-y)}}^2=1 + 2\widetilde q + \abs{q \widehat{(x-y)}}^2.
\end{gather*}
Starting from \eqref{eq:Sx2interaction} and using
formulas \eqref{eq:derpsiexpl} and
the previous identities, we obtain 
\begin{equation*}
  \begin{split}
  -S(x,y) \geq & 15 \frac{\sigma^4}{\brasigma{x-y}^7} 
      + 30 \widetilde{q} \frac{\sigma^4}{\brasigma{x-y}^7} 
       +\left(2\bar q - 6 \widetilde q + 2 \bar q \widetilde q \right)
       \frac{3\sigma^2}{\brasigma{x-y}^5} \\
       & + \left(4 \bar q - 12 \widetilde q - 6 \bar q \widetilde q
       -3{\widetilde{q}}^2-12\abs{q \widehat{(x-y)}}^2\right)\frac{1}{\brasigma{x-y}^3}.
    \end{split}
\end{equation*}
Since we have by assumption
\begin{gather*}
  \abs{\bar q} \leq 3 C_I \bra{x}^{-\delta}, 
  \quad
  \abs{\widetilde q}\leq C_I \bra{x}^{-\delta}, 
  \quad
  \abs{q \widehat{(x-y)}}\leq C_I \bra{x}^{-\delta}.
\end{gather*}
this implies 
\begin{equation}\label{eq:gathering3}
  -S(x,y) \geq  15 \sigma^4 \brasigma{x-y}^{-7} - 
  46 C_I \bra{x}^{-\delta}\brasigma{x-y}^{-3}.
\end{equation}
From \eqref{eq:summingup} and \eqref{eq:gathering3} we have
\begin{equation*}
\begin{split}
  \textstyle
  -\int_{\Omega^{2}}A^2_x & \psi(x-y)\abs{u(x)}^2\abs{u(y)}^2dxdy
  \\
  \gtrsim&
  \int_{\Omega^2} \left( 15 \frac{\sigma^4}{\brasigma{x-y}^7} - \frac{46 C_I}{ \brasigma{x-y}^{3}} \right)
  \abs{u(x)}^2 \abs{u(y)}^2 \,dxdy 
  \\
  &-III
  -C(n,N,\delta)C_{a}
  \norma{u}^4_{L^4}
\end{split}
\end{equation*}
with an implicit constant depending on $N,n$ only.
We let $\sigma \to 0$ and integrate in $t$ on $ [0,T]$: 
recalling \eqref{eq:finale1}, we get
\begin{equation}\label{eq:finale2}
\textstyle
  (15 - 46 C_I -C(n,N,\delta)C_{a}) \norma{u}_{L^4_T L^4}^4 \lesssim   
  \norma{u(0)}_{L^2}^2 
  \left[\norma{u(0)}_{{\dot H}^{\frac{1}{2}}}^2
  +\norma{u(T)}_{{\dot H}^{\frac{1}{2}}}^2\right] + \int_0^T III \, dt.
\end{equation}
Note that by \eqref{eq:stimadimbassasingolare}, \eqref{eq:equivnonhom}, and \eqref{eq:smoomorstrong} we have
\begin{equation}\label{eq:gathering2}
\begin{split}
  \textstyle
  \int_0^T IIIdt 
  &\le
  \|u(0)\|_{L^{2}}^{2}
  \||x|^{-1}\bra{x}^{-\left(1+{\delta}\right)/2}u\|^{2}
    _{L^{2}_{x}L^{2}_{T}} \\
  &\leq
  \|u(0)\|_{L^{2}}^{2}
  \delta^{-1} \left[ \|u\|_{ XL^{2}_{T}}^{2}+  \|\nabla^b u\|_{ Y L^{2}_{T}}^{2}\right] \\
  &\lesssim
  \norma{u(0)}_{L^2}^2 
  \left[\norma{u(0)}_{{\dot H}^{\frac{1}{2}}}^2
  +\norma{u(T)}_{{\dot H}^{\frac{1}{2}}}^2\right].
\end{split}
\end{equation}
If $C_I$ and $C_a$ are small enough,
we get immediately the claim 
from \eqref{eq:finale2} and \eqref{eq:gathering2}.

\section{Gaussian bounds and applications} 
\label{sec:gaussian_bds}

Let $L$ be the operator \eqref{eq:opL}, \eqref{eq:assopL}
defined on an open set $\Omega \subseteq\mathbb{R}^{n}$.
For the results in this section it is not necessary to
assume any condition on $\Omega$ which may be an arbitrary
open set; we shall anyway assume $\partial \Omega\in C^{1}$
for the sake of simplicity.
First of all, we chack that $L$
can be realized as a selfadjoint operator, with Dirichlet b.c.,
under very weak assumptions on the coefficients:

\begin{proposition}[]\label{pro:selfadjoint}
  Let $n\ge3$ and $\Omega \subseteq \mathbb{R}^{n}$
  an open set with $C^{1}$ boundary. Consider the operator
  $L$ defined on $C^{\infty}_{c}(\Omega)$ by
  \eqref{eq:opL}, \eqref{eq:assopL}, under the assumptions
  \begin{equation}\label{eq:assLcoefb}
    a\in L^{\infty},
    \qquad
    b\in L^{n,\infty},
    \qquad
    c\in L^{\frac n2,\infty},
    \qquad
    \|c_{-}\|_{L^{\frac n2,\infty}}
    <\epsilon.
  \end{equation}
  Then, if $\epsilon$ sufficiently small,
  $-L$ extends to a selfadjoint nonnegative operator in the
  sense of forms, and
  $D(-L)=H^{1}_{0}(\Omega)\cap H^{2}(\Omega)$ is a form core.
  Moreover we have 
  \begin{equation}\label{eq:equivLH1}
    (-Lv,v)_{L^{2}}=
    \|(-L)^{\frac12}v\|_{L^{2}}^{2}
    \simeq\|\nabla v\|^{2}_{L^{2}},
    \qquad
    \|(-L)^{\frac14}v\|_{L^{2}}
    \simeq
    \|v\|_{\dot H^{\frac12}}.
  \end{equation}
\end{proposition}

\begin{proof}
  We sketch the proof which is mostly standard,
  apart from the use of Lorentz spaces. The form
  \begin{equation*}
    \textstyle
    q(v)=(-Lv,v)_{L^{2}}=
    \int_{\Omega} a(\nabla^{b}v,\nabla^{b}v)dx
    +
    \int_{\Omega}c|v|^{2}dx
  \end{equation*}
  is bounded on $H^{1}_{0}(\Omega)$: indeed,
  by H\"{o}lder and Sobolev inequalities in Lorentz spaces,
  \begin{equation*}
    \textstyle
    \int_{\Omega}|c|\cdot |v|^{2}dx
    \lesssim
    \|c\|_{L^{\frac n2,\infty}}
    \||v|^{2}\|_{L^{\frac{n}{n-2},1}}
    \lesssim
    \|c\|_{L^{\frac n2,\infty}}
    \|v\|^{2}_{L^{\frac{2n}{n-2},2}}
    \lesssim
    \|c\|_{L^{\frac n2,\infty}}
    \|\nabla v\|^{2}_{L^{2}}
  \end{equation*}
  while by \eqref{eq:equivmag} we have
  $\|\nabla^{b}v\|_{L^{2}}\simeq\|\nabla v\|_{L^{2}}$. Thus
  if $\epsilon$ is sufficiently small
  we have $q(v)\simeq\|\nabla v\|_{L^{2}}$; in particular
  $q(v)$ is a symmetric, closed, nonnegative form on
  $H^{1}_{0}(\Omega)$, and defines a selfadjoint operator
  with $D(-L)=H^{2}(\Omega)\cap H^{1}_{0}(\Omega)$
  which is also a core for $q$.
  The last property in \eqref{eq:equivLH1} follows
  by complex interpolation, since
  $D((-L)^{s})$ for $0\le s\le 1$ is an interpolation
  family.
\end{proof}

Under slightly stronger assumptions, we can see that the heat flow
$e^{tL}$ satisfies an upper gaussian bound; this will
be a crucial tool in the following. Compare with
\cite{DAnconaPierfelice05-a} and
\cite{DAnconaFanelliVega10-a} for similar results in the case
$a=I$, $\Omega=\mathbb{R}^{n}$.
Note that for $a,b,c\in L^{\infty}$ with $c\ge0$
the bound is proved in Corollary 6.14 of \cite{Ouhabaz04-a}.
The following result is sufficient for our purposes,
although the assumptions on
the coefficients could be further relaxed.

\begin{proposition}[]\label{pro:heatk}
  Let $n\ge3$. Assume the operator $L$ is defined 
  as in \eqref{eq:opL}, \eqref{eq:assopL} on the open
  set $\Omega \subseteq \mathbb{R}^{n}$
  with $C^{1}$ boundary, and that $a,b,c$
  satisfy
  \begin{equation}\label{eq:assLheatk}
    a\in L^{\infty},
    \quad
    b\in L^{4}_{loc}\cap L^{n,\infty},
    \quad
    \nabla \cdot b\in L^{2}_{loc},
    \quad
    c\in L^{\frac n2,1},
    \quad
    \|c_{-}\|_{L^{\frac n2,1}}<\epsilon.
  \end{equation}
  Then, if $\epsilon$ is sufficiently 
  small, the heat kernel $e^{tL}$ satisfies, for some $C,C'>0$,
  \begin{equation}\label{eq:gaussestup}
    |e^{tL}(x,y)|\le 
    C' t^{-\frac n2}e^{-\frac{|x-y|^{2}}{Ct}},
    \qquad t>0,\ x,y\in \Omega.
  \end{equation}
\end{proposition}

\begin{proof}
  We can apply Proposition \ref{pro:selfadjoint}
  since the assumptions are stronger.
  When $b=c=0$, the gaussian
  bound follows directly from Corollary 6.14 in
  \cite{Ouhabaz05-a}; note that in this case the kernel
  of $e^{tL}$ is $\ge0$. 

  Next, in order to
  handle the case $b\neq0$, $c=0$, we adapt the
  proof of Lemma 10 in \cite{LeinfelderSimader81-a}.
  Let $\phi\in C^{\infty}_{c}(\Omega)$ and write
  $\phi_{\delta}=\sqrt{|\phi|^{2}+\delta^{2}}$ for $\delta>0$.
  It is easy to prove the pointwise inequality
  (recall notations \eqref{eq:notAAb})
  \begin{equation*}
    \textstyle
    A \phi_{\delta}\ge
    \Re(\frac{\overline{\phi}}{\phi_{\delta}}A^{b}\phi)
  \end{equation*}
  which implies, for all $\lambda>0$,
  \begin{equation*}
    \textstyle
    (-A+\lambda)\phi_{\delta}
    \le
    \Re
    (\frac{\overline{\phi}}{\phi_{\delta}}(-A^{b}+\lambda)\phi)
    +\lambda(\phi_{\delta}-\frac{|\phi|^{2}}{\phi_{\delta}}).
  \end{equation*}
  Proceeding as in \cite{LeinfelderSimader81-a}, we obtain
  \begin{equation*}
    |(-A^{b}+\lambda)^{-1}f|\le(-A+\lambda)^{-1}|f|
  \end{equation*}
  and iterating we have for all $k\ge0$
  \begin{equation}
    |(-A^{b}+\lambda)^{-k}f|\le(-A+\lambda)^{-k}|f|
  \end{equation}
  since $(-A+\lambda)^{-1}$ is positivity preserving
  (see Remark 1 in \cite{LeinfelderSimader81-a}).
  Then we deduce
  \begin{equation*}
    |e^{tA^{b}}\phi|\le e^{tA}|\phi|
  \end{equation*}
  via $e^{tA^{b}}=\lim_{k\to \infty}(I-tA^{b}/n)^{-n}$,
  and applying the last formula to a delta sequence 
  $\phi=\phi_{j}$ made of nonnegative
  functions, we conclude that the gaussian bound
  \eqref{eq:gaussestup} is valid for $e^{tA^{b}}$.

  It remains to consider the case $c\neq0$. To this end we apply
  the theory of \cite{LiskevichVogtVoigt06-a}.
  Let $U(t,s)$ be the propagator
  defined as $U(t,s)f=e^{(t-s)A^{b}}f$, for $t\ge s\ge0$.
  By the gaussian bound
  just proved we have that $U(t,s)$ extends to a uniformly
  bounded operator $L^{1}\to L^{1}$ and 
  $L^{\infty}\to L^{\infty}$, moreover
  $\|U(t,s)\|_{L^{1}\to L^{\infty}}\lesssim|t-s|^{-\frac n2}$;
  finally, the adjoint propagator
  $U_{*}(t,s):=(U(s,t))^{*}$ for $s\ge t\ge0$
  coincides with $U(s,t)$
  since $A^{b}$ is selfadjoint,
  so that $U_{*}$ is strongly continuous on $L^{1}$
  (notice that this last assumption is not actually necessary,
  as mentioned in the paper).
  Then by applying Theorem 3.10 from
  \cite{LiskevichVogtVoigt06-a} we conclude that the gaussian
  bound, with possibly different constants, is satisfied
  also by the perturbed propagator
  $U_{c}=e^{t(A^{b}-c)}$, provided the potential $c$
  is a Miyadera perturbation of both $U$ and $U_{*}$
  such that $c_{-}$ is Miyadera small with constants
  $(\infty,\gamma)$, $\gamma<1$. The verification of this
  last condition, in the special case considered here,
  reduces to the following inequality, for all $s\ge0$
  \begin{equation}\label{eq:miyad}
    \textstyle
    I:=\int_{s}^{+\infty}
    \|c(x)e^{(t-s)A^{b}}f\|_{L^{1}}dt\le \gamma\|f\|_{L1}
  \end{equation}
  (we are using formula (2.5) in \cite{LiskevichVogtVoigt06-a}
  with the choices
  $\alpha=\infty$, $J=\mathbb{R}^{+}$ and $p=1$)
  and the same inequality with $\gamma<1$ for $c_{-}$.
  The gaussian bound already proved for $e^{tA^{b}}$
  implies 
  \begin{equation*}
    \textstyle
    I
    \lesssim
    \int_{\Omega}\int_{\Omega}|c(x)||f(y)|
    \int_{0}^{+\infty}t^{-\frac n2}
    e^{-\frac{|x-y|^{2}}{Ct}}dtdydx
    \lesssim
    \|f\|_{L^{1}}
    \sup_{y\in \Omega}
    \int_{\Omega}\frac{|c(x)|}{|x-y|^{n-2}}dx
  \end{equation*}
  and by the Young inequality in Lorentz spaces we get
  \begin{equation*}
    I \lesssim
    \|c\|_{L^{\frac n2,1}}\|f\|_{L^{1}},
  \end{equation*}
  which concludes the proof
  (compare with the proof of Lemma 5.1 in \cite{Voigt86-a}).
\end{proof}

\begin{proposition}[]\label{pro:powereq}
  Let $n\ge3$. Assume the operator $L$ is defined 
  as in \eqref{eq:opL}, \eqref{eq:assopL} on the open
  set $\Omega \subseteq \mathbb{R}^{n}$
  with $C^{1}$ boundary, and that $a,b,c$
  satisfy
  \begin{equation}\label{eq:assLcoef}
    b^{2}+|\nabla \cdot b|\in L^{2}_{loc},
    \quad
    c\in L^{\frac n2,1},
    \quad
    \|a-I \|_{L^{\infty}}+
    \||b|+|a'|\|_{L^{n,\infty}}+
    \|b'\|_{L^{\frac n2,\infty}}
    +\|c_{-}\|_{L^{\frac n2,1}}
    <\epsilon.
  \end{equation}
  If $\epsilon$ sufficiently small
  then for all $0\le \sigma\le 1$ we have
  \begin{equation}\label{eq:equivfr}
    \|(-L)^{\sigma}v\|_{L^{p}}
    \simeq
    \|(-\Delta)^{\sigma}v\|_{L^{p}},
    \qquad
    1<p<\frac{n}{2 \sigma}.
  \end{equation}
\end{proposition}

\begin{proof}
  The assumptions of the two previous Propositions are satisfied,
  thus $-L$ is selfadjoint, nonnegative, and the gaussian bound
  \eqref{eq:gaussestup} is valid. 

  Consider first the case $\sigma=1$. Write the operator
  $L$ in the form
  \begin{equation*}
    \textstyle
    Lv=
    \sum_{jk}a_{jk}\partial_{j}\partial_{k}v+
       \sum_{j}\beta_{j}\partial_{j}v
       +\gamma_{0}v
       -c_{+} v
  \end{equation*}
  where
  \begin{equation*}
    \textstyle
    \beta_{k}=\sum_{j} (\partial_{j}a_{jk}+ 2ia_{jk}b_{k}),
    \qquad
    \gamma_{0}=\sum_{j,k} i \partial_{j}(a_{jk}b_{k})
      -a(b,b)+c_{-}.
  \end{equation*}
  Then by H\"{o}lder and Sobolev inequalities
  in Lorentz spaces we have for $1<p<\frac n2$
  \begin{equation*}
    \|Lv\|_{L^{p}}
    \le
    \|a\|_{L^{\infty}}\|D^{2}v\|_{L^{p}}
    +
    \|\beta\|_{L^{n,\infty}}
    \|Dv\|_{L^{\frac{np}{n-p},p}}
    +\|\gamma_{0}-c_{+}\|_{L^{\frac n2,\infty}}
    \|v\|_{L^{\frac{np}{n-2p},p}}
    \lesssim
    \|\Delta v\|_{L^{p}}.
  \end{equation*}
  To prove the converse inequality, we first represent the
  operator $(-\Delta+c_{+})^{-1}$ in the form
  \begin{equation*}
    \textstyle
    (-\Delta+c_{+})^{-1}=
    c(n)\int_{0}^{+\infty}e^{t(\Delta-c_{+})}dt
  \end{equation*}
  and we apply the gaussian bound to obtain
  \begin{equation*}
    \textstyle
    |(-\Delta+c_{+})^{-1}|\lesssim
    \int_{0}^{+\infty}e^{-\frac{|x-y|^{2}}{Ct}}t^{-\frac n2}dt
    \lesssim|x-y|^{2-n}.
  \end{equation*}
  As a consequence, using the Hardy-Sobolev inequality we get
  \begin{equation*}
    \|(-\Delta+c_{+})^{-1}v\|_{L^{\frac{np}{n-2  p}}}
    \lesssim
    \|v\|_{L^{p}}
    \quad\text{i.e.}\quad 
    \|v\|_{L^{\frac{np}{n-2  p}}}
    \lesssim
    \|(-\Delta+c_{+})v\|_{L^{p}}
  \end{equation*}
  for all
  \begin{equation*}
    \textstyle
    1<p<\frac{n}{2}.
  \end{equation*}
  In particular this gives
  (since $\|c_{+}\|_{L^{\frac n2,\infty}}
  \lesssim \|c_{+}\|_{L^{\frac n2,1}}$)
  \begin{equation*}
    \textstyle
    \|\Delta v\|_{L^{p}}\le
    \|(\Delta-c_{+})v\|_{L^{p}}
    +\|c_{+}\|_{L^{\frac n2,\infty}}
    \|v\|_{L^{\frac{np}{n-2p}}}
    \lesssim
    \|(\Delta-c_{+})v\|_{L^{p}},
    \qquad
    1<p<\frac n2.
  \end{equation*}
  Adding and subtracting the remaining terms in $L$ 
  in the last term, we obtain
  \begin{equation*}
    \textstyle
    \|(\Delta-c_{+})v\|_{L^{p}}\le
    \|Lv\|_{L^{p}}+
    \|\sum(a_{jk}-\delta_{jk})\partial_{j}\partial_{k}v\|_{L^{p}}
    +
    \|\sum\beta_{k}\partial_{k}v\|_{L^{p}}
    +\|\gamma_{0}v\|_{L^{p}}
  \end{equation*}
  and a last application of H\"{o}lder and Sobolev 
  inequalities gives
  \begin{equation*}
    \|\Delta v\|_{L^{p}}
    \lesssim
    \|(\Delta-c_{+})v\|_{L^{p}}
    \lesssim
    \|Lv\|_{L^{p}}+\epsilon\|\Delta v\|_{L^{p}.}
  \end{equation*}
  If $\epsilon$ is sufficiently small we can subtract the
  last term from the left hand side, and the proof
  of the case $\sigma=1$ is concluded.
  The case $\sigma=0$ is trivial, and
  the remaining cases will be handled by
  Stein-Weiss complex interpolation.

  Indeed, consider the family of operators
  $T_{z}=(-L)^{z}(-\Delta)^{-z}$ for $0\le\Re z\le1$;
  our first goal is to prove that $T_{z}:L^{p}\to L^{p}$
  is bounded provided $1<p<n/(2\Re z)$,
  which implies the inequality $\lesssim$ in \eqref{eq:equivfr}.
  Note that the following arguments work with trivial modifications
  also for $-1\le\Re z\le0$ and give then the converse inequality
  $\gtrsim$. 

  $T_{z}$ is obviously an
  analytic family of operators, and $T^{iy}$  for real $y$
  is bounded on all
  $L^{p}$ with $1<p<\infty$, with a norm growing
  at most polynomially as $|y|\to \infty$.
  This property is well known
  for $(-\Delta)^{iy}$, while for $L^{iy}$
  it follows from the theory developed in
  \cite{DuongOuhabazSikora02-a} (see also
  \cite{CacciafestaDAncona12-a} for the case 
  $\Omega=\mathbb{R}^{n}$), which requires the sole
  assumption that $L$ satisfies a gaussian bound like
  \eqref{eq:gaussestup}. A standard application of the
  Stein-Weiss theorem then gives the claim.
\end{proof}

To conclude this section we construct a family of regularizing
operators which will be needed
later in the proof of $H^{1}$ well posedness;
what follows is an adaptation of 
Section 1.5 in \cite{Cazenave03-a}.
Assume that $\Omega$ and $L$ satisfy the assumptions
of the previous Proposition.
We define for $0<\epsilon\le1$ the operators
\begin{equation}\label{eq:defJ}
  J_{\epsilon}:=(I-\epsilon L)^{-1}
  \equiv\epsilon^{-1}R(-\epsilon^{-1})
\end{equation}
where $R(z)=(-L-z)^{-1}$ is the resolvent operator of $-L$.
Then for every 
$f\in H^{-1}(\Omega)$ the function 
$u=J_{\epsilon}f\in H^{1}_{0}(\Omega)$ is well defined as
the unique weak solution of the elliptic equation
\begin{equation*}
  -Lu+\epsilon^{-1}u=\epsilon^{-1}f.
\end{equation*}
Thus $J_{\epsilon}:H^{-1}(\Omega)\to H^{1}_{0}(\Omega)$
is a bounded operator, $L:H^{1}_{0}(\Omega)\to H^{-1}(\Omega)$
is bounded, we have the equivalence
$\|(I-L)v\|_{H^{-1}(\Omega)}\simeq\|v\|_{H^{1}_{0}(\Omega)}$
and the estimates
\begin{equation}\label{eq:estJ1}
  \|J_{\epsilon}v\|_{H^{1}_{0}(\Omega)}\le 
    C \epsilon^{-1}\|v\|_{H^{-1}(\Omega)},
  \qquad
  \|J_{\epsilon}v\|_{H^{2}(\Omega)}\le 
    C \epsilon^{-1}\|v\|_{L^{2}(\Omega)}
\end{equation}
by standard elliptic theory, with a $C$ independent of 
$\epsilon$. Further we have
\begin{equation}\label{eq:estJ2}
  \|J_{\epsilon}v\|_{H^{1}_{0}(\Omega)}\le 
    C \|v\|_{H^{1}_{0}(\Omega)},
  \quad
  \|J_{\epsilon}v\|_{L^{2}(\Omega)}\le 
    C  \|v\|_{L^{2}(\Omega)},
  \quad
  \|J_{\epsilon}v\|_{H^{-1}(\Omega)}\le 
    C \|v\|_{H^{-1}(\Omega)}
\end{equation}
and by complex interpolation
\begin{equation*}
  \|J_{\epsilon}v\|_{H^{1}_{0}(\Omega)}\le 
    C \epsilon^{-\frac12} \|v\|_{L^{2}(\Omega)},
  \qquad
  \|J_{\epsilon}v\|_{L^{2}(\Omega)}\le 
    C \epsilon^{-\frac12} \|v\|_{H^{-1}(\Omega)}.
\end{equation*}
Then, using the identity 
$J_{\epsilon}-I=
J_{\epsilon}(I-I+\epsilon L)=\epsilon J_{\epsilon}L$,
we deduce
\begin{equation}\label{eq:estJ3}
  \|(J_{\epsilon}-I)v\|_{H^{-1}(\Omega)}
  \le
  C \epsilon\|Lv\|_{H^{-1}(\Omega)}
  \le
  C' \epsilon\|v\|_{H^{1}_{0}(\Omega)}.
\end{equation}
Note that if $v\in H^{-1}(\Omega)$ only, we can
still approximate it with $\phi\in C^{\infty}_{c}(\Omega)$
to get
\begin{equation*}
  \|(J_{\epsilon}-I)v\|_{H^{-1}(\Omega)}
  \le
  C\|v-\phi\|_{H^{-1}(\Omega)}+
  C \epsilon\|\phi\|_{H^{1}_{0}(\Omega)}
\end{equation*}
and this implies
\begin{equation}\label{eq:estJ3bis}
  \forall v\in H^{-1}(\Omega)
  \qquad
  J_{\epsilon}v\to v
  \ \ \text{in $H^{-1}(\Omega)$ as $\epsilon\to0$.}
\end{equation}
We also obtain
\begin{equation}\label{eq:estJ4}
  \|(J_{\epsilon}-I)v\|_{L^{2}(\Omega)}
  \le
  C\|(J_{\epsilon}-I)v\|^{\frac12}_{H^{1}_{0}(\Omega)}  
  \|(J_{\epsilon}-I)v\|^{\frac12}_{H^{-1}(\Omega)}
  \le
  C'\epsilon^{\frac12}\|v\|_{H^{1}_{0}(\Omega)}
\end{equation}
and an argument similar to the previous one gives
\begin{equation}\label{eq:estJ3ter}
  \forall v\in L^{2}(\Omega)
  \qquad
  J_{\epsilon}v\to v
  \ \ \text{in $L^{2}(\Omega)$ as $\epsilon\to0$.}
\end{equation}
Finally, by the equivalence
$\|(J_{\epsilon}-I)v\|_{H^{1}_{0}(\Omega)}\simeq
\|(J_{\epsilon}-I)(I-L)v\|_{H^{-1}(\Omega)}$
we get
\begin{equation}\label{eq:estJ5}
  \forall v\in H^{1}_{0}(\Omega)
  \qquad
  J_{\epsilon}v\to v
  \ \ \text{in $H^{1}_{0}(\Omega)$ as $\epsilon\to0$.}
\end{equation}

Concerning the convergence in $L^{p}(\Omega)$ we have:

\begin{proposition}[]\label{pro:regulJLp}
  Let $p\in[1,\infty)$ and let $\Omega$ and $L$ satisfy
  the assumptions of Proposition \ref{pro:powereq}. 
  Then $J_{\epsilon}$ extends to
  a bounded operator on $L^{p}(\Omega)$ and the
  following estimate holds for $0<\epsilon\le1$
  \begin{equation}\label{eq:estJLp1}
    \|J_{\epsilon}v\|_{L^{p}(\Omega)}
    \le
    C\|v\|_{L^{p}(\Omega)}
  \end{equation}
  with a constant depending on $p$ but not of $\epsilon$.
  Moreover, for $1<p<\infty$ we have
  \begin{equation}\label{eq:estJLp2}
    \forall v\in L^{p}(\Omega)
    \qquad
    J_{\epsilon}v\to v
    \ \ \text{in $L^{p}(\Omega)$ as $\epsilon\to0$.}
  \end{equation}
\end{proposition}

\begin{proof}
  Let $\phi:(0,\infty)\to[0,\infty)$ be a smooth nondecreasing
  function with $\phi(s),s \phi'(s)$ bounded.
  Starting from the identity
  \begin{equation*}
    \textstyle
    \Re(-Lv \cdot\phi(|v|)\overline{v})
    +\nabla \cdot\{\Re(\overline{v}\phi(|v|)a \nabla^{b}v) \}
    =
    \phi(|v|)a(\nabla^{b}v,\nabla^{b}v)+
    \frac{\phi'(|v|)}{|v|}
    |\Re(\overline{v}\cdot a \nabla^{b}v)|^{2}
    +c \phi(|v|) |v|^{2},
  \end{equation*}
  and proceeding exactly as in the proof of 
  Proposition 1.5.1 in \cite{Cazenave03-a},
  we obtain \eqref{eq:estJLp1}. In order to prove
  \eqref{eq:estJLp2}, we can assume $v\in C^{\infty}_{c}(\Omega)$
  (as above). Then by the interpolation
  inequality in $L^{p}$ we can write for
  all $0<\theta<1$
  \begin{equation*}
    \|(J_{\epsilon}-I)v\|_{L^{\frac{2}{1-\theta}}}
    \le
    \|(J_{\epsilon}-I)v\|_{L^{1}}^{\theta}
    \|(J_{\epsilon}-I)v\|_{L^{2}}^{1-\theta}
    \le
    C\|v\|_{L^{1}}^{\theta}
    \cdot
    \|(J_{\epsilon}-I)v\|_{L^{2}}^{\theta}
  \end{equation*}
  where we used \eqref{eq:estJLp1}, and by \eqref{eq:estJ3ter}
  we conclude that $J_{\epsilon}v\to v$ in $L^{p}(\Omega)$
  for all $p=\frac{2}{1-\theta}\in(1,2)$. A similar argument
  gives the result for $p\in(2,\infty)$, and the case
  $p=2$ we already know.
\end{proof}

\section{Global existence and Scattering: proof of 
Theorem \ref{the:3}} \label{sec:scattering}

Throughout this section $\Omega \subseteq \mathbb{R}^{n}$ 
is an open set with $C^{1}$ boundary, $n\ge3$,
while $L$ is the unbounded operator
on $L^{2}(\Omega)$ 
with Dirichlet boundary conditions
under the assumptions of Proposition \ref{pro:selfadjoint}.
As explained in the Introduction, we shall work under the
black box Assumption (S) which ensures that the
necessary Strichartz estimates are available.
Notice that we are restricting the range
of admissible indices at the left hand side
for the derivative of the flow $\nabla e^{itL}$.

Our goal is to extend the usual local and global $H^{1}$ theory
to the NLS with variable coefficients
\begin{equation}\label{eq:maineq}
  iu_{t}-Lu+f(u)=0,
  \qquad
  u(0,x)=u_{0}(x).
\end{equation}
We shall sketch only the essential results which
will be needed in the proof of scattering, and not
aim at the greatest possible generality. In the following
we use the notations
\begin{equation*}
  L^{p}_{T}L^{q}=L^{p}(0,T;L^{q}(\Omega)),
  \qquad
  C_{T}H^{1}_{0}=C([0,T],H^{1}_{0}(\Omega)).
\end{equation*}

\begin{proposition}[Local existence in $H^{1}_{0}(\Omega)$]
\label{pro:localex}
  Let $n\ge3$ and assume (S) holds,
  while $f\in C^{1}(\mathbb{C},\mathbb{C})$ satisfies
  \begin{equation}\label{eq:assflocal}
    \textstyle
    |f(z)|\lesssim|z|^{\gamma},
    \quad
    |f(z)-f(w)|
    \lesssim(|z|+|w|)^{\gamma-1}|z-w|
    \ \ \text{for some}\
    1\le \gamma<1+\frac{4}{n-2}.
  \end{equation}
  Then for all $u_{0}\in H^{1}_{0}(\Omega)$ there exists
  $T=T(\|u_{0}\|_{H^{1}})$ and a unique solution
  $u\in C([0,T];H^{1}_{0}(\Omega))$. 
\end{proposition}

\begin{proof}
  The proof is standard; we sketch the main steps
  in order to check that
  the restriction $q_{1}<n$ imposed in (S)
  is harmless. We apply a fixed point argument to the map
  $\Phi: v \mapsto u$ defined as the solution of
  $iu_{t}-Lu+f(v)=0$, $u(0,x)=u_{0}$,
  working in a suitable bounded subset of the space
  $X_{T}=C([0,T];H^{1}_{0}(\Omega))\cap 
    L^{p}(0,T;W^{1,q}(\Omega))$
  for an appropriate choice of $(p,q)$,
  endowed with the distance
  $d(u,v)=\|u-v\|_{C_{T}L^{2}\cap L^{p}_{T}L^{q}}$;
  note that bounded subsets of $X_{T}$ are complete with
  this distance.

  In order to choose the indices
  we pick a real number $k$ such that
  \begin{equation}\label{eq:condk}
    n<2kn<n+2,
    \qquad
    \gamma(n-4)+2<2kn<\gamma(n-2)+2.
  \end{equation}
  Note that for all $n\ge3$ and all $1<\gamma<\frac{n+2}{n-2}$
  the two intervals in \eqref{eq:condk} have a nonempty
  intersection. Moreover, the couples $(p_{j},q_{j})$
  defined by
  \begin{equation*}
    \textstyle
    p_{1}=\frac{4 \gamma}{2+\gamma(n-2)-2kn},
    \quad
    q_{1}=\frac{\gamma n}{kn+\gamma-1},
    \qquad
    p_{2}=\frac{4}{2kn-n},
    \quad
    q_{2}=\frac{1}{1-k}
  \end{equation*}
  are admissible and we can use the estimates in (S),
  provided $q_{1}<n$ which will be checked at the end.
  We choose then $(p,q)=(p_{1},q_{1})$ in the definition
  of $X_{T}$.
  Applying Strichartz estimates on a time interval
  $[0,T]$ with $T$ to be chosen, we have for $u=\Phi(v)$
  \begin{equation*}
    \|\nabla u\|_{L^{p_{1}}_{T}L^{q_{1}}}
    +
    \|\nabla u\|_{L^{\infty}_{T}L^{2}}
    \lesssim
    \|u_{0}\|_{\dot H^{1}}+
    \|\nabla f(v)\|_{L^{p_{2}'}_{T}L^{q_{2}'}}.
  \end{equation*}
  By H\"{o}lder and Sobolev inequalities, using the
  assumptions on $f$, we have
  \begin{equation*}
    \|\nabla f(v)\|_{L^{p_{2}'}_{T}L^{q_{2}'}}
    \lesssim
    \norma*{\|v\|^{\gamma-1}_{L^{\frac{\gamma n}{kn-1}}}
    \|\nabla v\|_{L^{q_{1}}}}_{L^{p_{2}'}_{T}}
    \lesssim
    \|\nabla v\|^{\gamma}_{L^{\gamma p_{2}'}_{T}L^{q_{1}}}.
  \end{equation*}
  Now we note that the condition $\gamma<\frac{n+2}{n-2}$ is
  equivalent to $\gamma p_{2}'<p_{1}$, thus H\"{o}lder inequality
  on $[0,T]$ gives
  \begin{equation*}
    \|\nabla u\|_{L^{p_{1}}_{T}L^{q_{1}}}
    +
    \|\nabla u\|_{L^{\infty}_{T}L^{2}}
    \lesssim
    \|u_{0}\|_{\dot H^{1}}+
    T^{\frac{1}{p_{2}'}-\frac{\gamma}{p_{1}}}
    \|\nabla v\|^{\gamma}_{L^{p_{1}}_{T}L^{q_{1}}}
  \end{equation*}
  with a strictly positive power of $T$.
  An analogous computation gives
  \begin{equation*}
    \|u\|_{L^{p_{1}}_{T}L^{q_{1}}}
    +
    \|u\|_{L^{\infty}_{T}L^{2}}
    \lesssim
    \|u_{0}\|_{L^{2}}+
    T^{\frac{1}{p_{2}'}-\frac{\gamma}{p_{1}}}
    \|\nabla v\|^{\gamma-1}_{L^{p_{1}}_{T}L^{q_{1}}}.
    \|v\|_{L^{p_{1}}_{T}L^{q_{1}}}
  \end{equation*}
  and summing up we have proved
  \begin{equation*}
    \textstyle
    \|\Phi(v)\|_{X_{T}}\lesssim 
        \|u_{0}\|_{H^{1}}+T^{\sigma}
        \|v\|^{\gamma}_{X_{T}},
        \qquad
    \sigma=\frac{1}{p_{2}'}-\frac{\gamma}{p_{1}}>0.
  \end{equation*}
  Similar computations give
  \begin{equation*}
    d(\Phi(v_{1}),\Phi(v_{2}))\lesssim 
      T^{\sigma}
      (1+\|v_{1}\|_{X_{T}}+\|v_{2}\|_{X_{T}})^{\gamma-1}
      \|v_{1}-v_{2}\|_{L^{p_{1}}_{T}L^{q_{1}}}
  \end{equation*}
  and by a standard contraction argument on a suitable ball
  of $X_{T}$ we obtain the existence of a fixed point
  i.e. a solution of \eqref{eq:maineq} provided
  $T$ is smaller than a quantity $T(\|u_{0}\|_{H^{1}})$
  which depends only on the $H^{1}$ norm of the initial data.

  It remains to check the claim $q_{1}<n$. Since
  $2kn>n$ and $\gamma<\frac{n+2}{n-2}$ we have
  \begin{equation*}
    \textstyle
    q_{1}=
    \frac{2\gamma n}{2kn+2\gamma-2}<
    \frac{2\gamma n}{n+2\gamma-2}<
    \frac{2n(n+2)}{n^{2}-2n+8}
  \end{equation*}
  and the last fraction is $\le3$ for all integers $n\ge5$,
  while it is equal to $70/33<4$ for $n=4$ and
  to $30/11<3$ when $n=3$.

  To prove uniqueness, if $u,v$ are two solutions in
  $C_{T}H^{1}$ for some $T>0$, we can write
  \begin{equation*}
    \|u-v\|_{L^{p}_{T}L^{\gamma+1}}
    \lesssim
    \|f(u)-f(v)\|_{L^{p'}_{T}L^{(\gamma+1)'}}
    \lesssim
    \|u-v\|_{L^{b}_{T}L^{\gamma+1}}
    \||u|+|v|\|^{\gamma-1}_{L^{p_{0}}_{T}L^{\gamma+1}}
  \end{equation*}
  where
  \begin{equation*}
    \textstyle
    p=\frac4n \frac{\gamma+1}{\gamma-1},
    \qquad
    \frac{1}{p_{0}}=\frac1p-\frac12,
    \qquad
    \frac1b=\frac \gamma 2-\frac \gamma p+\frac12.
  \end{equation*}
  (note that we are not using Strichartz estimates of
  $\nabla u$), hence by Sobolev embedding
  \begin{equation*}
    \|u-v\|_{L^{p}_{T}L^{\gamma+1}}
    \lesssim
    (\|u\|_{L^{p_{0}}_{T}H^{1}}+\|v\|_{L^{p_{0}}_{T}H^{1}})
    ^{\gamma-1}
    \|u-v\|_{L^{b}_{T}L^{\gamma+1}}
  \end{equation*}
  It is easy to check that $b<p$, thus we get
  \begin{equation*}
    \lesssim
    T^{\epsilon}
    (\|u\|_{L^{\infty}_{T}H^{1}}+\|v\|_{L^{\infty}_{T}H^{1}})
    \|u-v\|_{L^{p}_{T}L^{\gamma+1}}
  \end{equation*}
  for some $\epsilon>0$ and this implies $u-v \equiv 0$
  if $T$ is small enough.
\end{proof}

Define the \emph{energy} of a solution 
$u\in C([0,T];H^{1}_{0}(\Omega))$ as
\begin{equation}\label{eq:ener}
  \textstyle
  E(t)=\frac12\int_{\Omega}a(\nabla^{b}u,\nabla^{b}u)dx
  +\frac12\int_{\Omega}c(x)|u|^{2}dx
  +\int_{\Omega}F(u)dx
\end{equation}

\begin{theorem}[Global existence in $H^{1}$]\label{the:global}
  Let $n\ge3$ and assume the coefficients of $L$ satisfy
  \begin{equation}\label{eq:assLenergy}
    b^{2}+|\nabla \cdot b|\in L^{2}_{loc},
    \quad
    c\in L^{\frac n2,1},
    \quad
    \|a-I \|_{L^{\infty}}+
    \||b|+|a'|\|_{L^{n,\infty}}+
    \|b'\|_{L^{\frac n2,\infty}}
    +\|c_{-}\|_{L^{\frac n2,1}}
    <\epsilon.
  \end{equation}
  Assume $f(u)$ satisfies the conditions \eqref{eq:assflocal}
  of the previous result, and in addition it is
  gauge invariant \eqref{eq:gauinv}
  with $F(r)=\int_{0}^{r}f(s)ds\ge0$ for $s\in \mathbb{R}$.
  Moreover, assume condition (S) holds. 

  Then, if $\epsilon$
  is sufficiently small, for all initial data
  $u_{0}\in H^{1}_{0}(\Omega)$ problem 
  \eqref{eq:maineq} has a unique global solution
  $u\in C\cap L^{\infty}(\mathbb{R};H^{1}_{0}(\Omega))$. 
  In addition the solution has constant energy $E(t)\equiv E(0)$ 
  for all $t\in \mathbb{R}$.
\end{theorem}

\begin{proof}
  Since the lifespan of the local solution only depends on
  the $H^{1}$ norm of the data, in order to prove the claim
  it is sufficient to prove that the energy $E(t)$ of
  the solution is conserved. Indeed,
  $E(t)$ controls the $H^{1}$ norm of $u$, and
  then global existence follows from a standard
  continuation argument.

  Let $e(u)$ be the energy density
  \begin{equation*}
    \textstyle
    e(u)(t,x)=\frac12 a(x)\nabla^{b}u \cdot \overline{\nabla^{b}u}
        +\frac12c(x)|u|^{2}+F(u)
  \end{equation*}
  so that $E(t)=\int_{\Omega}e(u)dx$.
  By gauge invariance and the definition of $F$ we have
  $\partial_{t}F(u)
    =\partial_{t}\int_{0}^{|u|}f(s)ds
    =\Re\left( f(|u|)\frac{u}{|u|}\bar u_{t}\right)=\Re(f(u)\bar u_{t})$.
  If the function $u$ satisfies $u(t)\in H^{2}(\Omega)$,
  we can write
  \begin{equation}\label{eq:idener}
    \partial_{t}e(u)
    +
    \nabla \cdot\{\Re \overline{u_{t}}a(x)\nabla^{b}u \}
    =
    \Re \overline{u_{t}}(i u_{t} -Lu+f(u))
    \equiv0
  \end{equation}
  and integrating on $\Omega$, since
  $u_{t}\vert_{\partial \Omega}=0$ by
  the Dirichlet boundary conditions, we obtain
  that $E(u)(t)\equiv E(u)(0)$ is constant in time.

  Since we know only $u(t)\in H^{1}_{0}(\Omega)$, in order to use
  \eqref{eq:idener} we need a regularization procedure;
  we use the operators $J_{\epsilon}$ constructed at the
  end of Section \ref{sec:gaussian_bds}.
  Thus we define
  $u_{\epsilon}=J_{\epsilon}u$ and note that $u_{\epsilon}$
  belongs to $C_{T}H^{2}(\Omega)$ and satisfies
  \begin{equation*}
    i \partial_{t}u_{\epsilon}-Lu_{\epsilon}+J_{\epsilon}f(u)=0.
  \end{equation*}
  Using \eqref{eq:idener} we obtain, after an
  integration on $[t_{1},t_{2}]\times \Omega$, with
  $0\le t_{1}<t_{2}\le T$,
  \begin{equation*}
    \textstyle
    \int_{\Omega}e(u_{\epsilon})\vert_{t_{1}}^{t_{2}}dx=
    \Re\int_{t_{1}}^{t_{2}}\int_{\Omega}
    \partial_{t}\overline{u_{\epsilon}}\cdot
    (f(u_{\epsilon})-J_{\epsilon}f(u))dxdt.
  \end{equation*}
  Substituting $\partial_{t}u_{\epsilon}$ from the equation
  and using the Cauchy-Schwartz inequality
  and the assumption $a_{jk}\in L^{\infty}$ we get
  \begin{equation}\label{eq:engoal}
    \textstyle
    \left|
    \int_{\Omega}e(u_{\epsilon})\vert_{t_{1}}^{t_{2}}dx
    \right|\lesssim
    \int_{t_{1}}^{t_{2}}
    [\phi_{\epsilon}(t)+\psi_{\epsilon}(t)+\chi_{\epsilon}(t)]
    dt
  \end{equation}
  where
  \begin{equation*}
    \textstyle
    \phi_{\epsilon}=
    \int_{\Omega}|\nabla^{b} u_{\epsilon}|\cdot
      |\nabla^{b}(f(u_{\epsilon})-J_{\epsilon}f(u))|dx,
    \qquad
    \psi_{\epsilon}(t)=
    \int_{\Omega} |J_{\epsilon}f(u)|\cdot
    |f(u_{\epsilon})-J_{\epsilon}f(u)|dx.
  \end{equation*}
  \begin{equation*}
    \textstyle
    \chi_{\epsilon}(t)=
    \int_{\Omega}|c| |u_{\epsilon}|\cdot
    |f(u_{\epsilon})-J_{\epsilon}f(u)|dx
  \end{equation*}
  Since $u_{\epsilon}\to u$ in $H^{1}_{0}$ and hence
  by Sobolev embedding in $L^{\gamma+1}$, we see that
  $E(u_{\epsilon})\to E(u)$.
  Thus to conclude the proof it is sufficient to show
  that the right hand side of \eqref{eq:engoal}
  tends to 0 as $\epsilon\to0$, possibly through
  a subsequence; to this end we shall apply dominated
  convergence on the interval $[0,T]$.

  Consider first the case $n\ge4$, so that
  $\gamma+1<n$.
  We prepare a few additional inequalities:
  \begin{equation*}
    \|\nabla u_{\epsilon}\|_{L^{\gamma+1}}
    \simeq
    \|(-L)^{\frac12}J_{\epsilon}u\|_{L^{\gamma+1}}
    =
    \|J_{\epsilon}(-L)^{\frac12}u\|_{L^{\gamma+1}} 
    \lesssim
    \|(-L)^{\frac12}u\|_{L^{\gamma+1}} 
    \simeq
    \|\nabla u\|_{L^{\gamma+1}}
  \end{equation*}
  by the $L^{p}$ boundedness of $J_{\epsilon}$ and 
  \eqref{eq:equivfr} for $\sigma=1/2$. By H\"{o}lder and
  Sobolev inequalities in Lorentz spaces,
  using $b\in L^{n,\infty}$, we have also
  \begin{equation*}
    \textstyle
    \|bu_{\epsilon}\|_{L^{\gamma+1}}
    \lesssim
    \|u_{\epsilon}\|_{L^{q,\gamma+1}}
    \lesssim
    \|\nabla u_{\epsilon}\|_{L^{\gamma+1}}
    \lesssim
    \|\nabla u\|_{L^{\gamma+1}},
    \qquad
    \frac{1}{\gamma+1}=\frac1n+\frac1q
  \end{equation*}
  and summing the two
  \begin{equation*}
    \|\nabla^{b} u_{\epsilon}\|_{L^{\gamma+1}}
    \lesssim
    \|\nabla u\|_{L^{\gamma+1}}.
  \end{equation*}
  Thus we have
  \begin{equation*}
    \phi_{\epsilon}(t)
    \lesssim
    \|\nabla u\|_{L^{\gamma+1}}
    \|\nabla(f(u_{\epsilon})-J_{\epsilon}f(u)))\|
      _{L^{\frac{\gamma+1}{\gamma}}}
    \lesssim
    \|\nabla u\|_{L^{\gamma+1}}^{2}
    \|u\|_{L^{\gamma+1}}^{\gamma-1}
    =: \phi(t).
  \end{equation*}
  Note that $\phi\in L^{1}(0,T)$ since
  \begin{equation*}
    \textstyle
    \int_{0}^{T} \phi dt\le 
    \|\nabla u\|_{L^{2}_{T} L^{\gamma+1}}^{2}
    \|u\|_{L^{\infty}L^{\gamma+1}}^{\gamma-1}
  \end{equation*}
  and $\nabla u\in L^{p}_{T} L^{\gamma+1}$ for some
  $p>2$ by Strichartz estimates, while
  $u\in C_{T}H^{1}_{0} \hookrightarrow L^{\infty}_{T}L^{\gamma+1}$
  by Sobolev embedding. For $\psi_{\epsilon}$ we have easily
  \begin{equation*}
    \textstyle
    \psi_{\epsilon}(t)\lesssim\|u\|_{L^{2 \gamma}}^{2 \gamma}
    =:\psi(t),
  \end{equation*}
  and by the interpolation and Sobolev inequalities
  \begin{equation*}
    \textstyle
    \|u\|_{L^{2 \gamma}}^{2 \gamma}
    \le
    \|u\|_{L^{\gamma+1}}^{2 \gamma-\sigma}
    \|u\|_{L^{\frac{n(\gamma+1)}{n-(\gamma+1)}}}^{\sigma}
    \lesssim
    \|u\|_{L^{\gamma+1}}^{2 \gamma-\sigma}
    \|\nabla u\|_{L^{\gamma+1}}^{\sigma},
    \qquad
    \sigma=\frac{\gamma-1}{\gamma+1}n
  \end{equation*}
  so that
  \begin{equation*}
    \textstyle
    \int_{0}^{T}\psi dt
    \lesssim
    \|u\|_{L^{\infty}_{T}L^{\gamma+1}}^{2 \gamma-\sigma}
    \|\nabla u\|_{L^{\sigma}_{T}L^{\gamma+1}}^{\sigma}
  \end{equation*}
  and again we obtain $\psi\in L^{1}(0,T)$ since $0<\sigma<2$
  for $1<\gamma<\frac{n+2}{n-2}$.
  As to $\chi_{\epsilon}$, recalling that
  $|c|^{\frac12}\in L^{n,\infty}$, we can write
  \begin{equation*}
    \|cu_{\epsilon}J_{\epsilon}f(u)\|_{L^{1}}
    \le
    \||c|^{\frac12}u_{\epsilon}\|_{L^{\gamma+1}}
    \||c|^{\frac12}J_{\epsilon}f(u)\|_{L^{\frac{\gamma+1}{\gamma}}}
    \lesssim
    \|\nabla u\|_{L^{\gamma+1}}
    \|\nabla J_{\epsilon}f(u)\|_{L^{\frac{\gamma+1}{\gamma}}}
    \lesssim \phi(t)
  \end{equation*}
  proceeding as in the estimate of $bu_{\epsilon}$;
  the term $cu_{\epsilon}f(u_{\epsilon})$ is similar.
  Thus the sequences 
  $\phi_{\epsilon},\psi_{\epsilon},\chi_{\epsilon}$
  are dominated. Moreover, it is easy to check,
  using exactly the previous estimates and properties
  \eqref{eq:estJ2}, \eqref{eq:estJ5}, \eqref{eq:estJLp1}
  and \eqref{eq:estJLp2}, that for a.e.~$t\in[0,T]$ one has
  $\phi_{\epsilon}(t), \psi_{\epsilon}(t),\chi_{\epsilon}(t)\to0$
  as $\epsilon\to0$.

  In the case $n=3$, the quantity $\gamma+1$ is in the range
  $2\le \gamma+1< 6$ and can be
  larger than $n$. The previous computations
  work fine for $1\le \gamma<2$; when $2\le \gamma<5$
  it is not difficult to modify the choice of indices so
  to use only the allowed Strichartz norms.
  For the estimate of $\phi_{\epsilon}(t)$ we can write
  for $\frac14<\epsilon< \frac12$
  \begin{equation*}
    \phi_{\epsilon}(t)
    \lesssim
    \|\nabla u\|_{L^{\frac{3}{1+\epsilon}}}^{2}
    \|u\|_{L^{\frac{3(\gamma-1)}{1-2 \epsilon}}}^{\gamma-1}
    \lesssim
    \|\nabla u\|_{L^{\frac{3}{1+\epsilon}}}^{2}
    \|\nabla u\|
      _{L^{\frac{3(\gamma-1)}{\gamma-2 \epsilon}}}^{\gamma-1}
    =: \phi(t)
  \end{equation*}
  by H\"{o}lder and Sobolev inequalitites, and hence
  \begin{equation*}
    \textstyle
    \int_{0}^{T}\phi(t)dt
    \le
    \|\nabla u\|_{L^{\frac{4}{1-2 \epsilon}}_{T}
        L^{\frac{3}{1+\epsilon}}}^{2}
    \|\nabla u\|
      _{L^{\frac{2(\gamma-1)}{1+2 \epsilon}}_{T}
        L^{\frac{3(\gamma-1)}{\gamma-2 \epsilon}}}^{\gamma-1}.
  \end{equation*}
  Notice that the first factor is an (allowed) Strichartz
  norm, while the second factor can be estimated 
  by H\"{o}lder inequality in time with the Strichartz norm
  \begin{equation*}
    \|\nabla u\|
    _{L^{\frac{4(\gamma-1)}{\gamma-3+4 \epsilon}}_{T}
      L^{\frac{3(\gamma-1)}{\gamma-2 \epsilon}}}^{\gamma-1},
  \end{equation*}
  (which is allowed and meaningful for 
  $\frac14<\epsilon<\frac12$) since
  the condition 
  $\frac{4(\gamma-1)}{\gamma-3+4 \epsilon}>
    \frac{2(\gamma-1)}{1+2 \epsilon}$
  is equivalent to $\gamma<5$.
  The reamining estimates can be modified in a similar way;
  we omit the details.
\end{proof}


The next Proposition is the crucial step in the proof of
scattering. We follow the simpler approach to scattering
developed in 
\cite{Visciglia09-a} and
\cite{CassanoTarulli14-a}.
We prefer this to the more technical method of
\cite{TaoVisanZhang07-a},
which could also be used here.

\begin{proposition}[]\label{pro:decayLp}
  Let $n \geq 3$, and consider Problem \eqref{eq:maineq}
  under the assumptions of Theorem \ref{the:2} if $n \ge 4$
  or of Theorem \ref{the:2strong} if $n=3$.
  Then any solution 
  $u \in C\cap L^{\infty}(\mathbb{R};H^{1}_{0}(\Omega))$ 
  satisfies
  \begin{equation} \label{eq:stimaDecay}
    \lim_{t\to \pm\infty} 
    \|u(t,\cdot)\|_{L^{r}}=0
    \quad\text{for all}\quad 
    2 < r < \frac{2n}{n-2}.
  \end{equation}
\end{proposition}

\begin{proof}
  We consider only the case $t\to+\infty$; the proof in
  the case $t\to-\infty$ is identical.
  It is enough to prove
  \eqref{eq:stimaDecay} for $r=2+\frac4n$, i.e.,
  \begin{equation}\label{eq:rscelto}
    \lim_{t\to +\infty} 
    \|u(t,\cdot)\|_{L^{2+\frac4n}}=0.
  \end{equation}
  Indeed, the $H^{1}$ norm of $u$ is bounded
  for $t\in \mathbb{R}$, so that by Sobolev inequality we have
  \begin{equation}\label{eq:bddenergy}
    \|u(t,\cdot)\|_{L^{\frac{2n}{n-2}}}+
    \|u(t,\cdot)\|_{L^{2}}
    \lesssim
    \|u(t,\cdot)\|_{H^{1}}+
    \|u(t,\cdot)\|_{L^{2}}
    \le C
  \end{equation}
  with $C$ independent of $t$, and interpolating with
  \eqref{eq:rscelto} we obtain the full claim 
  \eqref{eq:stimaDecay}.

  Assume by contradiction that there exist an $\epsilon_{0}>0$
  and a sequence of times $t_{k}\uparrow +\infty$ such that
  for all $k$
  \begin{equation} \label{eq:dalbasso}
    \|u(t_{k},\cdot)\|_{L^{2+\frac4n}}\ge \epsilon_{0}.
  \end{equation}
  Denote with $Q_{R}(x)$ the intersection with $\Omega$
  of the cube of side $R$ and center $x$
  (with sides parallel to the axes). By interpolation
  in $L^{p}$ spaces and Sobolev embedding, we have
  for all $v\in H^{1}_{0}(\Omega)$ and $x\in \Omega$
  \begin{equation*}
    \|v\|_{L^{2+\frac4n}(Q_{1}(x))}^{2+\frac4n}
    \le
    \|v\|_{L^{\frac{2n}{n-2}}(Q_{1}(x))}^{2}
    \cdot
    \|v\|_{L^{2}(Q_{1}(x))}^{\frac4n}
    \lesssim
    \|v\|_{H^{1}(Q_{1}(x))}^{2}
    \cdot
    \|v\|_{L^{2}(Q_{1}(x))}^{\frac4n}
  \end{equation*}
  which implies, for all $x\in \Omega$,
  \begin{equation*}
    \|v\|_{L^{2+\frac4n}(Q_{1}(x))}^{2+\frac4n}
    \lesssim
    \|v\|_{H^{1}(Q_{1}(x))}^{2}
    \cdot
    \sup_{y\in \Omega}
    \|v\|_{L^{2}(Q_{1}(y))}^{\frac4n}.
  \end{equation*}
  Choosing a sequence of centers $x\in \Omega$ such that
  the cubes $Q_{1}(x)$ cover $\Omega$ and are
  almost disjoint, and summing over all cubes,
  we obtain the inequality
  \begin{equation}\label{eq:GNir}
    \|v\|_{L^{2+\frac4n}(\Omega)}^{2+\frac4n}
    \lesssim
    \|v\|_{H^{1}(\Omega)}^{2}
    \cdot
    \sup_{x\in \Omega}
    \|v\|_{L^{2}(Q_{1}(x))}^{\frac4n}.
  \end{equation}
  Combining \eqref{eq:GNir} with the energy bound 
  \eqref{eq:bddenergy} and recalling \eqref{eq:dalbasso},
  we obtain that there exists a sequence of points
  $x_{k}\in \Omega$ such that
  \begin{equation*}
    \norma{u(t_k,\cdot)}_{L^2(Q_1(x_k))}\geq \epsilon_1>0.
  \end{equation*}
  We claim that we can find $\bar t >0 $ such that 
  \begin{equation}\label{eq:claim0}
    \|u(t,\cdot)\|_{L^2(Q_2(x_k))} 
    \ge \epsilon_1/2
    \quad \text{for all } 
    t \in (t_k,t_k+\bar t).
  \end{equation}
  Indeed, consider a cut-off function 
  $\chi\in C^\infty_c(\R^n)$ such that 
  $\chi(x)=1$ on the cube of side 1 with center $x_{k}$,
  and
  $\chi(x)=0$ outside the cube of side 2 with center $x_{k}$.
  We integrate the elementary identity
  \begin{equation*}
    \textstyle
    \frac{d}{dt} \left[ \chi(x)\abs{u(t,x)}^2\right] = 
    2 \chi(x) \nabla \cdot \{\Im[a(x) \nabla^b u(t,x) 
    \bar u(t,x)]\}
  \end{equation*}
  on $\Omega$
  and we obtain, for all $t\in\R$,
  \begin{equation}    \label{eq:Morawetz1}
    \begin{split}
      \textstyle
      \left\vert \frac{d}{dt} 
      \int_{\Omega} \chi(x) \abs{u(t,x)}^2\,dx \right\vert
      &\lesssim
      \textstyle
      \left\vert \int_{\Omega} 
      \nabla\chi(x)\cdot 
      \Im[a(x)\nabla^b u(t,x)\bar u(t,x)] \right\vert \\
      &\lesssim \|u(t,\cdot)\|_{L^2(\Omega)} 
      \|\nabla^b u(t,\cdot)\|_{L^2(\Omega)} \\
      &\leq \|u(0,\cdot)\|_{L^2(\Omega)} 
      \sup_{t \in \R}\|\nabla u(t,\cdot)\|_{L^2(\Omega)} 
      =: \overline C  < +\infty,
    \end{split}
  \end{equation}
  where we used \eqref{eq:equivmag}. This implies
  \begin{equation*}
    \textstyle
    \int_{Q_2(x_k)} \abs{u(t,x)}^2\,dx 
    \geq \int_{Q_1(x_k)}\abs{u(t_k,x)}^2\,dx -
     \overline C \abs{t-t_k},
  \end{equation*}
  whence \eqref{eq:claim0} follows provided that we choose 
  $\bar t>0$ such that  
  $\epsilon_1^2 - \overline C \bar t>\epsilon_1^2/4$.
  Note, by passing to a subsequence, we can assume the
  intervals $(t_{k},t_{k}+\overline{t})$ to be disjoint.
  
  If $n\ge 4$,  we get 
  \begin{equation*}
    \textstyle
    \int\,\int_{\Omega\times \Omega} 
    \frac{|u(t,x)|^2|u(t,y)|^2}{|x-y|^3}dxdydt
    \gtrsim
    \sum_{k}
    \int_{t_{k}}^{t_{k}+\overline{t}}
    \int_{Q_2(x_k)\times Q_2(x_k)}
      |u(t,x)|^2|u(t,y)|^2dxdydt=\infty.
  \end{equation*}
  but this is in contradiction
  with \eqref{eq:interactiontesi}, since 
  $u \in L^\infty(\R,H^1_0(\Omega))$, and this concludes the proof
  in this case. On the other hand, if $n=3$,
  from \eqref{eq:claim0} we get that
  \begin{equation*}
    \norma{u}_{L^4((t_k,t_k+\bar t)\times Q_2(x_k))}^4 \geq 
    C \epsilon_1^4 \bar t,
  \end{equation*}
  which is in contradiction with \eqref{eq:interactiontesistrong}.
\end{proof}

By fairly standard arguments, property
\eqref{eq:stimaDecay} implies that the
Strichartz norms of a global $H^{1}$ solutions are
bounded, and scattering follows. The only limitation here
is the requirement $q_{1}<n$ in Assumption (S), which is
effective only in dimension $n=3,4$. We sketch the
arguments for the sake of completeness:

\begin{proposition}[]\label{pro:bddstrich}
  Let $u \in C\cap L^{\infty}(\mathbb{R};H^{1}_{0}(\Omega))$ 
  be a solution to Problem \eqref{eq:maineq} 
  under the assumptions  of Theorem \ref{the:2} if $n \ge 4$ and 
  under the assumptions  of Theorem \ref{the:2strong} if $n=3$.
  Moreover, assume that (S) holds
  and that $\gamma>1+\frac4n$.
  Then for every admissible pair $(p,q)$ we have
  $u\in L^{p}L^{q}$, and for every admissible pair
  $(p,q)$ with $q<n$ we have $\nabla u\in L^{p}L^{q}$.
\end{proposition}

\begin{proof}
  We consider in detail the case $n\ge4$, where $\gamma+1<n$.
  For the case $n=3$ in the range $2\le \gamma<6$, the following
  arguments can be easily modified as in the
  last part of the proof in Theorem \ref{the:global}.
  Note that we know that the Strichartz norms are finite on
  bounded time intervals, and we only need to prove an
  uniform bound as the time interval invades $\mathbb{R}$.

  We use the notation 
  $L^{p}_{T,t}L^{q}:=L^{p}(T,t;L^{q}(\Omega))$
  for $t>T$.
  By Strichartz estimates on the time interval
  $[T,t]$ for the admissible couple
  $(p,\gamma+1)$ where $p=\frac4n \frac{\gamma+1}{\gamma-1}$
  we have
  \begin{equation*}  
  \begin{split}
    \|u\|_{L^{p}_{T,t}L^{\gamma+1}}  
    & \lesssim
    \|u(T)\|_{L^{2}}+\|f(u)\|_{L^{p'}_{T,t}L^{(\gamma+1)'}} 
  \\
    &  \lesssim
    \|u(T)\|_{L^{2}}+
    \|\|u\|_{L^{\gamma+1}}^{\gamma}\|_{L^{ p'}_{T,t}}
  \end{split}
  \end{equation*}
  since $|f(u)|\lesssim|u|^{\gamma}$
  and $(\gamma+1)'\gamma=\gamma+1$. The condition 
  $\gamma>1+\frac4n$ is equivalent to $\gamma>\frac{p}{p'}$,
  thus we can continue the estimate as follows:
  \begin{equation*}
  \begin{split}
    &\lesssim
    \|u(T)\|_{L^{2}}+
    \|u\|_{L^{\infty}_{T,t}L^{\gamma+1}}^{\gamma-\frac{p}{p'}}
    \|\|u\|_{L^{\gamma+1}}^{\frac{p}{p'}}\|_{L^{ p'}_{T,t}}
    \\
    & \le
    \|u(T)\|_{L^{2}}+
    \|u\|_{L^{\infty}_{T,\infty}L^{\gamma+1}}^{\gamma-\frac{p}{p'}}
    \|u\|_{L^{ p'}_{T,t}L^{\gamma+1}}^{\frac{p}{p'}}.
  \end{split}
  \end{equation*}
  By Proposition \ref{pro:decayLp} we know that
  $o(T)=\|u\|_{L^{\infty}_{T,\infty}L^{\gamma+1}}\to0$ 
  as $T\to \infty$.
  Thus the function 
  $\phi(t):=\|u\| _{L^{p}_{T,t}L^{\gamma+1}}$
  satisfies an inequality of the form
  $\phi(t)\le C+o(T)\phi(t)^{\frac{p}{p'}}$.
  Taking $T$ large enough, an easy continuity argument
  shows that $\phi(t)$ is bounded for all $t>T$.
  This proves that $u\in L^{p}L^{\gamma+1}$. Now
  we notice that in the previous computations 
  we have also proved that 
  $f(u)\in L^{p'}L^{(\gamma+1)'}$, and using again Strichartz
  estimates we conclude that $u\in L^{r}L^{q}$ for all
  admissible $(r,q)$.

  The estimate of $\nabla u$ is similar:
  \begin{equation*}  
  \begin{split}
    \|\nabla u\|_{L^{p}_{T,t}L^{\gamma+1}}  
    & \lesssim
    \|\nabla u(T)\|_{L^{2}}
    +\|\nabla f(u)\|_{L^{p'}_{T,t}L^{(\gamma+1)'}} 
  \\
    &  \lesssim
    \|\nabla u(T)\|_{L^{2}}+
    \|\|u\|_{L^{\gamma+1}}^{\gamma-1}
      \|\nabla u\|_{L^{\gamma+1}}\|_{L^{ p'}_{T,t}}
  \end{split}
  \end{equation*}
  since $|f'(u)|\lesssim|u|^{\gamma-1}$, and as before,
  using H\"{o}lder inequality,
  \begin{equation*}  
  \begin{split}
    & \lesssim
    \|\nabla u(T)\|_{L^{2}}+
    \|u\|_{L^{\infty}_{T,\infty}L^{\gamma+1}}^{\gamma-\frac{p}{p'}}
    \|\|u\|_{L^{\gamma+1}}^{\frac{p}{p'}-1}
      \|\nabla u\|_{L^{\gamma+1}}\|_{L^{ p'}_{T,t}}
  \\
    &  \lesssim
    \|\nabla u(T)\|_{L^{2}}+
    \|u\|_{L^{\infty}_{T,\infty}L^{\gamma+1}}^{\gamma-\frac{p}{p'}}
    \|u\|_{L^{p}_{T,t}L^{\gamma+1}}^{\frac{p}{p'}-1}
    \|\nabla u\|_{L^{p}_{T,t}L^{\gamma+1}}.
  \end{split}
  \end{equation*}
  By the bound already proved, this implies
  \begin{equation*}
    \|\nabla u\|_{L^{p}_{T,t}L^{\gamma+1}}
    \lesssim
    \|\nabla u(T)\|_{L^{2}}+
    o(T)\|\nabla u\|_{L^{p}_{T,t}L^{\gamma+1}}
  \end{equation*}
  and taking $T$ large enough we obtain the claim.
\end{proof}

We can now conclude the proof of Theorem \ref{the:3}. 
Part (i) is Theorem \ref{the:global}. Scattering is an immediate
consequence of the a priori bounds of the Strichartz norms
proved in Proposition \ref{pro:bddstrich}. We briefly sketch the
main steps of the proof which are completely standard,
in the case $t\to+\infty$; the case
$t\to-\infty$ is identical.

To construct the wave operator (claim (ii) of the
Theorem), given $u_{+}\in H^{1}_{0}(\Omega)$,
we consider the integral equation
\begin{equation}\label{eq:inteq}
  u(t):= e^{-it L }u_{+} + i \int_t^\infty e^{-i(t-s)}f(u(s))ds
\end{equation}
and we look for a solution defined on $[T,\infty)$, for $T$ 
sufficiently large. Using Strichartz estimates with the
same choice of indices as in the proof of local existence,
and noticing that the Strichartz norms of $e^{-itL}u_{+}$
are arbitrarily small for $T$ large,
by a fixed point approach we construct a solution 
$u\in C\cap L^{\infty}([T,+\infty),H^{1}_{0}(\Omega))$
to \eqref{eq:inteq}. This is also a solution to the
Schr\"{o}dinger equation in \eqref{eq:maineq}, and
thanks to the global existence result,
$u$ can be extended to a solution 
$u\in C\cap L^{\infty}(\mathbb{R},H^{1}_{0}(\Omega))$
defined for all $t\in \mathbb{R}$.
We can then choose $u_{0}=u(0)$. Uniqueness follows by
a similar argument: if two solutions $u_{1}$, $u_{2}$ 
of \eqref{eq:maineq}, with possibly different data, have the
same asymptotic behaviour i.e.~$\|u_{1}(t)-u_{2}(t)\|_{H^{1}}\to0$
as $t\to+\infty$, then they both solve \eqref{eq:inteq}, and
the previous fixed point argument implies $u_{1}(t)= u_{2}(t)$
for $t$ large. Then $u_{1}\equiv u_{2}$ by global uniqueness.

To prove asymptotic completeness (claim (iii) of the Theorem),
we fix a $u_{0}\in H^{1}_{0}(\Omega)$ and let $u(t)$ be the
corresponding global solution to Problem 
\eqref{eq:maineq}. 
Then we define $v(t)=e^{itL}u(t)$ and note that
\begin{equation*}
  v(t)=u_{0}+i\int_{0}^{t}e^{isL}f(u(s))ds.
\end{equation*}
Note that $\|e^{itL}\phi\|_{L2}=\|\phi\|_{L^{2}}$
by the unitarity of $e^{itL}$; moreover,
since $(-L \phi,\phi)_{L^{2}}\simeq\|\phi\|^{2}_{\dot H^{1}}$, 
we have
$\|e^{itL}\phi\|^{2}_{\dot H^{1}}\simeq
(-Le^{itL}\phi,e^{itL} \phi)_{L^{2}}\simeq\|\phi\|_{\dot H^{1}}$,
and in conclusion we get
\begin{equation*}
  \|e^{itL}\phi\|_{H^{1}}\simeq \|\phi\|_{H^{1}}
  \qquad
  \forall \phi\in H^{1}_{0}(\Omega)
\end{equation*}
with constants uniform in $t$.
Thus for $0< \tau < t$ we can write
\begin{equation*}
  \norma{v(t)-v(\tau)}_{H^1} 
  \simeq   \left\Vert e^{-itL} (v(t)-v(\tau))
    \right\Vert_{H^1}
  =
  \left\Vert \int_\tau^t e^{-i(t-s) L} f(u)\,ds 
  \right\Vert_{L^\infty_t H^1}
\end{equation*}
and by Strichartz estimates, 
H\"older inequality and interpolation, we get
\begin{equation*}
  \|v(t)-v(\tau)\|_{H^1} 
  \lesssim 
  \|f(u)\|_{L^{p'}_{\tau,t}W^{1,(\gamma+1)''}}
\end{equation*}
where $p=\frac4n \frac{\gamma+1}{\gamma-1}$;
this choice is always possible in dimension $n\ge4$;
in dimension $n=3$ for the range $2\le \gamma<6$
one needs to modify the choice as in the proof of
Theorem \ref{the:global}.
By Proposition \ref{pro:bddstrich} we know that the Strichartz
norms of $u$ are bounded, and by the same argument
used in that proof we see that $f(u)\in L^{p'}W^{1,(\gamma+1)'}$.
As a consequence, the right hand side of the previous
inequality can be made arbitrarily small provided
$t,\tau$ are large enough. We deduce that
$v(t)$ converges
in $H^{1}_{0}(\Omega)$ as $t\to+\infty$ to a limit $u_{+}$,
and finally
\begin{equation*}
  \|u(t)-e^{-itL}u_{+}\|_{H^{1}}
  \simeq
  \|v(t)-u_{+}\|_{H^{1}}\to0
\end{equation*}
as claimed.

\section{Strichartz estimates}\label{sec:strichartz}

Throughout this section, $\Omega=\mathbb{R}^{n}$ and
$L$ is the selfadjoint operator on $L^{2}(\mathbb{R}^{n})$
defined in Proposition \ref{pro:selfadjoint}.
We look for sufficient conditions on the coefficients
$a,b,c$ in order to have Strichartz estimates on 
$\mathbb{R}^{n}$ for the flow $e^{itL}$
\begin{equation}\label{eq:strest}
  \|e^{itL}u_{0}\|_{L^{p_{1}}L^{q_{1}}}
  \lesssim
  \|u_{0}\|_{L^{2}},
\end{equation}
\begin{equation}\label{eq:strestnh}
  \textstyle
  \|\int_{0}^{t}e^{i(t-s)L}Fds\|_{L^{p_{1}}L^{q_{1}}}
  \lesssim
  \|F\|_{L^{p_{2}'}L^{q_{2}'}}
\end{equation}
and for the derivative of the flow $\nabla e^{itL}$
\begin{equation}\label{eq:1strest}
  \|\nabla e^{itL}u_{0}\|_{L^{p_{1}}L^{q_{1}}}
  \lesssim
  \|\nabla u_{0}\|_{L^{2}},
\end{equation}
\begin{equation}\label{eq:1strestnh}
  \textstyle
  \|\nabla\int_{0}^{t}e^{i(t-s)L}Fds\|_{L^{p_{1}}L^{q_{1}}}
  \lesssim
  \|\nabla F\|_{L^{p_{2}'}L^{q_{2}'}}
\end{equation}
for admissible couples of indices $(p_{j},q_{j})$.
Recall that \emph{admissible couples} $(p,q)$ satisfy 
$p\in[2,\infty]$, $q\in[2,\frac{2n}{n-2}]$ with
$\frac2p+\frac nq=\frac n2$
and the \emph{endpoint} is the couple $(2,\frac{2n}{n-2})$. 

We shall derive the estimates of the first kind
by combining Tataru's results in
\cite{Tataru08-a} with our smoothing estimates.
On the other hand, in order to deduce \eqref{eq:1strest}, 
\eqref{eq:1strestnh} we shall use the equivalence of
Sobolev norms proved in Proposition \ref{pro:powereq}.
The following result is a direct application of
\cite{Tataru08-a}:

\begin{theorem}[] \label{the:strich1}
  Let $n\ge3$.
  Assume the coefficients $a,b,c$ of $L$ satisfy
  \begin{equation}\label{eq:suffassT}
    |a-I|+\bra{x}(|a'|+|b|)+\bra{x}^{2}(|a''|+|b'|+|c|)
    \le
    \epsilon \bra{x}^{-\delta}
  \end{equation}
  for some $\epsilon,\delta>0$. If $\epsilon$ is sufficiently
  small, the flow $e^{itL}$ satisfies the 
  Strichartz estimates \eqref{eq:strest},
  \eqref{eq:strestnh}
  for all admissible couples $(p_{j},q_{j})$, $j=1,2$,
  including the endpoint.
\end{theorem}

\begin{proof}
  We rewrite $L$ as the sum of 
  $Au=\nabla \cdot(a \nabla u)$ plus lower order terms
  \begin{equation*}
    Lu=Au+2ia(\nabla u,b)+i \partial_{j}(a_{jk}b_{k})u-a(b,b)u
    -c(x)u.
  \end{equation*}
  Define the norm
  \begin{equation*}
    \|v\|_{Z}=
    \|v\|_{L^{\infty}(|x|\le1)}
    +
    \sum_{j\ge1}\|v\|
    _{L^{\infty}(2^{j-1}\le|x|\le 2^{j})}.
  \end{equation*}
  By Theorem 4 and Remarks 6 and 7 in \cite{Tataru08-a},
  if $a,b,c$ satisfy
  \begin{equation}\label{eq:assTa}
    \|\bra{x}^{2}|a''(x)|\|_{Z}
    +
    \|\bra{x}|a'(x)|\|_{Z}
    +
    \||a(x)-I|\|_{Z}
    \le \epsilon,
  \end{equation}
  \begin{equation}\label{eq:assTb}
    \|\bra{x}^{2}\partial_{m}(a_{jk}b_{k})\|_{Z}
    +\|\bra{x}a_{jk}b_{k}\|_{Z}
    \le \epsilon,
  \end{equation}
  \begin{equation}\label{eq:assTc}
    \|\bra{x}^{2}[|\partial_{j}(a_{jk}b_{k})|
    +|a(b,b)|+|c(x)|]\|_{Z}\le \epsilon
  \end{equation}
  for $\epsilon$ small enough, then the linear flow
  $e^{itL}$ satisfies the full set of Strichartz
  estimates \eqref{eq:strest}, \eqref{eq:strestnh}. 
  It is immediate to
  check that condition \eqref{eq:suffassT}
  implies \eqref{eq:assTa}--\eqref{eq:assTc}.
\end{proof}

Combining the previous Theorem with our smoothing
estimate (Corollary \ref{cor:smooheat}) we cover the
case of repulsive electric potentials with a large positive
part:

\begin{theorem}[]\label{the:strich2}
  Let $n\ge3$.
  Assume the coefficients $a,b$ of $L$ satisfy
  \begin{equation}\label{eq:suffassT2}
    |a-I|+\bra{x}(|a'|+|b|)+\bra{x}^{2}(|a''|+|b'|)
    +\bra{x}^{3}|a'''|
    \le
    \epsilon \bra{x}^{-\delta}
  \end{equation}
  while the potential $c(x)$ satisfies
  \begin{equation}\label{eq:assc1}
    -\epsilon \bra{x}^{-2}\le c(x)\le C_{+}^{2}\bra{x}^{-2},
    \qquad
    \bra{x}^{1+\delta}c\in L^{n}
  \end{equation}
  and the repulsivity condition
  \begin{equation}\label{eq:assc2}
    a(x)x \cdot \nabla c(x)\le \epsilon|x|^{-1}\bra{x}^{-1-\delta}
  \end{equation}
  for some $\epsilon,\delta,C_{+}>0$.
  If $\epsilon$ is sufficiently
  small, the flow $e^{itL}$ satisfies the homogeneous
  Strichartz estimates \eqref{eq:strest}
  for all admissible couples, and
  the inhomogeneous estimates \eqref{eq:strestnh}
  for all couples with the exception of the
  endpoint-endpoint case.
\end{theorem}

\begin{proof}
  By Theorem \ref{the:strich1}, Strichartz estimates
  are valid for the flow $e^{itL_{0}}$ with $c=0$.
  The complete flow $u=e^{itL}u_{0}$
  satisfies the equation $iu_{t}+L_{0}u=cu$, hence it can
  be written
  \begin{equation*}
    \textstyle
    u=e^{itL}u_{0}=
    e^{itL_{0}}u_{0}-
    i\int_{0}^{t}e^{i(t-s)L_{0}}(cu)ds
  \end{equation*}
  so that, by the previous result,
  \begin{equation*}
    \|u\|_{L^{p}L^{q}}
    \lesssim
    \|u_{0}\|_{L^{2}}
    +
    \|cu\|_{L^{2}L^{\frac{2n}{n+2}}}
  \end{equation*}
  for all admissible couples $(p,q)$.
  By H\"{o}lder inequality we have
  \begin{equation*}
    \|cu\|_{L^{2}L^{\frac{2n}{n+2}}}
    \lesssim
    \|\bra{x}^{1+\delta}c\|_{L^{n}}
    \|\bra{x}^{-1-\delta}u\|_{L^{2}L^{2}}
  \end{equation*}
  and the homogeneous estimate 
  will be proved if we can prove the estimate
  \begin{equation}\label{eq:close}
    \|\bra{x}^{-1-\delta}u\|_{L^{2}L^{2}}
    \lesssim
    \|u_{0}\|_{L^{2}}.
  \end{equation}
  Indeed, the assumptions of Corollary \ref{cor:smooheat}
  are satisfied by $L$; in particular, the gaussian upper 
  bound for the heat flow $e^{itL}$ is valid for
  general $L^{\infty}$ coefficients
  (see Theorem 5.4 in \cite{Ouhabaz04-a} or \cite{Ouhabaz05-a}).
  Thus \eqref{eq:close} follows from inequality
  \eqref{eq:smoomorzero} and we obtain the full set
  of homogeneous Strichartz estimates for the flow
  $e^{itL}$.

  To prove inhomogeneous estimates it is sufficient
  to apply a standard
  $TT^{*}$ argument combined with the Christ-Kiselev lemma,
  and this gives \eqref{eq:strestnh}
  with the exception of the endpoint-endpoint case.
\end{proof}

We conclude the section by proving the estimates for
the flow $\nabla e^{itL}$, which are now a 
straightforward consequence of the previous results.
Note that the application of Proposition \ref{pro:powereq}
imposes an additional condition $q_{1}<n$,
which is restrictive only in
dimensions $n=3$ and $4$.

\begin{corollary}[]\label{cor:strichder}
  Let $n\ge3$. Estimates
  \eqref{eq:1strest}, \eqref{eq:1strestnh} hold
  for the flow $\nabla e^{itL}$, for all admissible couples
  $(p_{j},q_{j})$, $j=1,2$, provided $q_{1}<n$
  and the coefficients $a,b,c$ of $L$ satisfy
  either assumption \eqref{eq:suffassT},
  or assumptions
  \eqref{eq:suffassT2}, \eqref{eq:assc1}, \eqref{eq:assc2},
  provided $\epsilon$ is small enough.
\end{corollary}

\begin{proof}
  In both cases we see that the assumptions of
  Proposition \ref{pro:powereq} are satisfied. In particular,
  in the second case the smallness of the
  $L^{\frac n2,1}$ norm of $c_{-}$ follows from the fact that
  the $L^{n}$ norm of $\bra{x}^{1+\delta}c$ is arbitrarily
  small outside a sufficiently large ball, and inside
  the ball we have $|c_{-}|\le \epsilon$ by condition
  \eqref{eq:assc1}.

  Now in the first case the assumptions of
  Theorem \ref{the:strich1} are satisfied and we can write
  \begin{equation*}
  \begin{split}
    \|\nabla e^{itL}u_{0}\|_{L^{p_{1}}L^{q_{1}}}
    \simeq
    &
    \|(-L)^{\frac12} e^{itL}u_{0}\|_{L^{p_{1}}L^{q_{1}}}
    =
    \| e^{itL}(-L)^{\frac12}u_{0}\|_{L^{p_{1}}L^{q_{1}}}
    \\
    \lesssim
    &
    \|(-L)^{\frac12}u_{0}\|_{L^{2}}
    \simeq
    \|\nabla u_{0}\|_{L^{2}}
  \end{split}
  \end{equation*}
  by a repeated application of \eqref{eq:equivfr} for
  $\sigma=\frac12$. The proof of the remaining claims is
  identical.
\end{proof}

\bibliography{interaction-20141214.bib}
\bibliographystyle{plain}

\end{document}